\documentclass[a4paper,english]{article}

\usepackage{subcaption}
\usepackage{graphicx}
\usepackage{amssymb}
\usepackage{amsmath}
\usepackage{amsthm}
\usepackage{enumerate}
\usepackage{fullpage}

\usepackage{tikz-cd}
\usepackage{tikz, tikz-3dplot}
\usetikzlibrary{arrows}
\usetikzlibrary{decorations.markings, positioning}
\usetikzlibrary{calc}
\usepackage{xcolor}

\usepackage{algorithm}
\usepackage{algpseudocode}

\usepackage{placeins}
\usepackage{hyperref}
\usepackage[colorinlistoftodos]{todonotes}

\usepackage{subcaption}

\newtheorem{definition}{Definition}
\newtheorem{theorem}{Theorem}
\newtheorem{lemma}{Lemma}
\newtheorem{corollary}{Corollary}

\newcommand{\suchthat}{\ifnum\currentgrouptype=16 \;\middle|\;\else\mid\fi}

\newcommand{\A}{\mathbb{A}}
\newcommand{\Abar}{\vec{\mathbb{A}}}

\DeclareMathOperator{\Hom}{Hom}
\DeclareMathOperator{\End}{End}

\DeclareMathOperator{\rep}{rep}

\newcommand{\arr}[1]{\operatorname{arr}(#1)}

\newcommand{\smat}[1]{ \left[\begin{smallmatrix} #1 \end{smallmatrix}\right] }

\newcommand{\Id}{\mathrm{id}}

\newcommand{\CL}{C\!L}
\newcommand{\intv}{\mathbb{I}}
\newcommand{\reltoeq}{\trianglerighteq}

\usepackage{scalerel,xcolor}
\def\darrow{\mathrel{\ThisStyle{\ooalign{$\SavedStyle\rightarrow$\cr%
        \hfil\textcolor{white}{\rule{2\LMpt}{1\LMex}}\kern2\LMpt\hfil}}}}

\def\clap#1{\hbox to 0pt{\hss#1\hss}}

\makeatletter
\def\dimVec #1{\expandafter\dimVec@i#1,,,,,\@nil}
\def\dimVec@i #1,#2,#3,#4,#5,#6,#7\@nil{%
  \ifx$#1$ \edef\paramA{0}\else \edef\paramA{#1} \fi
  \ifx$#2$ \edef\paramB{0}\else \edef\paramB{#2} \fi
  \ifx$#3$ \edef\paramC{0}\else \edef\paramC{#3} \fi
  \ifx$#4$ \edef\paramD{0}\else \edef\paramD{#4} \fi
  \ifx$#5$ \edef\paramE{0}\else \edef\paramE{#5} \fi
  \ifx$#6$ \edef\paramF{0}\else \edef\paramF{#6} \fi
  {\begin{smallmatrix}
      \paramD & \paramE & \paramF\\
      \paramA & \paramB & \paramC
    \end{smallmatrix}}
}

\def\dimVecF #1{\expandafter\dimVecF@i#1,,,\@nil}
\def\dimVecF@i #1,#2,#3,#4,#5\@nil{%
  \ifx$#1$ \edef\paramA{0}\else \edef\paramA{#1} \fi
  \ifx$#2$ \edef\paramB{0}\else \edef\paramB{#2} \fi
  \ifx$#3$ \edef\paramC{0}\else \edef\paramC{#3} \fi
  \ifx$#4$ \edef\paramD{0}\else \edef\paramD{#4} \fi
  {\begin{smallmatrix}
      \paramC & \paramD \\
      \paramA & \paramB
    \end{smallmatrix}}
}

\def\Int #1{\expandafter\Int@i#1\@nil}
\def\Int@i #1,#2\@nil{{#1{:}#2}}

\def\IntC #1{\expandafter\IntC@i#1\@nil}
\def\IntC@i #1,#2\@nil{{#1{,}#2}}

\def\itoi #1{\expandafter\itoi@i#1\@nil}
\def\itoi@i #1,#2,#3,#4\@nil{_{\Int{#1,#2}}^{\Int{#3,#4}}}

\makeatother

\definecolor{Teal}{HTML}{008080}
\colorlet{LightTeal}{Teal!85!white}

\definecolor{DarkCharcoal}{HTML}{4D4944}
\colorlet{Charcoal}{DarkCharcoal!85!white}
\colorlet{LightCharcoal}{Charcoal!50!white}

\colorlet{DarkRed}{red!70!black}
\colorlet{DarkBlue}{blue!70!black}

\colorlet{DarkGreen}{green!70!black}
\colorlet{DarkYellow}{yellow!70!black}

\definecolor{stage1}{HTML}{FF7F00}
\definecolor{stage2}{HTML}{00CC00}
\definecolor{stage3}{HTML}{FF9999}
\definecolor{stage4}{HTML}{FF0000}
\definecolor{stage5}{HTML}{7F00FF}

\colorlet{stage1t}{stage1!30}
\colorlet{stage2t}{stage2!30}
\colorlet{stage3t}{stage3!30}
\colorlet{stage4t}{stage4!30}
\colorlet{stage5t}{stage5!30}

\newcommand\Atile[4]{
\coordinate (cbl)	at (0+#1,		0+#2,	0+#3);
\coordinate (cbr)	at (5+#1,		0+#2,	0+#3);
\coordinate (ctl)	at (0+#1,		5+#2,	0+#3);
\coordinate (ctr)	at (5+#1,		5+#2,	0+#3);
\coordinate (ebl)	at (1.5+#1,		0+#2,	0+#3);
\coordinate (ebr)	at (3.5+#1,		0+#2,	0+#3);
\coordinate (erb)	at (5+#1,		.7+#2,	0+#3);
\coordinate (ert) 	at (5+#1,		4.3+#2,	0+#3);
\coordinate (etl) 	at (1.5+#1,		5+#2,	0+#3);
\coordinate (etr) 	at (3.5+#1,		5+#2,	0+#3);
\coordinate (elb)	at (0+#1,		1.5+#2,	0+#3);
\coordinate (elt) 	at (0+#1,		3.5+#2,	0+#3);
\coordinate (mid) 	at (1.86+#1,		2.5+#2,	0+#3);

\ifthenelse{#4>0}{
\fill[stage1t] (ert) -- (ctr) -- (etr);
\fill[stage1t] (erb) -- (cbr) -- (ebr);
\fill[stage1t] (elt) -- (elb) -- (mid);
\fill[stage1t] (elt) -- (ctl) -- (etl);
\fill[stage1t] (elb) -- (cbl) -- (ebl);
}{}

\ifthenelse{#4>2}{
\fill[stage3t] (elt) -- (etl) -- (mid);
\fill[stage3t] (elb) -- (ebl) -- (mid);
\fill[stage3t] (etl) -- (etr) -- (mid);
\fill[stage3t] (ebl) -- (ebr) -- (mid);
}{}

 \ifthenelse{#4>4}{
 \fill[stage5t] (ert) -- (etr) -- (mid) -- (ebr) -- (erb);
 }{}

\ifthenelse{#4>0}{
\draw[thick,stage1] (erb) -- (cbr) -- (ebr) -- (erb);
\draw[thick,stage1] (ert) -- (ctr) -- (etr) -- (ert);
\draw[thick,stage1] (etr) -- (etl) -- (ctl) -- (elt) -- (elb) -- (cbl) -- (ebl) -- (ebr);
\draw[thick,stage1] (elt) -- (mid) -- (elb);
\draw[thick,stage1] (elt) -- (etl);
\draw[thick,stage1] (elb) -- (ebl);
}{}

\ifthenelse{#4>2}{
\draw[thick,stage3] (etl) -- (mid);
\draw[thick,stage3] (ebl) -- (mid);
\draw[thick,stage3] (ebr) -- (mid) -- (etr);
}{}

\ifthenelse{#4>3}{
\draw[thick,stage4] (erb) -- (ert);
}{}

\ifthenelse{#4>4}{
\draw[thick,stage5] (ert) -- (mid) -- (erb);
}{}

\fill (cbl) circle (2pt);
\fill (cbr) circle (2pt);
\fill (ctl) circle (2pt);
\fill (ctr) circle (2pt);
\fill (ebl) circle (2pt);
\fill (ebr) circle (2pt);
\fill (elb) circle (2pt);
\fill (elt) circle (2pt);
\fill (erb) circle (2pt);
\fill (ert) circle (2pt);
\fill (etl) circle (2pt);
\fill (etr) circle (2pt);
\fill (mid) circle (2pt);
}

\newcommand\Btile[4]{
\coordinate (cbl)	at (0+#1,		0+#2,	0+#3);
\coordinate (cbr)	at (5+#1,		0+#2,	0+#3);
\coordinate (ctl)	at (0+#1,		5+#2,	0+#3);
\coordinate (ctr)	at (5+#1,		5+#2,	0+#3);
\coordinate (ebl)	at (1.3+#1,	0+#2,	0+#3);
\coordinate (ebr)	at (3.7+#1,	0+#2,	0+#3);
\coordinate (erb)	at (5+#1,		1.5+#2,	0+#3);
\coordinate (ert) 	at (5+#1,		3.5+#2,	0+#3);
\coordinate (etl) 	at (1.3+#1,	5+#2,	0+#3);
\coordinate (etr) 	at (3.7+#1,	5+#2,	0+#3);
\coordinate (elb)	at (0+#1,		.7+#2,	0+#3);
\coordinate (elt) 	at (0+#1,		4.3+#2,	0+#3);
\coordinate (mid) 	at (3+#1,		2.5+#2,	0+#3);

\ifthenelse{#4>0}{
\fill[stage1t] (ert) -- (ctr) -- (etr);
\fill[stage1t] (erb) -- (cbr) -- (ebr);
\fill[stage1t] (elt) -- (ctl) -- (etl);
\fill[stage1t] (elb) -- (cbl) -- (ebl);
}{}

\ifthenelse{#4>1}{
\fill[stage2t] (ert) -- (mid) -- (erb);
}{}

\ifthenelse{#4>2}{
\fill[stage3t] (ert) -- (mid) -- (etr);
\fill[stage3t] (erb) -- (mid) -- (ebr);
\fill[stage3t] (etl) -- (etr) -- (mid);
\fill[stage3t] (ebl) -- (ebr) -- (mid);
}{}

\ifthenelse{#4>3}{
\fill[stage4t] (elt) -- (etl) -- (mid) -- (ebl) -- (elb) -- (mid);
\fill[stage4t] (elt) -- (elb) -- (mid);
}{}

\ifthenelse{#4>0}{
\draw[thick,stage1] (ert) -- (ctr) -- (etr) -- (ert);
\draw[thick,stage1] (erb) -- (cbr) -- (ebr) -- (erb);
\draw[thick,stage1] (elt) -- (ctl) -- (etl) -- (elt);
\draw[thick,stage1] (elb) -- (cbl) -- (ebl) -- (elb);
\draw[thick,stage1] (erb) -- (ert);
}{}

\ifthenelse{#4>1}{
\draw[thick,stage2] (etl) -- (etr);
\draw[thick,stage2] (ebl) -- (ebr);
\draw[thick,stage2] (erb) -- (mid) -- (ert);
}{}

\ifthenelse{#4>2}{
\draw[thick,stage3] (mid) -- (etr);
\draw[thick,stage3] (mid) -- (ebr);
\draw[thick,stage3] (etl) -- (mid) -- (ebl);
}{}

\ifthenelse{#4>3}{
\draw[thick,stage4] (elt) -- (elb);
\draw[thick,stage4] (elt) -- (mid) -- (elb);
}{}

\fill (cbl) circle (2pt);
\fill (cbr) circle (2pt);
\fill (ctl) circle (2pt);
\fill (ctr) circle (2pt);
\fill (ebl) circle (2pt);
\fill (ebr) circle (2pt);
\fill (elb) circle (2pt);
\fill (elt) circle (2pt);
\fill (erb) circle (2pt);
\fill (ert) circle (2pt);
\fill (etl) circle (2pt);
\fill (etr) circle (2pt);
\fill (mid) circle (2pt);
}

\newcommand\Ctile[4]{
\coordinate (cbl) at 	(0+#1,	5+#2,	0+#3);
\coordinate (cbr) at 	(0+#1,	0+#2,	0+#3);
\coordinate (ctl) at 	(0+#1,	5+#2,	5+#3);
\coordinate (ctr) at 	(0+#1,	0+#2,	5+#3);
\coordinate (erb) at 	(0+#1,	0+#2,	1.5+#3);
\coordinate (ert) at 	(0+#1,	0+#2,	3.5+#3);
\coordinate (etr) at 	(0+#1,  1.5+#2,	5+#3);
\coordinate (etl) at 	(0+#1,	3.5+#2,	5+#3);
\coordinate (elb) at 	(0+#1,	5+#2,	1.5+#3);
\coordinate (elt) at 	(0+#1,	5+#2,	3.5+#3);
\coordinate (ebr) at 	(0+#1,	.7+#2,	0+#3);
\coordinate (ebl) at 	(0+#1,	4.3+#2,	0+#3);
\coordinate (mid) at 	(0+#1,	2.5+#2,	2.9+#3);

\ifthenelse{#4>0}{
\fill[stage1t] (erb) -- (cbr) -- (ebr);
\fill[stage1t] (elb) -- (cbl) -- (ebl);
\fill[stage1t] (elt) -- (ctl) -- (etl);
\fill[stage1t] (etr) -- (ctr) -- (ert);
}{}

\ifthenelse{#4>1}{
\fill[stage2t] (etl) -- (mid) -- (etr);
}{}

\ifthenelse{#4>2}{
\fill[stage3t] (elb) -- (elt) -- (etl) --(mid);
\fill[stage3t] (etr) -- (ert) -- (erb) -- (mid);
}{}

\ifthenelse{#4>3}{
\fill[stage4t] (mid) -- (erb) -- (ebr);
\fill[stage4t] (ebl) -- (elb) -- (mid);
\fill[stage4t] (ebl) -- (mid) -- (ebr);
}{}

\ifthenelse{#4>0}{
\draw[thick,stage1] (erb) -- (cbr) -- (ebr) -- (erb);
\draw[thick,stage1] (elb) -- (cbl) -- (ebl) -- (elb);
\draw[thick,stage1] (ert) -- (ctr) -- (etr) -- (ert);
\draw[thick,stage1] (elt) -- (ctl) -- (etl) -- (elt);
\draw[thick,stage1] (etl) -- (etr);
\draw[thick,stage1] (elb) -- (elt);
\draw[thick,stage1] (erb) -- (ert);
}{}

\ifthenelse{#4>1}{
\draw[thick,stage2] (etl) -- (mid) -- (etr);
}{}

\ifthenelse{#4>2}{
\draw[thick,stage3] (elb) -- (mid) -- (elt);
\draw[thick,stage3] (erb) -- (mid) -- (ert);
}{}

\ifthenelse{#4>3}{
\draw[thick,stage4] (ebl) -- (mid) -- (ebr);
\draw[thick,stage4] (ebl) -- (ebr);
}{}

\fill (cbl) circle (2pt);
\fill (cbr) circle (2pt);
\fill (ctl) circle (2pt);
\fill (ctr) circle (2pt);
\fill (ebl) circle (2pt);
\fill (ebr) circle (2pt);
\fill (elb) circle (2pt);
\fill (elt) circle (2pt);
\fill (erb) circle (2pt);
\fill (ert) circle (2pt);
\fill (etl) circle (2pt);
\fill (etr) circle (2pt);
\fill (mid) circle (2pt);
}

\newcommand\Dtile[4]{
\coordinate (cbl)	at (0+#1,		0+#2,	0+#3);
\coordinate (cbr)	at (5+#1,		0+#2,	0+#3);
\coordinate (ctl)	at (0+#1,		5+#2,	0+#3);
\coordinate (ctr)	at (5+#1,		5+#2,	0+#3);
\coordinate (ebl)	at (1.3+#1,	0+#2,	0+#3);
\coordinate (ebr)	at (3.7+#1,	0+#2,	0+#3);
\coordinate (erb)	at (5+#1,		.5+#2,	0+#3);
\coordinate (ert) 	at (5+#1,		4.5+#2,	0+#3);
\coordinate (etl) 	at (1.3+#1,	5+#2,	0+#3);
\coordinate (etr) 	at (3.7+#1,	5+#2,	0+#3);
\coordinate (elb)	at (0+#1,		1.3+#2,0+#3);
\coordinate (elt) 	at (0+#1,		3.7+#2,0+#3);
\coordinate (mid) 	at (2+#1,		2.5+#2,	0+#3);

\ifthenelse{#4>0}{
\fill[stage1t] (etr) -- (ctr) -- (ert);
\fill[stage1t] (ebr) -- (cbr) -- (erb);
\fill[stage1t] (ebl) -- (cbl) -- (elb);
\fill[stage1t] (etl) -- (ctl) -- (elt);
}{}

\ifthenelse{#4>1}{
\fill[stage2t] (mid) -- (elb) -- (elt);
}{}

\ifthenelse{#4>2}{
\fill[stage3t] (etl) -- (mid) -- (elt);
\fill[stage3t] (ebl) -- (mid) -- (elb);
\fill[stage3t] (mid)-- (etr) -- (etl);
\fill[stage3t] (mid) -- (ebl) -- (ebr);
}{}

\ifthenelse{#4>3}{
}{}

\ifthenelse{#4>4}{
\fill[stage5t] (mid) -- (ebr) -- (erb);
\fill[stage5t] (mid) -- (ert) -- (etr);
\fill[stage5t] (erb) -- (mid) -- (ert);
}{}

\ifthenelse{#4>0}{
\draw[thick,stage1] (ebl) -- (cbl) -- (elb) -- (ebl);
\draw[thick,stage1] (cbr) -- (ebr) -- (erb) -- (cbr);
\draw[thick,stage1] (elt) -- (ctl) -- (etl) -- (elt);
\draw[thick,stage1] (ctr) -- (etr) -- (ert) -- (ctr);
}{}

\ifthenelse{#4>1}{
\draw[thick,stage2] (ebl) -- (ebr);
\draw[thick,stage2] (etl) -- (etr);
\draw[thick,stage2] (elb) -- (elt);
\draw[thick,stage2] (elb) -- (mid) -- (elt);
}{}

\ifthenelse{#4>2}{
\draw[thick,stage3] (etl) -- (mid) -- (ebl);
\draw[thick,stage3] (ebr) -- (mid) -- (etr);
}{}

\ifthenelse{#4>3}{
}{}

\ifthenelse{#4>4}{
\draw[thick,stage5] (erb) -- (mid) -- (ert);
\draw[thick,stage5] (erb) -- (ert);
}{}

\fill (cbl) circle (2pt);
\fill (cbr) circle (2pt);
\fill (ctl) circle (2pt);
\fill (ctr) circle (2pt);
\fill (ebl) circle (2pt);
\fill (ebr) circle (2pt);
\fill (elb) circle (2pt);
\fill (elt) circle (2pt);
\fill (erb) circle (2pt);
\fill (ert) circle (2pt);
\fill (etl) circle (2pt);
\fill (etr) circle (2pt);
\fill (mid) circle (2pt);
}

\newcommand\Etile[4]{
\coordinate (cbl)	at (0+#1,		0+#2,	0+#3);
\coordinate (cbr)	at (5+#1,		0+#2,	0+#3);
\coordinate (ctl)	at (0+#1,		5+#2,	0+#3);
\coordinate (ctr)	at (5+#1,		5+#2,	0+#3);
\coordinate (ebl)	at (1.3+#1,	0+#2,	0+#3);
\coordinate (ebr)	at (3.7+#1,	0+#2,	0+#3);
\coordinate (erb)	at (5+#1,		1+#2,	0+#3);
\coordinate (ert) 	at (5+#1,		4+#2,	0+#3);
\coordinate (etl) 	at (1.3+#1,	5+#2,	0+#3);
\coordinate (etr) 	at (3.7+#1,	5+#2,	0+#3);
\coordinate (elb)	at (0+#1,		.5+#2,0+#3);
\coordinate (elt) 	at (0+#1,		4.5+#2,0+#3);
\coordinate (mid) 	at (3+#1,		2.5+#2,	0+#3);

\ifthenelse{#4>0}{
\fill[stage1t] (ebr) -- (cbr) -- (erb);
\fill[stage1t] (etr) -- (ctr) -- (ert);
\fill[stage1t] (ebl) -- (cbl) -- (elb);
\fill[stage1t] (etl) -- (ctl) -- (elt);
}{}

\ifthenelse{#4>2}{
\fill[stage3t] (mid) -- (etl) -- (etr) -- (ert) -- (erb) -- (ebr) -- (ebl);
}{}

\ifthenelse{#4>4}{
\fill[stage5t] (etl) -- (mid) -- (elt);
\fill[stage5t] (ebl) -- (mid) -- (elb);
\fill[stage5t] (elt) -- (mid) -- (elb);
}{}

\ifthenelse{#4>0}{
\draw[thick,stage1] (ebl) -- (cbl) -- (elb) -- (ebl);
\draw[thick,stage1] (cbr) -- (ebr) -- (erb) -- (cbr);
\draw[thick,stage1] (elt) -- (ctl) -- (etl) -- (elt);
\draw[thick,stage1] (ctr) -- (etr) -- (ert) -- (ctr);
}{}

\ifthenelse{#4>1}{
\draw[thick,stage2] (etl) -- (etr);
\draw[thick,stage2] (ebl) -- (ebr);
}{}

\ifthenelse{#4>2}{
\draw[thick,stage3] (ert) -- (mid) -- (etr);
\draw[thick,stage3] (ebl) -- (mid) -- (etl);
\draw[thick,stage3] (erb) -- (mid) -- (ebr);
\draw[thick,stage3] (ert) -- (erb);
}{}

\ifthenelse{#4>4}{
\draw[thick,stage5] (elb) -- (mid) -- (elt);
\draw[thick,stage5] (elb) -- (elt);
}{}

\fill (cbl) circle (2pt);
\fill (cbr) circle (2pt);
\fill (ctl) circle (2pt);
\fill (ctr) circle (2pt);
\fill (ebl) circle (2pt);
\fill (ebr) circle (2pt);
\fill (etl) circle (2pt);
\fill (etr) circle (2pt);
\fill (ert) circle (2pt);
\fill (erb) circle (2pt);
\fill (elt) circle (2pt);
\fill (elb) circle (2pt);
\fill (mid) circle (2pt);
}

\newcommand\EtileR[4]{
\coordinate (cbl)	at (0+#1,		0+#2,	0+#3);
\coordinate (cbr)	at (5+#1,		0+#2,	0+#3);
\coordinate (ctl)	at (0+#1,		5+#2,	0+#3);
\coordinate (ctr)	at (5+#1,		5+#2,	0+#3);
\coordinate (ebl)	at (1.3+#1,	0+#2,	0+#3);
\coordinate (ebr)	at (3.7+#1,	0+#2,	0+#3);
\coordinate (erb)	at (5+#1,		.5+#2,	0+#3);
\coordinate (ert) 	at (5+#1,		4.5+#2,	0+#3);
\coordinate (etl) 	at (1.3+#1,	5+#2,	0+#3);
\coordinate (etr) 	at (3.7+#1,	5+#2,	0+#3);
\coordinate (elb)	at (0+#1,		.7+#2,0+#3);
\coordinate (elt) 	at (0+#1,		4.3+#2,0+#3);
\coordinate (mid) 	at (2+#1,		2.5+#2,	0+#3);

\ifthenelse{#4>0}{
\fill[stage1t] (ebr) -- (cbr) -- (erb);
\fill[stage1t] (etr) -- (ctr) -- (ert);
\fill[stage1t] (ebl) -- (cbl) -- (elb);
\fill[stage1t] (etl) -- (ctl) -- (elt);
}{}

\ifthenelse{#4>2}{
\fill[stage3t] (mid) -- (etr) -- (etl) -- (elt) -- (elb) -- (ebl) -- (ebr);
}{}

\ifthenelse{#4>4}{
\fill[stage5t] (etr) -- (mid) -- (ert);
\fill[stage5t] (ebr) -- (mid) -- (erb);
\fill[stage5t] (ert) -- (mid) -- (erb);
}{}

\ifthenelse{#4>0}{
\draw[thick,stage1] (ebl) -- (cbl) -- (elb) -- (ebl);
\draw[thick,stage1] (cbr) -- (ebr) -- (erb) -- (cbr);
\draw[thick,stage1] (elt) -- (ctl) -- (etl) -- (elt);
\draw[thick,stage1] (ctr) -- (etr) -- (ert) -- (ctr);
}{}

\ifthenelse{#4>1}{
\draw[thick,stage2] (etl) -- (etr);
\draw[thick,stage2] (ebl) -- (ebr);
}{}

\ifthenelse{#4>2}{
\draw[thick,stage3] (ebr) -- (mid) -- (etr);
\draw[thick,stage3] (elt) -- (mid) -- (etl);
\draw[thick,stage3] (elb) -- (mid) -- (ebl);
\draw[thick,stage3] (elb) -- (elt);
}{}

\ifthenelse{#4>4}{
\draw[thick,stage5] (ert) -- (mid) -- (erb);
\draw[thick,stage5] (ert) -- (erb);
}{}

\fill (cbl) circle (2pt);
\fill (cbr) circle (2pt);
\fill (ctl) circle (2pt);
\fill (ctr) circle (2pt);
\fill (ebl) circle (2pt);
\fill (ebr) circle (2pt);
\fill (etl) circle (2pt);
\fill (etr) circle (2pt);
\fill (ert) circle (2pt);
\fill (erb) circle (2pt);
\fill (elt) circle (2pt);
\fill (elb) circle (2pt);
\fill (mid) circle (2pt);
}

\newcommand\Ftile[4]{
\coordinate (cbl) at 	(0+#1,	5+#2,	0+#3);
\coordinate (cbr) at 	(0+#1,	0+#2,	0+#3);
\coordinate (ctl) at 	(0+#1,	5+#2,	5+#3);
\coordinate (ctr) at 	(0+#1,	0+#2,	5+#3);
\coordinate (erb) at 	(0+#1,	0+#2,	1.3+#3);
\coordinate (ert) at 	(0+#1,	0+#2,	3.7+#3);
\coordinate (etr) at 	(0+#1,  .7+#2,	5+#3);
\coordinate (etl) at 	(0+#1,	4.3+#2,	5+#3);
\coordinate (elb) at 	(0+#1,	5+#2,	1.3+#3);
\coordinate (elt) at 	(0+#1,	5+#2,	3.7+#3);
\coordinate (ebr) at 	(0+#1,	.5+#2,	0+#3);
\coordinate (ebl) at 	(0+#1,	4.5+#2,	0+#3);
\coordinate (mid) at 	(0+#1,	2.5+#2,	3.5+#3);

\ifthenelse{#4>0}{
\fill[stage1t] (ebr) -- (cbr) -- (erb);
\fill[stage1t] (etr) -- (ctr) -- (ert);
\fill[stage1t] (ebl) -- (cbl) -- (elb);
\fill[stage1t] (etl) -- (ctl) -- (elt);
}{}

\ifthenelse{#4>2}{
\fill[stage3t] (elt) -- (etl) -- (mid);
\fill[stage3t] (ert) -- (etr) -- (mid);
}{}

\ifthenelse{#4>3}{
\fill[stage4t] (etl) -- (etr) -- (mid);
\fill[stage4t] (ert) -- (erb) -- (mid);
\fill[stage4t] (elt) -- (elb) -- (mid);
}{}

\ifthenelse{#4>0}{
\draw[thick,stage1] (ctl) -- (etl) -- (elt) -- (ctl);
\draw[thick,stage1] (cbl) -- (ebl) -- (elb) -- (cbl);
\draw[thick,stage1] (ebr) -- (cbr) -- (erb) -- (ebr);
\draw[thick,stage1] (etr) -- (ctr) -- (ert) -- (etr);
}{}

\ifthenelse{#4>1}{
\draw[thick,stage2] (etr) -- (mid) -- (etl);
\draw[thick,stage2] (elt) -- (elb);
\draw[thick,stage2] (ert) -- (erb);
}{}

\ifthenelse{#4>2}{
\draw[thick,stage3] (elt) -- (mid) -- (ert);
}{}

\ifthenelse{#4>3}{
\draw[thick,stage4] (mid) -- (elb);
\draw[thick,stage4] (mid) -- (erb);
\draw[thick,stage4] (etl) -- (etr);
}{}

\ifthenelse{#4>4}{
\draw[thick,stage5] (ebl) -- (ebr);
}{}

\fill (cbl) circle (2pt);
\fill (cbr) circle (2pt);
\fill (ctl) circle (2pt);
\fill (ctr) circle (2pt);
\fill (ebl) circle (2pt);
\fill (ebr) circle (2pt);
\fill (etl) circle (2pt);
\fill (etr) circle (2pt);
\fill (ert) circle (2pt);
\fill (erb) circle (2pt);
\fill (elt) circle (2pt);
\fill (elb) circle (2pt);
\fill (mid) circle (2pt);
}

\newcommand\sandallow[5]{
\Dtile{#2}{#3}{#4}{#5}
\Etile{5+#2}{#3}{#4}{#5}
\Ftile{5+#2}{#3}{#4}{#5}
\EtileR{10+#2}{#3}{#4}{#5}
\Etile{15+#2}{#3}{#4}{#5}
\Ftile{15+#2}{#3}{#4}{#5}
}

\newcommand\sandalup[5]{
\Atile{#2}{#3}{#4}{#5}
\Btile{5+#2}{#3}{#4}{#5}
\Ctile{5+#2}{#3}{#4}{#5}
\Atile{10+#2}{#3}{#4}{#5}
\Btile{15+#2}{#3}{#4}{#5}
\Ctile{15+#2}{#3}{#4}{#5}
}

\begin{document}

\title{Realizations of Indecomposable Persistence Modules of Arbitrarily Large Dimension}
\author{Micka\"{e}l Buchet\\TU Graz \and Emerson G. Escolar\\RIKEN AIP}



%
%
%

%

\maketitle

\begin{abstract}
  While persistent homology has taken strides towards becoming a widespread tool for data analysis, multidimensional persistence has proven more difficult to apply. One reason is the serious drawback of no longer having a concise and complete descriptor analogous to the persistence diagrams of the former. We propose a simple algebraic construction to illustrate the existence of infinite families of indecomposable persistence modules over regular grids of sufficient size.
  On top of providing a constructive proof of representation infinite type, we also provide realizations by topological spaces and Vietoris-Rips filtrations, showing that they can actually appear in real data and are not the product of degeneracies.
\end{abstract}


\section{Introduction}

Recently, persistent homology~\cite{topo_pers} has been established as a flagship tool of topological data analysis. It provides the persistence diagrams, an easy to compute and understand compact summary of topological features in various scales in a filtration.
Fields where it has been successfully applied include materials science~\cite{glass,leeDlotko}, neuroscience~\cite{cliqueneuro,kanari}, genetics~\cite{genetic} and medicine~\cite{cancer, cancer2}.

Over a filtration, persistence diagrams can be used because the persistence modules can be uniquely decomposed into indecomposable modules which are intervals.
To these intervals, we
associate the lifespans of topological features. However, when considering
persistence modules over more general underlying structures,
indecomposables are no longer intervals and can be more complicated.

As an example, in multidimensional persistence~\cite{multipers} over
the commutative grid, the representation category is no longer
representation finite. In other words, the number of possible
indecomposables is infinite for a large enough finite commutative
grid.
The minimal size for a commutative grid to be representation infinite is relatively small.
For two dimensional grids, we need a size of at least $2\times 5$ or $3\times 3$, as the $2\times n$ grid is representation finite for $n\leq 4$ \cite{pmcl}.
When considering three dimensional grids, it is enough to have a $2\times 2\times 2$ grid.
For a particular finite dimensional persistence module, it is
true that it can be uniquely decomposed into a direct sum of
indecomposables, but we cannot list all the possible
indecomposables a priori.

From another point of view, recent progress in software~\cite{lesnick}
has made practical application of multidimensional persistence more
accessible. However, they~\cite{lesnick} approach this by computing incomplete
invariants instead of indecomposable decompositions. This has the
advantage of being easier to compute than the full decomposition and
easier to visualize.

In this work, we aim to provide more intuition for the structure of some indecomposable modules.
First, we provide algebraic constructions of some infinite families of indecomposable modules over all representation infinite grids.
Next, we tackle the problem of realizing these constructions.
Any module over a commutative grid can realized as the $k$-th persistent homology module of a simplicial complex for any $k>0$. This result was first claimed in~\cite{multipers} and proved in~\cite{otter} with a construction which is straightforward algebraically but difficult to visualize.
Here, we realize our infinite families of indecomposable modules using a more visual construction. In all our algebraic constructions, our infinite families depend on a dimension parameter $d$ and a parameter $\lambda \in K$, which we set to $\lambda = 0$ for our topological constructions.
Our topological constructions are formed by putting together $d$ copies of a repeating simple pattern.
Finally we provide  a Vietoris-Rips construction for our realization of the $2\times 5$ case.
Moreover, we show that this construction is stable with respect to small perturbations.

A direct corollary of our result is a constructive proof of the representation infinite type of the grids we consider.
It also provides insight into one possible topological origin of these algebraic complications.
Given that the construction is stable with respect to noise and appears through a relatively simple configuration, we argue that these kinds of complicated indecomposables may appear when applying multidimensional persistence to real data.
Therefore these structures cannot be ignored and we hope that our construction provides insight into a source of indecomposability.


\section{Background}

We start with a quick overview of the necessary background.
First we recall some basic definitions from the representation theory of bound quivers and detail how we check the indecomposability of representations.
More details can be found in \cite{blue}, for example.
In the second part, we explain the block matrix formalism~\cite{matrixmethod} we use to simplify some computations.
We assume some familiarity with algebraic topology \cite{munkres}, in particular homology.

\subsection{Representations of bound quivers}

A quiver is a directed graph. In this work, we consider only quivers with a finite number of vertices and arrows and no cycles.
A particular example is the linear quiver with $n$ vertices.
Let $n\in \mathbb{N}$ and $\tau = (\tau_1,\tau_2,\hdots,\tau_{n-1})$ be a sequence of symbols $\tau_i = f$ or $\tau_i = b$. Below, the two-headed arrow
$\begin{tikzcd}{}\rar[leftrightarrow] & {} \end{tikzcd}$ stands for either an arrow pointing to the right or left.
The quiver $\A_n(\tau)$ is the quiver with $n$ vertices and $n-1$ arrows, where the $i^{th}$ arrow points to the right if $\tau_i = f$ and to the left otherwise:
$\begin{tikzcd}
  \bullet \rar[leftrightarrow]
  & \bullet \rar[leftrightarrow]
  & \hdots \rar[leftrightarrow]
  & \bullet
\end{tikzcd}$.
In the case that all arrows are pointing forwards, we use the notation $\Abar_n = \A_n(f\hdots f)$.

Throughout this work, let $K$ be a field. A representation $V$ of a quiver $Q$ is a collection $V = (V_i,V_\alpha)$  where $V_i$ is a finite dimensional $K$-vector space for each vertex $i$ of $Q$, and each \emph{internal map} $V_\alpha : V_i \rightarrow V_j$ is a linear map for each arrow $\alpha$ from $i$ to $j$ in $Q$.

A homomorphism from $V$ to $W$, both representations of the same quiver $Q$, is a collection of linear maps $\{f_i:V_i \rightarrow W_i\}$ ranging over vertices $i$ of $Q$ such that $W_\alpha f_i = f_j V_\alpha$ for each arrow $\alpha$ from $i$ to $j$. The set of all homomorphisms from $V$ to $W$ is the $K$-vector space $\Hom(V,W)$. The endomorphism ring of a representation $V$ is $\End(V) = \Hom(V,V)$.

General quivers do not impose any constraints on the internal maps of
their representations. However, we will mostly consider the case of
commutative quivers, a special kind of quiver bound by relations. In
particular, representations of a commutative quiver are required to
satisfy the property that they form a commutative diagram. For more
details see~\cite{blue}. In the rest of the article, we shall take $Q$ to mean either a quiver or a commutative quiver, depending on context, and $\rep Q$ its category of representations.

The language of representation theory was introduced to persistence in ~\cite{zigzag}, where zigzag persistence modules were considered as representations of $\A_n(\tau)$.
From now, we use the term persistence module over $Q$ interchangeably with a representation of $Q$.

Any pair of representations $V = (V_i,V_\alpha)$ and $W = (W_i,W_\alpha)$ of a quiver $Q$ has a direct sum $V\oplus W = (V_i\oplus W_i, V_\alpha \oplus W_\alpha)$ which is also a representation of $Q$.
A representation $V$ is said to be indecomposable if $V \cong W\oplus W'$ implies that either $W$ or $W'$ is the zero representation.
We are concerned with the indecomposability of representations and use the following property relating the endomorphism ring with indecomposability.

\begin{definition}
Let $R$ be a ring with unity.
$R$ is said to be \emph{local} if $0 \neq 1$ in $R$ and for each $x\in R$, $x$ or $1-x$ is invertible.
\end{definition}

\begin{lemma}[Corollary 4.8 of \cite{blue}]
  \label{lem:indec_endo}
  Let $V$ be a representation of a (bound) quiver $Q$.
  \begin{enumerate}
  \item If $\End{V}$ is local then $V$ is indecomposable.
  \item If $V$ is finite dimensional and indecomposable, then $\End{V}$ is local.
  \end{enumerate}
\end{lemma}

\subsection{Block matrix formalism}

Our first construction will be for a particular family of bound quivers called commutative ladders, which are the commutative grids of size $2 \times n$.
\begin{definition}
  The commutative ladder of length $n$ with orientation $\tau$, denoted $\CL_n(\tau)$ is
  \[
    \begin{tikzcd}
      \bullet \rar[leftrightarrow]
      & \bullet \rar[leftrightarrow] \arrow[dl, phantom, "\circlearrowleft", description]{}
      & \hdots \rar[leftrightarrow] \arrow[dl, phantom, "\circlearrowleft", description]{}
      & \bullet \arrow[dl, phantom, "\circlearrowleft", description]{} \\
      \bullet \rar[leftrightarrow] \uar
      & \bullet \rar[leftrightarrow] \uar
      & \hdots \rar[leftrightarrow]
      & \bullet \uar
    \end{tikzcd}
  \]
  which is a quiver with two copies of $\A_n(\tau)$ with the same orientation $\tau$ for the top and bottom rows, and  bound by all commutativity relations.
\end{definition}

Let us review the block matrix formalism for persistence modules on commutative ladders $\CL_n(\tau)$ introduced in \cite{matrixmethod}. We denote by $\arr{\rep{\A_n(\tau)}}$ the \emph{arrow category} (also known as the \emph{morphism category}) of $\rep{\A_n(\tau)}$, which is formed by the morphisms of $\rep{\A_n(\tau)}$ as objects. The following proposition allows us to identify representations of the commutative ladder $\CL_n(\tau)$ with morphisms between representations of $\A_n(\tau)$. Since the structure of the latter is well-understood, we use this to simplify some computations.
\begin{lemma}
  \label{prop:isothm1}
  Let $\tau$ be an orientation of length $n$. There is an isomorphism of
  $K$-categories
  \[
    F:\rep{\CL_n(\tau)} \cong \arr{\rep{\A_n(\tau)}}.
  \]
\end{lemma}
\begin{proof}
  Given $M \in \rep{\CL_n(\tau)}$, the bottom row (denoted $M_1$) of $M$ and the top row (denoted $M_2$) of $M$ are representations of $\A_n(\tau)$.  By the commutativity relations imposed on $M$, the internal maps of $M$ pointing upwards defines a morphism $\phi:M_1 \rightarrow M_2$ in $\rep{\A_n(\tau)}$. The functor $F$ maps $M$ to this morphism and admits an obvious inverse.
\end{proof}

Each morphism $(\phi:U\rightarrow V) \in \arr{\rep{\A_n(\tau)}}$ has  representations of $\A_n(\tau)$ as domain and codomain. As such, Both of them can be independently decomposed into interval representations. Thus, $\phi$ is isomorphic to some
\[
  \Phi:
  \bigoplus_{1\leq a \leq b \leq n} \intv[a,b]^{m_{a,b}} \rightarrow
  \bigoplus_{1\leq c \leq d \leq n} \intv[c,d]^{m'_{c,d}}.
\]
Relative to these decompositions, $\Phi$ can be written in a block matrix form $\Phi = [\Phi\itoi{a,b,c,d}]$ where each block is defined by composition with the corresponding inclusion and projection
\[
  \Phi\itoi{a,b,c,d}:
  \begin{tikzcd}[column sep=1.5em]
    \intv[a,b]^{m_{a,b}} \rar{\iota} &
    \bigoplus\limits_{1\leq a \leq b \leq n} \intv[a,b]^{m_{a,b}} \rar{\Phi} &
    \bigoplus\limits_{1\leq c \leq d \leq n} \intv[c,d]^{m'_{c,d}} \rar{\pi} &
    \intv[c,d]^{m'_{c,d}}.
  \end{tikzcd}
\]

Next, we analyze these blocks by looking at the homomorphism spaces between intervals.
\begin{definition}
  The relation $\reltoeq$ is the relation on the set of interval
  representations of $\A_n(\tau)$, $\{\intv[b,d] : 1\leq b \leq d \leq n\}$,
  such that $\intv[a,b] \reltoeq \intv[c,d]$ if
  and only if $\Hom(\intv[a,b],\intv[c,d])$ is nonzero.
\end{definition}

\begin{lemma}[Lemma 1 of \cite{matrixmethod}]
  \label{lemma:homdim}
  Let $\intv[a,b],\intv[c,d]$ be interval representations of $\A_n(\tau)$.
  \begin{enumerate}
  \item The dimension of $\Hom(\intv[a, b], \intv[c, d])$ as a $K$-vector space is either $0$ or $1$.

  \item There exists a canonical basis $\left\{f\itoi{a,b,c,d}\right\}$ for each nonzero $\Hom(\intv[a, b], \intv[c, d])$ such that
    \[
      (f\itoi{a,b,c,d})_i = \left\{
        \begin{array}{ll}
          1_K: K \rightarrow K, & \text{if }i \in [a,b] \cap [c,d] \\
          0, & \text{otherwise.}
        \end{array}
      \right.
    \]
  \end{enumerate}
\end{lemma}
\begin{proof}
 By the commutativity requirement on morphisms between representations, a nonzero morphism $g = \{g_i\}\in\Hom(\intv[a,b],\intv[c,d])$, if it exists, is completely determined by one of its internal morphisms, $g_j \in \Hom(K,K) $ for some fixed $j$ in $[a,b] \cap [c,d]$.
Since $\Hom(K,K)$ is of dimension $1$, part 1 follows.
The $f\itoi{a,b,c,d}$ in part 2 is the $g$ determined by $g_j = 1_K$.
\end{proof}

Each block $\Phi\itoi{a,b,c,d}: \intv[a,b]^{m_{a,b}} \rightarrow \intv[c,d]^{m'_{c,d}}$ can be written in a matrix form where each entry is a morphism in $\Hom(\intv[a,b],\intv[c,d])$.
Lemma~\ref{lemma:homdim} allows us to factor the common basis element of each entry and rewrite $\Phi\itoi{a,b,c,d}$ using a $K$-matrix $M\itoi{a,b,c,d}$ of size $m'_{c,d} \times m_{a,b}$:
\[
  \Phi\itoi{a,b,c,d} = \left\{
    \begin{array}{ll}
      M\itoi{a,b,c,d} f\itoi{a,b,c,d}, & \text{if } \intv[a,b] \reltoeq \intv[c,d], \\
      0, & \text{otherwise}
    \end{array}
  \right.
\]


\section{Commutative grid $2\times 5$}\label{sec:two_by_five}

We now study the construction of a family of indecomposable persistence modules for the commutative ladders of length $5$.
Details are provided for $\rep{\CL_5(ffff)}$ using the formalism introduced for $\arr{\rep{\Abar_5}}$.
Slight adaptations to the construction make it work for any orientation $\tau$.

\subsection{Algebraic construction}

Define the interval representations $D_1 = \intv[2,5]$, $D_2 =\intv[3,4]$, $R_1 = \intv[1,4]$, and $R_2 = \intv[2,3]$ of $\Abar_5$.
Note that for each of the four choices of ordered pairs $(D_i, R_j)$, there exists a nonzero morphism $D_i \rightarrow R_j$, while no nonzero morphism exists between $D_1$ and $D_2$, and between $R_1$ and $R_2$.
The directed graph with vertices given by the chosen intervals and arcs $x\rightarrow y$ defined by $x \reltoeq y$ is the complete bipartite directed graph $\vec{K}_{2,2}$.
\[
  \begin{tikzcd}[column sep=5em, row sep=3em]
    D_1 = \intv[2,5]
    \rar{f\itoi{2,5,1,4}}
    \ar[pos=0.3]{dr}{f\itoi{2,5,2,3}}
    & \intv[1,4] = R_1 \\
    D_2 = \intv[3,4]
    \rar[swap]{f\itoi{3,4,2,3}}
    \ar[swap, crossing over, pos=0.3]{ur}{f\itoi{3,4,1,4}}
    & \intv[2,3] = R_2
  \end{tikzcd}
\]

We provide some insight concerning this configuration from the Auslander-Reiten quiver, which describes irreducible maps between indecomposables. For our purpose, we use directed paths in the Auslander-Reiten quiver to \emph{suggest} possible locations of nonzero maps between indecomposable representations. For $I_1 \not\cong I_2$ indecomposable representations, the presence of a directed path in the Auslander-Reiten quiver from $I_1$ to $I_2$ does \emph{not} imply that there is a nonzero morphism $f:I_1 \rightarrow I_2$. See for example the path $\intv[4,4] \rightarrow \intv[3,4] \rightarrow \intv[3,3]$ below. The converse (existence of nonzero map implies directed path) does hold in the representation finite case (see Collorary IV.5.6 of \cite{blue}).

Below, we exhibit the Auslander-Reiten quiver
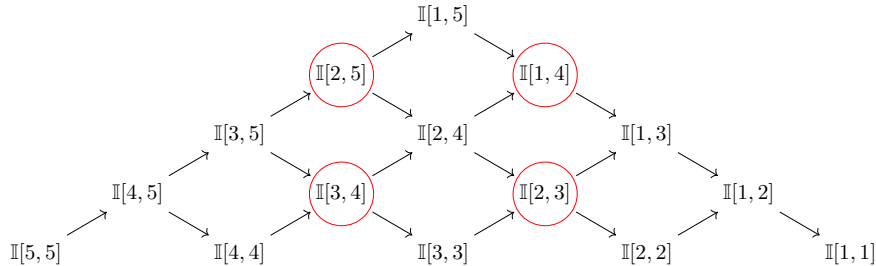
\[
  \scalebox{0.8}{
    \begin{tikzpicture}[baseline=(A.center)]
      \matrix (A) [matrix of math nodes, column sep= 6 mm, row sep = 4
      mm, nodes in empty cells, ampersand replacement=\&]
      {
        \&           \&           \&           \& \intv[1,5] \&           \&           \&           \&    \\
        \&           \&           \& \intv[2,5] \&           \& \intv[1,4] \&           \&           \& \\
        \&           \& \intv[3,5] \&           \& \intv[2,4] \&           \& \intv[1,3] \&           \& \\
        \& \intv[4,5] \&           \& \intv[3,4] \&           \& \intv[2,3] \&           \& \intv[1,2] \&      \\
        \intv[5,5] \&           \& \intv[4,4] \&           \& \intv[3,3] \&           \& \intv[2,2] \&           \& \intv[1,1]   \\
      };

      \draw[red] (A-2-4) circle (1.5em);
      \draw[red] (A-4-4) circle (1.5em);

      \draw[red] (A-2-6) circle (1.5em);
      \draw[red] (A-4-6) circle (1.5em);

      \foreach \y in {5,4,3,2} {
        \foreach \l in {2,...,\y} {
          \pgfmathtruncatemacro{\x}{2*\l-3  + 5-\y};

          \pgfmathtruncatemacro{\yd}{\y-1};
          \pgfmathtruncatemacro{\xu}{\x+1};
          \pgfmathtruncatemacro{\xuu}{\x+2};
          \path [->] (A-\y-\x) edge (A-\yd-\xu);
          \path [->] (A-\yd-\xu) edge (A-\y-\xuu);
        }
      }
    \end{tikzpicture}
  }
\]
of $\Abar_5$, together with our choice of indecomposables encircled.

On the other hand, for example, we note that $D'_1 = \intv[3,5]$, $D'_2 =\intv[4,4]$, $R'_1 = \intv[2,4]$, and $R'_2 = \intv[3,3]$ does not work since there is no nonzero map from $D'_2$ to $R'_2$, even though there is a directed path from $D_2'$ to $R'_2$ in the Auslander-Reiten quiver.

Such a configuration is crucial to ensure that all four blocks can be nonzero in $\phi(d,\lambda)$ defined below, and to ensure indecomposability. The second ingredient we use is the Jordan cell $J_d(\lambda)$.
Given $d\geq 1$ and $\lambda\in K$, $J_d(\lambda)$ is the matrix with value $\lambda$ on the diagonal, $1$ on the superdiagonal, and $0$ elsewhere.
\[
J_3(\lambda) = \left[\begin{array}{ccc} \lambda & 1 & 0\\ 0 & \lambda &1 \\ 0 & 0 & \lambda\end{array}\right].
\]

With this, we are ready to define an arrow $\phi(d,\lambda)$ and a
representation $M(d,\lambda)$, which are identified using
Proposition~\ref{prop:isothm1} as $F(M(d,\lambda)) =\phi(d,\lambda)$.

\begin{definition}
  Let $d\geq 1$ and $\lambda \in K$.
  \begin{enumerate}
  \item We define the arrow $\phi(d,\lambda) \in \arr{\rep{\Abar_5}}$ $
      \phi(d,\lambda):
      \intv[3,4]^d \oplus \intv[2,5]^d
      \rightarrow
      \intv[1,4]^d \oplus \intv[2,3]^d$
    by the matrix form
   $
      \phi(d,\lambda) = \left[
        \begin{array}{cc}
          I f\itoi{3,4,1,4} & I f\itoi{2,5,1,4} \\
          I f\itoi{3,4,2,3} & J_d(\lambda) f\itoi{2,5,2,3}
        \end{array}
      \right]
    $
    where $I$ is the $d\times d$ identity matrix.
  \item We also define the representation $M(d,\lambda) \in \rep{\CL_5(ffff)}$ by
    \[
      M(d,\lambda):
      \begin{tikzcd}[ampersand replacement=\&]
        K^d \rar{\smat{I\\0}} \&
        K^{2d} \rar{\Id} \&
        K^{2d} \rar{\smat{I&0}} \&
        K^d \rar \&
        0
        \\
        0 \rar \uar \&
        K^d \rar[swap]{\smat{0\\I}} \uar{\smat{I\\J_d(\lambda)}} \&
        K^{2d} \rar[swap]{\Id} \uar{\smat{I&I\\I&J_d(\lambda)}} \&
        K^{2d} \rar[swap]{\smat{0&I}} \uar{\smat{I&I}} \&
        K^d \uar
      \end{tikzcd}
    \]
  \end{enumerate}
\end{definition}

We now show that the representations constructed above are indeed indecomposable and are pairwise non-isomorphic.

\begin{theorem}
  \label{prop:cl5_example}
  Let $d \geq 1$ and $\lambda,\lambda' \in K$.
  \begin{enumerate}
  \item $M(d, \lambda)$ is indecomposable.
  \item If $\lambda \neq \lambda'$ then $M(d,\lambda) \not\cong M(d,\lambda')$.
  \end{enumerate}
\end{theorem}
\begin{proof}
  We check that $\End{M(d,\lambda)}$ is local. By Proposition~\ref{prop:isothm1}, $\End M(d,\lambda) \cong \End\phi(d,\lambda)$.
  Letting $(g_0,g_1)$ be an endomorphism of $\phi(d,\lambda)$, the diagram
  \begin{equation}
    \label{eq:endo_comm}
    \begin{tikzcd}
      \intv[3,4]^d \oplus \intv[2,5]^d \rar{\phi(d,\lambda)} \dar{g_0}&
      \intv[1,4]^d \oplus \intv[2,3]^d \dar{g_1}
      \\
      \intv[3,4]^d \oplus \intv[2,5]^d \rar{\phi(d,\lambda)} &
      \intv[1,4]^d \oplus \intv[2,3]^d
    \end{tikzcd}
  \end{equation}
  commutes.
  Then, $g_0$ and $g_1$ in matrix form with respect to the decompositions are
  \[
    g_0 = \left[
      \begin{array}{cc}
        Af\itoi{3,4,3,4} & 0 \\
        0 & Bf\itoi{2,5,2,5}
      \end{array}
    \right]
    \text{ and }
    g_1 = \left[
      \begin{array}{cc}
        Cf\itoi{1,4,1,4} & 0 \\
        0 & Df\itoi{2,3,2,3}
      \end{array}
    \right]
  \]
  where $A,B,C,D$ are $K$-matrices of size $d\times d$.
  Since there are no nonzero morphisms from $\intv[2,5]$ to $\intv[3,4]$, nor vice versa, the off-diagonal entries of $g_0$ are $0$. Likewise, there are no nonzero morphisms between $\intv[1,4]$ and $\intv[2,3]$ so the off-diagonal entries of $g_1$ are $0$.

  From the commutativity of Eq.~\eqref{eq:endo_comm}, we get the equality
  \[
    \left[
      \begin{array}{cc}
        A f\itoi{3,4,1,4} & A f\itoi{2,5,1,4}\\
        B f\itoi{3,4,2,3} & BJ_d(\lambda) f\itoi{2,5,2,3}
      \end{array}
    \right]
    =
    \left[
      \begin{array}{cc}
        C f\itoi{3,4,1,4} & D f\itoi{2,5,1,4}\\
        C f\itoi{3,4,2,3} & J_d(\lambda)D f\itoi{2,5,2,3}
      \end{array}
    \right]
  \]
  which implies that $A=B=C=D$ and $AJ_d(\lambda) = J_d(\lambda)A$ as $K$-matrices since the morphisms $f\itoi{a,b,c,d}$ appearing above are nonzero. We infer that
  \[
    \End{\phi(d,\lambda)} \cong \{A\in K^{d\times d} \suchthat AJ_d(\lambda) =J_d(\lambda)A\}.
  \]
  A direct computation shows that $A$ is a member of the above ring if and only if $A$ is upper triangular Toeplitz:
  \[
    A = \left[
      \begin{array}{ccccc}
        a_1 & a_2 & \hdots & a_{d-1} &  a_d \\
        0 & a_1 & a_2 & \hdots & a_{d-1} \\
        & & \ddots & \ddots & \vdots \\
        0 &  \hdots & 0 & a_1 & a_2 \\
        0 &  \hdots & 0 & 0 & a_1 \\
      \end{array}
    \right].
  \]
The matrix $A$ is invertible if and only if $a_1$ is nonzero, and so for any $A$, either $A$ or $I-A$ is invertible. Thus, $\End{\phi(d,\lambda)}$ is local and $M(d,\lambda)$ is indecomposable.

  By a similar computation, $M(d,\lambda) \cong M(d,\lambda')$ implies that there is some invertible matrix $A$ such that $AJ_d(\lambda) = J_d(\lambda')A$ which is impossible when $\lambda\neq\lambda'$.
\end{proof}


Proposition~\ref{prop:cl5_example} together with the observation that if $d \neq d'$ then $M(d,\lambda) \not\cong M(d',\lambda')$ provides an easy proof for the following corollary when $\tau = ffff$.
Moreover, the method above can be used to produce similar examples for any orientation $\tau$ on length $n=5$ by finding a similar configuration isomorphic to $\vec{K}_{2,2}$ from the intervals and arcs determined by $\reltoeq$.
As a result, we get a constructive proof of the following result as a corollary.
\begin{corollary}
For any $n\geq 5$ and orientation $\tau$, the commutative ladder
  $\CL_n(\tau)$ is representation infinite.
\end{corollary}


We give realizations of our indecomposable persistence modules for $\lambda = 0$, first relying purely on topological spaces and then using a geometric Vietoris-Rips construction.

\subsection{Topological construction}
Given $d\geq 1$, we build a diagram $\mathbb{S}(d)$ of topological spaces and inclusions.
The spaces in the middle column take the form of a sandal consisting of a planar sole and a set of $d$ straps.
Other spaces are either missing some edges or have some faces filled in.
Figure~\ref{fig:sandal_cycles} presents the complete realization.

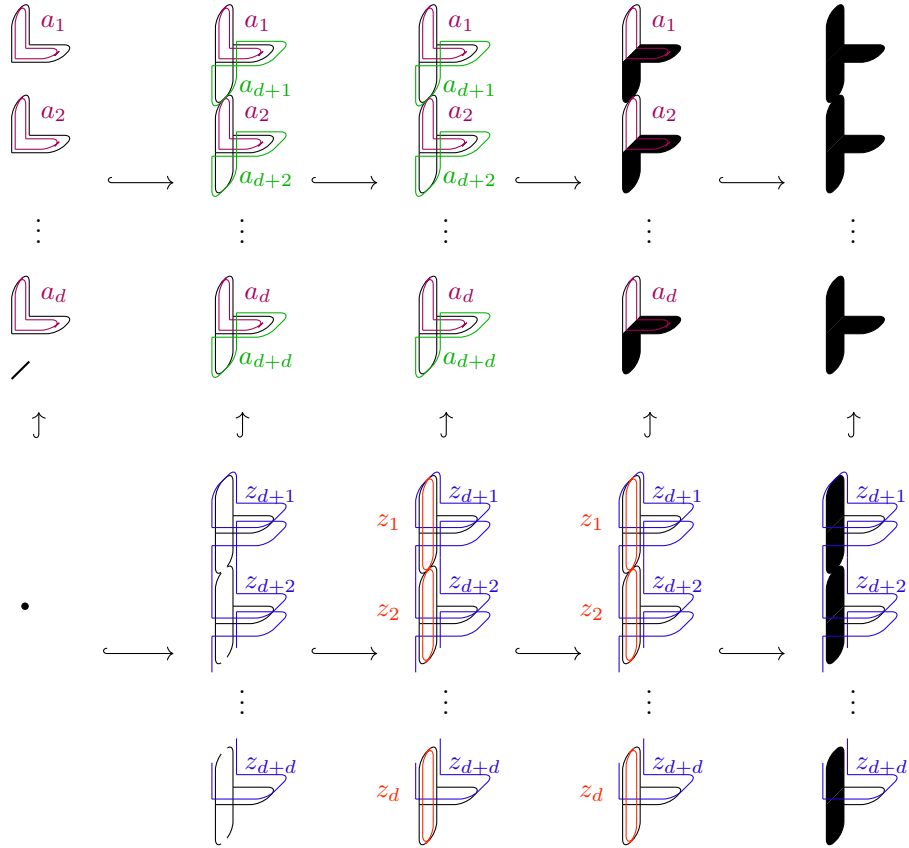
\begin{figure}[ht]
  \caption{Diagram of spaces; and representatives for homology bases (in color).} \label{fig:sandal_cycles}
  \newcommand\definesandalpts{
  \coordinate (Ia) at (0,1,0);
  \coordinate (Ib) at (0,1,1);
  \coordinate (Sa) at (0,0,0);
  \coordinate (Sb) at (0,0,1);
  \coordinate (Sc) at (1,0,1);
  \coordinate (Sd) at (1,0,0);
  \coordinate (Id) at (1,1,0);

  \coordinate (IaSh) at (0,1,0.2);
  \coordinate (IbSh) at (0,1,0.8);
  \coordinate (SaSh) at (0,0,0.2);
  \coordinate (SbSh) at (0,0,0.8);
  \coordinate (ScSh) at (0.8,0,0.8);
  \coordinate (SdSh) at (0.8,0,0.2);
  \coordinate (IdSh) at (0.8,0.8,0);

  \coordinate (IaEx) at (0,1,-0.2);
  \coordinate (IbEx) at (0,1,1.2);
  \coordinate (SaEx) at (0,0,-0.2);
  \coordinate (SbEx) at (0,0,1.2);
  \coordinate (ScEx) at (1.2,0,1.2);
  \coordinate (SdEx) at (1.2,0,-0.2);
  \coordinate (IdEx) at (1.2,1.2,0);
}

\colorlet{UpperRowColor1}{violet!70!red}
\colorlet{UpperRowColor2}{green!70!black}

\colorlet{LowerRowColor2}{red!70!orange}
\colorlet{LowerRowColor1}{blue!70!violet}

\newcommand{\myscale}{0.6}

\newcommand{\shiftpt}[2]{
  ($ #1 + (0,2*#2,0) $)
}

\newcommand{\drawuphk}[2]{
  \draw[rounded corners, #1] \shiftpt{(Sa)}{#2} -- \shiftpt{(Ia)}{#2} -- \shiftpt{(Ib)}{#2} -- \shiftpt{(Sb)}{#2};
}
\newcommand{\drawdwhk}[2]{
  \draw[rounded corners, #1] \shiftpt{(Sa)}{#2} -- \shiftpt{(Ia)}{#2 - 2}-- \shiftpt{(Ib)}{#2-2} -- \shiftpt{(Sb)}{#2};
}

\newcommand{\drawstrp}[2]{
  \draw[rounded corners, #1] \shiftpt{(Sb)}{#2} -- \shiftpt{(Sc)}{#2} -- \shiftpt{(Sd)}{#2} -- \shiftpt{(Sa)}{#2};
}

\newcommand{\shrunkuphk}[2]{
  \draw[rounded corners, #1] \shiftpt{(SaSh)}{#2} -- \shiftpt{(IaSh)}{#2} -- \shiftpt{(IbSh)}{#2} -- \shiftpt{(SbSh)}{#2};
}
\newcommand{\shrunkdwhk}[2]{
  \draw[rounded corners, #1] \shiftpt{(SaSh)}{#2} -- \shiftpt{(IaSh)}{#2 - 2}-- \shiftpt{(IbSh)}{#2-2} -- \shiftpt{(SbSh)}{#2};
}

\newcommand{\shrunkstrp}[2]{
  \draw[rounded corners, #1] \shiftpt{(SbSh)}{#2} -- \shiftpt{(ScSh)}{#2} -- \shiftpt{(SdSh)}{#2} -- \shiftpt{(SaSh)}{#2};
}

\newcommand{\expanduphk}[3]{
  \draw[rounded corners, #1] \shiftpt{(SaEx)}{#3} -- \shiftpt{(IaEx)}{#2} -- \shiftpt{(IbEx)}{#2} -- \shiftpt{(SbEx)}{#3};
}
\newcommand{\expanddwhk}[3]{
  \draw[rounded corners, #1] \shiftpt{(SaEx)}{#3} -- \shiftpt{(IaEx)}{#2 - 2}-- \shiftpt{(IbEx)}{#2-2} -- \shiftpt{(SbEx)}{#3};
}

\newcommand{\expanduplegs}[3]{
  \draw[#1] \shiftpt{(SaEx)}{#3} -- \shiftpt{(IaEx)}{#2};
  \draw[#1] \shiftpt{(IbEx)}{#2} -- \shiftpt{(SbEx)}{#3};
}
\newcommand{\expanddwlegs}[3]{
  \draw[#1] \shiftpt{(SaEx)}{#3} -- \shiftpt{(IaEx)}{#2 - 2};
  \draw[#1] \shiftpt{(IbEx)}{#2-2} -- \shiftpt{(SbEx)}{#3};
}

\newcommand{\expandstrp}[2]{
  \draw[rounded corners, #1] \shiftpt{(SbEx)}{#2} -- \shiftpt{(ScEx)}{#2} -- \shiftpt{(SdEx)}{#2} -- \shiftpt{(SaEx)}{#2};
}

\centering
\begin{tikzcd}[ampersand replacement = \&, row sep = 1em, column sep = 2.5em,nodes={inner sep = 1.2em}, arrows={hookrightarrow}]
  \begin{tikzpicture}[scale=\myscale,baseline=(current bounding box.center)]
    \definesandalpts
    \drawuphk{black}{3}
    \drawstrp{black}{3}
    \drawuphk{black}{2}
    \drawstrp{black}{2}
    \drawuphk{black}{0}
    \drawstrp{black}{0}
    \node[xshift=1em] at \shiftpt{(Sb)}{1} {$\vdots$};
    \draw[black, thick, rounded corners] \shiftpt{(Ia)}{-1} -- \shiftpt{(Ib)}{-1};
    \path[use as bounding box] \shiftpt{(Id)}{3} rectangle \shiftpt{(Ib)}{-1};

    \shrunkuphk{UpperRowColor1}{3}
    \shrunkstrp{UpperRowColor1}{3}
    \shrunkuphk{UpperRowColor1}{2}
    \shrunkstrp{UpperRowColor1}{2}
    \shrunkuphk{UpperRowColor1}{0}
    \shrunkstrp{UpperRowColor1}{0}

    \node[anchor=west, UpperRowColor1] at \shiftpt{(IbEx)}{3} {$a_1$};
    \node[anchor=west, UpperRowColor1] at \shiftpt{(IbEx)}{2} {$a_2$};
    \node[anchor=west, UpperRowColor1] at \shiftpt{(IbEx)}{0} {$a_d$};
  \end{tikzpicture}
  \rar \&
  \begin{tikzpicture}[scale=\myscale,baseline=(current bounding box.center)]
    \definesandalpts
    \drawuphk{black}{3}
    \drawstrp{black}{3}
    \drawdwhk{black}{3}
    \drawuphk{black}{2}
    \drawstrp{black}{2}
    \drawdwhk{black}{2}
    \drawuphk{black}{0}
    \drawstrp{black}{0}
    \drawdwhk{black}{0}
    \node[xshift=1em] at \shiftpt{(Sb)}{1} {$\vdots$};
    \path[use as bounding box] \shiftpt{(Id)}{3} rectangle \shiftpt{(Ib)}{-1};

    \shrunkuphk{UpperRowColor1}{3}
    \shrunkstrp{UpperRowColor1}{3}
    \shrunkuphk{UpperRowColor1}{2}
    \shrunkstrp{UpperRowColor1}{2}
    \shrunkuphk{UpperRowColor1}{0}
    \shrunkstrp{UpperRowColor1}{0}
    \expandstrp{UpperRowColor2}{3}
    \expanddwhk{UpperRowColor2}{3}{3}
    \expandstrp{UpperRowColor2}{2}
    \expanddwhk{UpperRowColor2}{2}{2}
    \expandstrp{UpperRowColor2}{0}
    \expanddwhk{UpperRowColor2}{0}{0}

    \node[anchor=west, UpperRowColor1] at \shiftpt{(IbEx)}{3} {$a_1$};
    \node[anchor=west, UpperRowColor1] at \shiftpt{(IbEx)}{2} {$a_2$};
    \node[anchor=west, UpperRowColor1] at \shiftpt{(IbEx)}{0} {$a_d$};
    \node[anchor=north, UpperRowColor2] at \shiftpt{(SdSh)}{3} {$a_{d+1}$};
    \node[anchor=north, UpperRowColor2] at \shiftpt{(SdSh)}{2} {$a_{d+2}$};
    \node[anchor=north, UpperRowColor2] at \shiftpt{(SdSh)}{0} {$a_{d+d}$};
  \end{tikzpicture}
  \rar \&
  \begin{tikzpicture}[scale=\myscale,baseline=(current bounding box.center)]
    \definesandalpts
    \drawuphk{black}{3}
    \drawstrp{black}{3}
    \drawdwhk{black}{3}
    \drawuphk{black}{2}
    \drawstrp{black}{2}
    \drawdwhk{black}{2}
    \drawuphk{black}{0}
    \drawstrp{black}{0}
    \drawdwhk{black}{0}
    \node[xshift=1em] at \shiftpt{(Sb)}{1} {$\vdots$};
    \path[use as bounding box] \shiftpt{(Id)}{3} rectangle \shiftpt{(Ib)}{-1};

    \shrunkuphk{UpperRowColor1}{3}
    \shrunkstrp{UpperRowColor1}{3}
    \shrunkuphk{UpperRowColor1}{2}
    \shrunkstrp{UpperRowColor1}{2}
    \shrunkuphk{UpperRowColor1}{0}
    \shrunkstrp{UpperRowColor1}{0}
    \expandstrp{UpperRowColor2}{3}
    \expanddwhk{UpperRowColor2}{3}{3}
    \expandstrp{UpperRowColor2}{2}
    \expanddwhk{UpperRowColor2}{2}{2}
    \expandstrp{UpperRowColor2}{0}
    \expanddwhk{UpperRowColor2}{0}{0}

    \node[anchor=west, UpperRowColor1] at \shiftpt{(IbEx)}{3} {$a_1$};
    \node[anchor=west, UpperRowColor1] at \shiftpt{(IbEx)}{2} {$a_2$};
    \node[anchor=west, UpperRowColor1] at \shiftpt{(IbEx)}{0} {$a_d$};
    \node[anchor=north, UpperRowColor2] at \shiftpt{(SdSh)}{3} {$a_{d+1}$};
    \node[anchor=north, UpperRowColor2] at \shiftpt{(SdSh)}{2} {$a_{d+2}$};
    \node[anchor=north, UpperRowColor2] at \shiftpt{(SdSh)}{0} {$a_{d+d}$};
  \end{tikzpicture}
  \rar \&
  \begin{tikzpicture}[scale=\myscale,baseline=(current bounding box.center)]
    \definesandalpts
    \drawuphk{black}{3}
    \drawstrp{black,fill}{3}
    \drawdwhk{black,fill}{3}
    \drawuphk{black}{2}
    \drawstrp{black,fill}{2}
    \drawdwhk{black,fill}{2}
    \drawuphk{black}{0}
    \drawstrp{black,fill}{0}
    \drawdwhk{black,fill}{0}
    \node[xshift=1em] at \shiftpt{(Sb)}{1} {$\vdots$};
    \path[use as bounding box] \shiftpt{(Id)}{3} rectangle \shiftpt{(Ib)}{-1};

    \shrunkuphk{UpperRowColor1}{3}
    \shrunkstrp{UpperRowColor1}{3}
    \shrunkuphk{UpperRowColor1}{2}
    \shrunkstrp{UpperRowColor1}{2}
    \shrunkuphk{UpperRowColor1}{0}
    \shrunkstrp{UpperRowColor1}{0}

    \node[anchor=west, UpperRowColor1] at \shiftpt{(IbEx)}{3} {$a_1$};
    \node[anchor=west, UpperRowColor1] at \shiftpt{(IbEx)}{2} {$a_2$};
    \node[anchor=west, UpperRowColor1] at \shiftpt{(IbEx)}{0} {$a_d$};
  \end{tikzpicture}
  \rar \&
  \begin{tikzpicture}[scale=\myscale,baseline=(current bounding box.center)]
    \definesandalpts
    \drawuphk{black,fill}{3}
    \drawstrp{black,fill}{3}
    \drawdwhk{black,fill}{3}
    \drawuphk{black,fill}{2}
    \drawstrp{black,fill}{2}
    \drawdwhk{black,fill}{2}
    \drawuphk{black,fill}{0}
    \drawstrp{black,fill}{0}
    \drawdwhk{black,fill}{0}
    \node[xshift=1em] at \shiftpt{(Sb)}{1} {$\vdots$};
    \path[use as bounding box] \shiftpt{(Id)}{3} rectangle \shiftpt{(Ib)}{-1};
  \end{tikzpicture}
  \\
  \begin{tikzpicture}[scale=\myscale,baseline=(current bounding box.center)]
    \definesandalpts
    \draw[fill] \shiftpt{(Sa)}{2} circle (2pt);
    \path[use as bounding box] \shiftpt{(Id)}{3} rectangle \shiftpt{(Ib)}{-1};
  \end{tikzpicture}
  \rar \uar \&
  \begin{tikzpicture}[scale=\myscale,baseline=(current bounding box.center)]
    \definesandalpts
    \drawuphk{black}{3}
    \drawstrp{black}{3}
    \drawdwhk{black}{3}
    \drawuphk{black}{2}
    \drawstrp{black}{2}
    \drawdwhk{black}{2}
    \drawuphk{black}{0}
    \drawstrp{black}{0}
    \drawdwhk{black}{0}
    \node[xshift=1em] at \shiftpt{(Sb)}{1} {$\vdots$};
    \path[use as bounding box] \shiftpt{(Id)}{3} rectangle \shiftpt{(Ib)}{-1};

    \node [coordinate] (Ina) at \shiftpt{(Ia)}{-1.2} {};
    \node [coordinate] (Inb) at \shiftpt{(Ib)}{-0.8} {};

    \node [coordinate] (I0a) at \shiftpt{(Ia)}{-0.2} {};
    \node [coordinate] (I0b) at \shiftpt{(Ib)}{0.2} {};

    \node [coordinate] (I1a) at \shiftpt{(Ia)}{0.8} {};
    \node [coordinate] (I1b) at \shiftpt{(Ib)}{1.2} {};

    \node [coordinate] (I2a) at \shiftpt{(Ia)}{1.8} {};
    \node [coordinate] (I2b) at \shiftpt{(Ib)}{2.2} {};

    \expanduphk{LowerRowColor1}{3}{3.1}
    \expandstrp{LowerRowColor1}{3.1}

    \expandstrp{LowerRowColor1}{2.9}
    \expanddwlegs{LowerRowColor1}{3}{2.9}
    \expanduplegs{LowerRowColor1}{2}{2.1}
    \expandstrp{LowerRowColor1}{2.1}

    \expandstrp{LowerRowColor1}{1.9}
    \expanddwlegs{LowerRowColor1}{2}{1.9}

    \expanduplegs{LowerRowColor1}{0}{0.1}
    \expandstrp{LowerRowColor1}{0.1}
    \draw[fill, white] ($ (I2a)!0.35!(I2b) $) rectangle ($ (I2b)!0.35!(I2a) $);
    \draw[fill, white] ($ (I1a)!0.35!(I1b) $) rectangle ($ (I1b)!0.35!(I1a) $);
    \draw[fill, white] ($ (I0a)!0.35!(I0b) $) rectangle ($ (I0b)!0.35!(I0a) $);
    \draw[fill, white] ($ (Ina)!0.35!(Inb) $) rectangle ($ (Inb)!0.35!(Ina) $);

    \node[anchor=west, LowerRowColor1] at \shiftpt{(IbEx)}{3} {$z_{d+1}$};
    \node[anchor=west, LowerRowColor1] at \shiftpt{(IbEx)}{2} {$z_{d+2}$};
    \node[anchor=west, LowerRowColor1] at \shiftpt{(IbEx)}{0} {$z_{d+d}$};
  \end{tikzpicture}
  \rar \uar \&
  \begin{tikzpicture}[scale=\myscale,baseline=(current bounding box.center)]
    \definesandalpts
    \drawuphk{black}{3}
    \drawstrp{black}{3}
    \drawdwhk{black}{3}
    \drawuphk{black}{2}
    \drawstrp{black}{2}
    \drawdwhk{black}{2}
    \drawuphk{black}{0}
    \drawstrp{black}{0}
    \drawdwhk{black}{0}
    \node[xshift=1em] at \shiftpt{(Sb)}{1} {$\vdots$};
    \path[use as bounding box] \shiftpt{(Id)}{3} rectangle \shiftpt{(Ib)}{-1};

    \expanduphk{LowerRowColor1}{3}{3.1}
    \expandstrp{LowerRowColor1}{3.1}
    \expandstrp{LowerRowColor1}{2.9}
    \expanddwlegs{LowerRowColor1}{3}{2.9}
    \expanduplegs{LowerRowColor1}{2}{2.1}
    \expandstrp{LowerRowColor1}{2.1}
    \expandstrp{LowerRowColor1}{1.9}
    \expanddwlegs{LowerRowColor1}{2}{1.9}
    \expanduplegs{LowerRowColor1}{0}{0.1}
    \expandstrp{LowerRowColor1}{0.1}

    \shrunkuphk{LowerRowColor2}{3}
    \shrunkdwhk{LowerRowColor2}{3}
    \shrunkuphk{LowerRowColor2}{2}
    \shrunkdwhk{LowerRowColor2}{2}
    \shrunkuphk{LowerRowColor2}{0}
    \shrunkdwhk{LowerRowColor2}{0}

    \node[anchor=west, LowerRowColor1] at \shiftpt{(IbEx)}{3} {$z_{d+1}$};
    \node[anchor=west, LowerRowColor1] at \shiftpt{(IbEx)}{2} {$z_{d+2}$};
    \node[anchor=west, LowerRowColor1] at \shiftpt{(IbEx)}{0} {$z_{d+d}$};
    \node[anchor=east, LowerRowColor2] at \shiftpt{(SaSh)}{3} {$z_{1}$};
    \node[anchor=east, LowerRowColor2] at \shiftpt{(SaSh)}{2} {$z_{2}$};
    \node[anchor=east, LowerRowColor2] at \shiftpt{(SaSh)}{0} {$z_{d}$};
  \end{tikzpicture}
  \rar \uar \&
  \begin{tikzpicture}[scale=\myscale,baseline=(current bounding box.center)]
    \definesandalpts
    \drawuphk{black}{3}
    \drawstrp{black}{3}
    \drawdwhk{black}{3}
    \drawuphk{black}{2}
    \drawstrp{black}{2}
    \drawdwhk{black}{2}
    \drawuphk{black}{0}
    \drawstrp{black}{0}
    \drawdwhk{black}{0}
    \node[xshift=1em] at \shiftpt{(Sb)}{1} {$\vdots$};
    \path[use as bounding box] \shiftpt{(Id)}{3} rectangle \shiftpt{(Ib)}{-1};

    \expanduphk{LowerRowColor1}{3}{3.1}
    \expandstrp{LowerRowColor1}{3.1}
    \expandstrp{LowerRowColor1}{2.9}
    \expanddwlegs{LowerRowColor1}{3}{2.9}
    \expanduplegs{LowerRowColor1}{2}{2.1}
    \expandstrp{LowerRowColor1}{2.1}
    \expandstrp{LowerRowColor1}{1.9}
    \expanddwlegs{LowerRowColor1}{2}{1.9}
    \expanduplegs{LowerRowColor1}{0}{0.1}
    \expandstrp{LowerRowColor1}{0.1}

    \shrunkuphk{LowerRowColor2}{3}
    \shrunkdwhk{LowerRowColor2}{3}
    \shrunkuphk{LowerRowColor2}{2}
    \shrunkdwhk{LowerRowColor2}{2}
    \shrunkuphk{LowerRowColor2}{0}
    \shrunkdwhk{LowerRowColor2}{0}

    \node[anchor=west, LowerRowColor1] at \shiftpt{(IbEx)}{3} {$z_{d+1}$};
    \node[anchor=west, LowerRowColor1] at \shiftpt{(IbEx)}{2} {$z_{d+2}$};
    \node[anchor=west, LowerRowColor1] at \shiftpt{(IbEx)}{0} {$z_{d+d}$};
    \node[anchor=east, LowerRowColor2] at \shiftpt{(SaSh)}{3} {$z_{1}$};
    \node[anchor=east, LowerRowColor2] at \shiftpt{(SaSh)}{2} {$z_{2}$};
    \node[anchor=east, LowerRowColor2] at \shiftpt{(SaSh)}{0} {$z_{d}$};
  \end{tikzpicture}
  \rar \uar \&
  \begin{tikzpicture}[scale=\myscale,baseline=(current bounding box.center)]
    \definesandalpts
    \drawuphk{black,fill}{3}
    \drawstrp{black}{3}
    \drawdwhk{black,fill}{3}
    \drawuphk{black,fill}{2}
    \drawstrp{black}{2}
    \drawdwhk{black,fill}{2}
    \drawuphk{black,fill}{0}
    \drawstrp{black}{0}
    \drawdwhk{black,fill}{0}
    \node[xshift=1em] at \shiftpt{(Sb)}{1} {$\vdots$};
    \path[use as bounding box] \shiftpt{(Id)}{3} rectangle \shiftpt{(Ib)}{-1};

    \expanduphk{LowerRowColor1}{3}{3.1}
    \expandstrp{LowerRowColor1}{3.1}
    \expandstrp{LowerRowColor1}{2.9}
    \expanddwlegs{LowerRowColor1}{3}{2.9}
    \expanduplegs{LowerRowColor1}{2}{2.1}
    \expandstrp{LowerRowColor1}{2.1}
    \expandstrp{LowerRowColor1}{1.9}
    \expanddwlegs{LowerRowColor1}{2}{1.9}
    \expanduplegs{LowerRowColor1}{0}{0.1}
    \expandstrp{LowerRowColor1}{0.1}
    \node[anchor=west, LowerRowColor1] at \shiftpt{(IbEx)}{3} {$z_{d+1}$};
    \node[anchor=west, LowerRowColor1] at \shiftpt{(IbEx)}{2} {$z_{d+2}$};
    \node[anchor=west, LowerRowColor1] at \shiftpt{(IbEx)}{0} {$z_{d+d}$};
  \end{tikzpicture}
  \uar
\end{tikzcd}
\end{figure}

The resulting diagram of spaces has maps that are all inclusions, and therefore all squares commute.
Using the singular homology functor with coefficient in field $K$, we obtain a representation $H_1(\mathbb{S}(d))$ of $\CL_5(ffff)$.

\begin{theorem}
$$H_1(\mathbb{S}(d)) \cong M(d,0)$$
\end{theorem}

\begin{proof}
  Relative to the choice of bases indicated in Fig.~\ref{fig:sandal_cycles}, the induced maps have the same matrix forms as the matrices in $M(d,0)$.
\end{proof}

\subsection{Vietoris-Rips construction}

Next, we use the well-known Vietoris-Rips construction to build
simplicial complexes having the suitable topology. Recall that the
Vietoris-Rips complex $V(P,r)$ is the clique complex of the set of all
edges that can be formed from points $p\in P$ with length less than $2r$.
Note that, for $r \leq r'$, $V(P,r)\subset V(P,r')$, and hence we have a filtration.

In our construction, we provide two different points sets $P_\ell$ and $P_u$ corresponding to the lower and upper rows. The point sets $P_\ell$ and $P_u$ are built by assembling what we call \emph{tiles} in a regular pattern.
Every tile we consider is a planar point set contained within a square with sides of length $5$.
We define three types of tiles: $A$, $B$, $C$, used for $P_u$ and three: $D$, $E$, $F$, used for $P_\ell$. The tile denoted $\bar E$ is obtained by reflection of tile $E$ and we call it the reversed $E$ tile.
In local coordinates, the points in the tiles are displayed in Table~\ref{tab:coordinates}.
\begin{table}[htb]
\caption{Local coordinates of points within tiles.}\label{tab:coordinates}
\begin{center}
\begin{tabular}{|l|c|c|c|c|c|c|c|}
\hline
Index\textbackslash Tile & $A$        & $B$     & $C$       & $D$     & $E$     & $F$       & $\bar{E}$  \\ \hline
1                        &(0, 0)			&(0, 0)		&(0, 0)			&(0, 0)		&(0, 0)		&(0, 0)			&(0,0) \\
2                        &(0, 5)			&(0, 5)		&(0, 5)			&(0, 5)		&(0, 5)		&(0, 5)			&(0,5) \\
3                        &(5, 5)			&(5, 5)		&(5, 5)			&(5, 5)		&(5, 5)		&(5, 5)			&(5,5) \\
4                        &(5, 0)			&(5, 0)		&(5, 0)			&(5, 0)		&(5, 0)		&(5, 0)			&(5,0) \\ \hline
5                        &(1.5, 0)		&(1.3, 0)	&(.7, 0)		&(1.3, 0)	&(1.3, 0)	&(.5, 0)		&(1.3,0) \\
6                        &(3.5, 0)		&(3.7, 0)	&(4.3, 0)		&(3.7, 0)	&(3.7, 0)	&(4.5, 0)		&(3.7,0) \\ \hline
7                        &(1.5, 5)		&(1.3, 5)	&(1.5, 5)		&(1.3, 5)	&(1.3, 5)	&(.7, 5)		&(1.3,5) \\
8                        &(3.5, 5)		&(3.7, 5)	&(3.5, 5)		&(3.7, 5)	&(3.7, 5)	&(4.3, 5)		&(3.7,5) \\ \hline
9                        &(0, 1.5)		&(0, 0.7)	&(0, 1.5)		&(0, 1.3)	&(0, .5)	&(0, 1.3)		&(0,1) \\
10                       &(0, 3.5)		&(0, 4.3)	&(0, 3.5)		&(0, 3.7)	&(0, 4.5)	&(0, 3.7)		&(0,4) \\ \hline
11                       &(5, 0.7)		&(5, 1.5)	&(5, 1.5)		&(5, .5)	&(5, 1)		&(5, 1.3)		&(5,.5) \\
12                       &(5, 4.3)		&(5, 3.5)	&(5, 3.5)		&(5, 4.5)	&(5, 4)		&(5, 3.7)		&(5,4.5) \\ \hline
13                       &(1.86, 2.5)	&(3, 2.5)	&(2.5, 2.9)	&(2, 2.5)	&(3, 2.5)	&(2.5, 3.5)	&(2,2.5) \\ \hline
\end{tabular}
\end{center}
\end{table}

These tiles are arranged as in Fig.~\ref{fig:building_plans} to obtain $P_u$ and $P_\ell$, via the union (sharing coinciding points along the shared edges) of translations of the tiles $A$, $B$, $D$, $E$, $\bar{E}$, and $90^\circ$ rotation and translations of tiles $C$ and $F$. Note that the lower row presents an asymmetry as the $D$ tile is only used once.

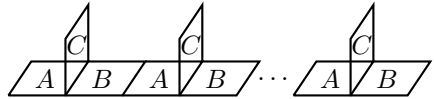
\begin{figure}[ht!]
  \caption{Pattern for assembling of the tiles.} \label{fig:building_plans}
  \begin{center}
    \begin{subfigure}[t]{0.5\textwidth}
      \begin{center}
        \begin{tikzpicture}[scale=.15]
          \draw[thick] (0,0) -- (5,0) -- (7,3) -- (2,3) -- (0,0);
          \draw[thick] (5,0) -- (10,0) -- (12,3) -- (7,3);
          \draw[thick] (10,0) -- (15,0) -- (17,3) -- (12,3);
          \draw[thick] (15,0) -- (20,0) -- (22,3) -- (17,3);
          \draw[thick] (5,0) -- (5,5) -- (7,8) -- (7,3);
          \draw[thick] (15,0) -- (15,5) -- (17,8) -- (17,3);

          \draw[thick] (25,0) -- (30,0) -- (32,3) -- (27,3) -- (25,0);
          \draw[thick] (30,0) -- (35,0) -- (37,3) -- (32,3);
          \draw[thick] (30,0) -- (30,5) -- (32,8) -- (32,3);
          \draw (28.2,1.5) node {$A$};
          \draw (33.2,1.5) node {$B$};
          \draw (31,4.5) node {$C$};

          \draw (23.2,1.5) node {$\hdots$};

          \draw (3.2,1.5) node {$A$};
          \draw (8.2,1.5) node {$B$};
          \draw (13.2,1.5) node {$A$};
          \draw (18.2,1.5) node {$B$};
          \draw (6,4.5) node {$C$};
          \draw (16,4.5) node {$C$};
        \end{tikzpicture}
      \end{center}
      \caption{For upper row $P_u$}
    \end{subfigure}
    \\\vspace{1em}
    \begin{subfigure}[t]{0.5\textwidth}
      \begin{center}
        \begin{tikzpicture}[scale=.15]
          \draw[thick] (0,0) -- (5,0) -- (7,3) -- (2,3) -- (0,0);
          \draw[thick] (5,0) -- (10,0) -- (12,3) -- (7,3);
          \draw[thick] (10,0) -- (15,0) -- (17,3) -- (12,3);
          \draw[thick] (15,0) -- (20,0) -- (22,3) -- (17,3);
          \draw[thick] (5,0) -- (5,5) -- (7,8) -- (7,3);
          \draw[thick] (15,0) -- (15,5) -- (17,8) -- (17,3);

          \draw[thick] (25,0) -- (30,0) -- (32,3) -- (27,3) -- (25,0);
          \draw[thick] (30,0) -- (35,0) -- (37,3) -- (32,3);
          \draw[thick] (30,0) -- (30,5) -- (32,8) -- (32,3);
          \draw (28.2,1.5) node {$\bar E$};
          \draw (33.2,1.5) node {$E$};
          \draw (31,4.5) node {$F$};

          \draw (23.2,1.5) node {$\hdots$};

          \draw (3.2,1.5) node {$D$};
          \draw (8.2,1.5) node {$E$};
          \draw (13.2,1.5) node {$\bar E$};
          \draw (18.2,1.5) node {$E$};
          \draw (6,4.5) node {$F$};
          \draw (16,4.5) node {$F$};
        \end{tikzpicture}
      \end{center}
      \caption{For lower row $P_\ell$}
    \end{subfigure}
  \end{center}
\end{figure}

Each tile has $13$ points: one at each corner, two on the interior of each edge, and one in the interior of the tile. Via the index in Table~\ref{tab:coordinates}, we have one-to-one correspondences between points of all tiles.  We define a vertex map $f: P_\ell \rightarrow P_u$ via this correspondence, as follows. Each point $p \in P_\ell$ is a point in a copy of tile $D$, $E$, $\bar{E}$, or $F$, as in Fig.~\ref{fig:building_plans}. Then, $p$ is mapped to $f(p) \in P_u$ defined to be the corresponding point with the same index in Table~\ref{tab:coordinates} in the corresponding copy of tile $A$, $B$, or $C$. Note that along shared edges, there may be ambiguities as to which tile $D$, $E$, $\bar{E}$, or $F$ a point $p\in P_\ell$ is coming from. However, it can be checked that the corresponding point $f(p)$ is well-defined.

We then choose five radius parameters $r_1, \hdots, r_5$ so that $x_i < r_i < y_i$  for each $i$, with the bounds $(x_i,y_i)$ described in Table~\ref{tab:range}.
The values of $x_i$ and $y_i$ are chosen such that for all $x_i < r, r' < y_i$, $V(P_u,r)=V(P_u,r')$ and $V(P_\ell,r)=V(P_\ell,r')$.
\begin{table}[htb]
\caption{Ranges for the choice of parameters.}\label{tab:range}
\begin{center}
\begin{tabular}{|l|c|c|c|c|c|}
\hline
Parameter & $r_1$ & $r_2$ & $r_3$ & $r_4$ & $r_5$ \\ \hline
$x_i$ & 1.06 & 1.21 & 1.6 & 1.8 & 2\\
$y_i$ & 1.12 & 1.25 & 1.665 & 1.805 & 2.015\\ \hline
\end{tabular}
\end{center}
\end{table}

We construct the diagram
\begin{equation}
  \label{eq:vr_realization}
  \begin{tikzcd}[ampersand replacement=\&]
    V(P_u,r_1)\rar \&
    V(P_u,r_2)\rar \&
    V(P_u,r_3)\rar \&
    V(P_u,r_4)\rar \&
    V(P_u,r_5)
    \\
    V(P_\ell,r_1)\rar \uar{f_1} \&
    V(P_\ell,r_2)\rar \uar{f_2} \&
    V(P_\ell,r_3)\rar \uar{f_3} \&
    V(P_\ell,r_4)\rar \uar{f_4} \&
    V(P_\ell,r_5) \uar{f_5}
  \end{tikzcd}
\end{equation}
where it can be checked that for $i=1,2,3,4,5$, the vertex map $f$ restricted to the vertices in $V(P_\ell,r_i)$, defines a simplicial map $f_i : V(P_\ell,r_i)\rightarrow V(P_u,r_i)$. After identification of corresponding vertices by $f$, these are subcomplex inclusions (of abstract simplicial complexes).
We note that the fact that we have subcomplex inclusions is only valid for those chosen radii, and not accross the whole range of possible radii. Then, Diagram~\ref{eq:vr_realization} has the same homology as $\mathbb{S}(d)$

Below, we provide illustrations of the Vietoris-Rips complexes restricted to the tiles, for the chosen radii.
\FloatBarrier
\begin{figure}[htb!]
  \caption{Tile A}
  \begin{center}
  \begin{tikzpicture}[scale=0.34]
    \Atile{0}{0}{0}{0}
    \Atile{7}{0}{0}{1}
    \Atile{14}{0}{0}{2}
    \Atile{21}{0}{0}{3}
    \Atile{28}{0}{0}{4}
    \Atile{35}{0}{0}{5}
  \end{tikzpicture}
\end{center}
\end{figure}

\begin{figure}[htb!]
  \caption{Tile B}
  \begin{center}
  \begin{tikzpicture}[scale=0.34]
    \Btile{0}{0}{0}{0}
    \Btile{7}{0}{0}{1}
    \Btile{14}{0}{0}{2}
    \Btile{21}{0}{0}{3}
    \Btile{28}{0}{0}{4}
    \Btile{35}{0}{0}{5}
  \end{tikzpicture}
\end{center}
\end{figure}

\begin{figure}[htb!]
  \caption{Tile C}
  \begin{center}
  \tdplotsetmaincoords{90}{270}
  \begin{tikzpicture}[scale=0.34, tdplot_main_coords]
    \Ctile{0}{35}{0}{0}
    \Ctile{0}{28}{0}{1}
    \Ctile{0}{21}{0}{2}
    \Ctile{0}{14}{0}{3}
    \Ctile{0}{7}{0}{4}
    \Ctile{0}{0}{0}{5}
  \end{tikzpicture}
\end{center}
\end{figure}

\begin{figure}[htb!]
  \caption{Tile D}
  \begin{center}
  \begin{tikzpicture}[scale=0.34]
    \Dtile{0}{0}{0}{0}
    \Dtile{7}{0}{0}{1}
    \Dtile{14}{0}{0}{2}
    \Dtile{21}{0}{0}{3}
    \Dtile{28}{0}{0}{4}
    \Dtile{35}{0}{0}{5}
  \end{tikzpicture}
\end{center}
\end{figure}

\begin{figure}[htb!]
  \caption{Tile E}
  \begin{center}
  \begin{tikzpicture}[scale=0.34]
    \Etile{0}{0}{0}{0}
    \Etile{7}{0}{0}{1}
    \Etile{14}{0}{0}{2}
    \Etile{21}{0}{0}{3}
    \Etile{28}{0}{0}{4}
    \Etile{35}{0}{0}{5}
  \end{tikzpicture}
\end{center}
\end{figure}

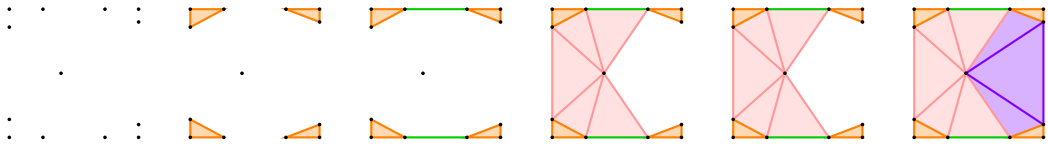
\begin{figure}[htb!]
  \caption{Tile E reversed}
  \begin{center}
  \begin{tikzpicture}[scale=0.34]
    \EtileR{0}{0}{0}{0}
    \EtileR{7}{0}{0}{1}
    \EtileR{14}{0}{0}{2}
    \EtileR{21}{0}{0}{3}
    \EtileR{28}{0}{0}{4}
    \EtileR{35}{0}{0}{5}
  \end{tikzpicture}
\end{center}
\end{figure}

\begin{figure}[htb!]
  \caption{Tile F}
  \begin{center}
  \tdplotsetmaincoords{90}{270}
  \begin{tikzpicture}[tdplot_main_coords,scale=0.34]
    \Ftile{0}{35}{0}{0}
    \Ftile{0}{28}{0}{1}
    \Ftile{0}{21}{0}{2}
    \Ftile{0}{14}{0}{3}
    \Ftile{0}{7}{0}{4}
    \Ftile{0}{0}{0}{5}
  \end{tikzpicture}
\end{center}
\end{figure}
\FloatBarrier

The resulting Diagram~\ref{eq:vr_realization}  with $d=2$ is displayed in Figure~\ref{fig:complete}.
For clarity of the illustration, we omit some extra edges and triangles that do not affect the homology.

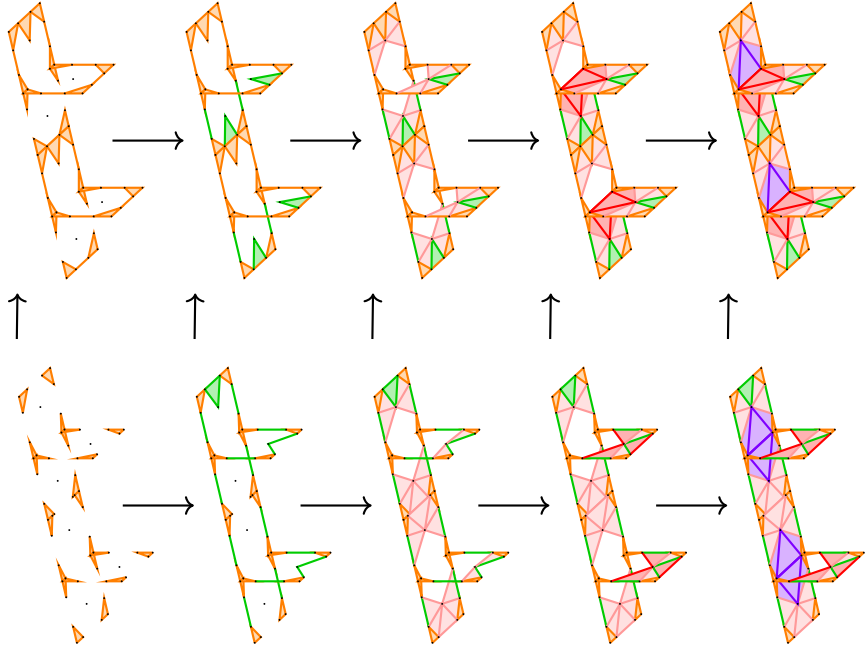
\begin{figure}[htb!]
  \caption{Complete realization, $d=2$ case}\label{fig:complete}
  \tdplotsetmaincoords{60}{-25}
  \centering
  \begin{tikzpicture}[tdplot_main_coords,scale=.18, rotate=270]
    \sandalup{2}{0}{0}{0}{1}
    \draw[thick,->] (10,2.5,5) -- (10,2.5,11);
    \sandalup{2}{0}{0}{15}{2}
    \draw[thick,->] (10,2.5,20) -- (10,2.5,26);
    \sandalup{2}{0}{0}{30}{3}
    \draw[thick,->] (10,2.5,35) -- (10,2.5,41);
    \sandalup{2}{0}{0}{45}{4}
    \draw[thick,->] (10,2.5,50) -- (10,2.5,56);
    \sandalup{2}{0}{0}{60}{5}

    \sandallow{2}{25}{-10}{0}{1}
    \draw[thick,->] (35,-7.5,5) -- (35,-7.5,11);
    \sandallow{2}{25}{-10}{15}{2}
    \draw[thick,->] (35,-7.5,20) -- (35,-7.5,26);
    \sandallow{2}{25}{-10}{30}{3}
    \draw[thick,->] (35,-7.5,35) -- (35,-7.5,41);
    \sandallow{2}{25}{-10}{45}{4}
    \draw[thick,->] (35,-7.5,50) -- (35,-7.5,56);
    \sandallow{2}{25}{-10}{60}{5}

    \draw[thick,->] (21,-8.5,0) -- (18,-7,0);
    \draw[thick,->] (21,-8.5,15) -- (18,-7,15);
    \draw[thick,->] (21,-8.5,30) -- (18,-7,30);
    \draw[thick,->] (21,-8.5,45) -- (18,-7,45);
    \draw[thick,->] (21,-8.5,60) -- (18,-7,60);
  \end{tikzpicture}
\end{figure}

Importantly, our construction does not rely on degeneracy.
With $P_\ell$ and $P_u$ fixed, we can freely choose the radius parameters $r_i$ within intervals $I_i=(x_i,y_i)$ of non-zero width where there are no changes in the filtration.
Dually, small perturbations of the point sets do not change the homology of the construction.
Formally we have,

\begin{lemma}
  Let $\rho$ be the minimum of the diameters of $I_i$, and fix each $r_i$ to be the center of $I_i$, for $i=1,\hdots,5$.
  Replacing each point of our input by a point within a ball of radius $\frac{\rho}{2}$ around it does not change the topology.
\end{lemma}
\begin{proof}
By construction, for any $i$, there are no edges of length $l$ such that $2r_i-\rho< l < 2r_i+\rho$.

We now replace every point of $P$ by a point located within distance $\frac{\rho}{2}$.
Note that the pairwise distance after this addition of noise are modified by at most $\rho$.
Therefore, the complexes for radii $r_i$ are unaffected as no pairwise distance can cross that threshold.
\end{proof}


\section{Commutative cube}

The commutative cube $C$ is defined to be the quiver
\[
  \begin{tikzcd}[row sep=1em, column sep=1em]
    & 3'
    \ar{rr}{}
    &
    & 4'
    \\
    3
    \ar{ur}{}
    \ar{rr}{}
    &
    &
    4
    \ar{ur}{}
    \\
    & 1'
    \ar{rr}
    \ar{uu}{}
    &
    & 2'
    \ar{uu}{}
    \\
    1
    \ar{uu}{}
    \ar{ur}{}
    \ar{rr}{}
    &
    & 2
    \ar{uu}{}
    \ar{ur}{}
  \end{tikzcd}
\]
bound by commutativity relations. Similar to Lemma~\ref{prop:isothm1}, it can be checked that
$\rep{C} \cong \arr{\rep{\CL_2(f)}}$.
Thus, we write a representation of $C$ as a morphism between representations of $\CL_2(f)$ by taking the morphism from the front face to the back face.

\subsection{Algebraic construction}

In $\rep(\CL_2(f))$, an analogue of Lemma~\ref{lemma:homdim} does not hold. In particular, the indecomposables $I_1$, $I_2$ in $\rep(\CL_2(f))$ given by:
\[
  I_1 :
  \begin{tikzcd}
    K \rar{1} & K \\
    0 \uar \rar & K \uar{1}
  \end{tikzcd}
  \text{ and }
  I_2 :
  \begin{tikzcd}
    K \rar & 0 \\
    K \uar{1} \rar{1} & K \uar
  \end{tikzcd}
\]
have $\dim\Hom(I_1,I_2) = 2$. The vector space $\Hom(I_1,I_2)$ can be given the basis $\{f_2,
f_3\}$, where $f_2$ is the identity on the lower right corner and zero
elsewhere, and $f_3$ is the identity on the upper left corner and zero
elsewhere.

The Auslander-Reiten quiver of $\CL_2(f)$, with chosen indecomposables encircled, is given by
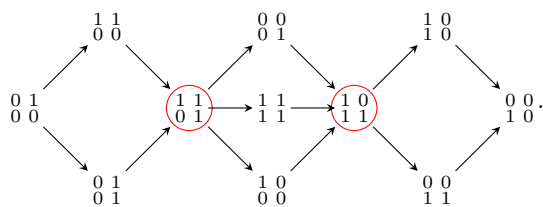
\[
  \begin{tikzpicture}[baseline=(A.center),
    every node/.append style={anchor=center,text depth=1ex,text
      height=2.2ex,text width=3.2ex,inner sep=0mm}]
    \matrix (A) [matrix of math nodes, column sep= 6 mm, row sep = 6mm, inner sep = 0mm,
    nodes in empty cells, scale = 1]
    {
      & \dimVecF{0,0,1,1} &                  & \dimVecF{0,1,0,0}      &                  & \dimVecF{1,0,1,0} &
      \\
      \dimVecF{0,0,0,1} &                  & \dimVecF{0,1,1,1} & \dimVecF{1,1,1,1}      & \dimVecF{1,1,1,0} &                  & \dimVecF{1,0,0,0}
      \\
      & \dimVecF{0,1,0,1} &                  & \dimVecF{0,0,1,0}      &                  & \dimVecF{1,1,0,0} &
      \\
    };

    \draw[red] (A-2-3) circle (3mm);
    \draw[red] (A-2-5) circle (3mm);

    \foreach \y in {4,5}{
      \pgfmathtruncatemacro{\x}{\y-1};
      \path [-stealth] (A-2-\x) edge (A-2-\y);
    }

    \foreach \y in {2,4,6}{
      \pgfmathtruncatemacro{\x}{\y-1};
      \path [-stealth] (A-2-\x) edge (A-1-\y);
      \path [-stealth] (A-2-\x) edge (A-3-\y);
    }

    \foreach \y in {2,4,6}{
      \pgfmathtruncatemacro{\x}{\y+1};
      \path [-stealth] (A-1-\y) edge (A-2-\x);
      \path [-stealth] (A-3-\y) edge (A-2-\x);
    }
  \end{tikzpicture}.
\]
Locating the $I_1$ and $I_2$ indecomposables and maps $f_2,f_3:I_1\rightarrow I_2$ employed in the construction,
we see that $f_2$ is the composition of the morphisms on the upper path, while $f_3$ is of those on the lower path from $I_1$ to $I_2$ in the Auslander-Reiten quiver above.

Intuitively, we see that this is related to representations of the \emph{Kronecker quiver}
$Q_2:
  \begin{tikzcd}
    1 \rar[shift left]\rar[shift right] & 2
  \end{tikzcd}$
by thinking about the two arrows as the linearly independent $f_2$, $f_3$.
This statement can be made precise by the following.
\begin{theorem}
  \label{th:cube}
  There is a fully faithful $K$-functor
  $\theta : \rep{Q_2} \rightarrow \rep{C}$
  that preserves indecomposability and isomorphism classes,
  where $\theta$ takes a representation
  $V:
    \begin{tikzcd}
      V_1 \rar[shift left]{g_1} \rar[shift right,swap]{g_2} & V_2
    \end{tikzcd}
  $
  to
  \[
    \theta(V):
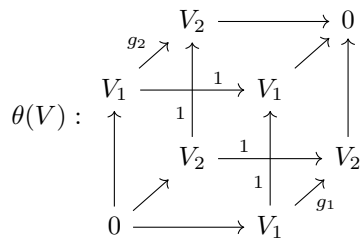
    \begin{tikzcd}[row sep=1em, column sep=1em]
      & V_2 \ar{rr}{} & & 0
      \\
      V_1 \ar{ur}{g_2} \ar{rr}[near end]{1} & & V_1 \ar{ur} &
      \\
      & V_2 \ar{rr}[near start]{1} \ar{uu}[near start]{1} & & V_2 \ar{uu}{}
      \\
      0 \ar{uu}{} \ar{ur}{} \ar{rr}{} & & V_1 \ar{uu}[near start]{1} \ar[swap]{ur}{g_1} &
    \end{tikzcd}
  \]
  and a morphism $\phi= (\phi_1,\phi_2):V\rightarrow W$
  to
  $
    (0,\phi_1,\phi_1,\phi_1,\phi_2,\phi_2,\phi_2,0) : \theta(V) \rightarrow \theta(W)
  $,
  where these maps are specified for the vertices in the order $1,\hdots,4,1',\hdots,4'$.
\end{theorem}
\begin{proof}
  That $\theta$ is a $K$-functor is easy to check.
  By the definition, $\theta(\phi) = 0$ implies that $\phi = 0$. Thus, $\theta$ is faithful.  To see that $\theta$ is full, let $V$ be as above, $W:
    \begin{tikzcd}
      W_1 \rar[shift left]{g'_1} \rar[shift right,swap]{g'_2} & W_2
    \end{tikzcd}
  $ and
  \[
    \alpha = (\alpha_1,\alpha_2,\alpha_3,\alpha_4,\alpha_{1'},\alpha_{2'},\alpha_{3'},\alpha_{4'}) : \theta(V) \rightarrow \theta(W)
  \]
  be a morphism. Then,
  $\alpha_1 : 0 \rightarrow 0$ and $\alpha_{4'}: 0\rightarrow 0$
  are zero maps by the forms of
  $\theta(V)$ and $\theta(W)$. The commutativity requirements for
  morphisms imply that $\alpha_2 = \alpha_3 =
  \alpha_4$ and $\alpha_{2'} = \alpha_{3'} = \alpha_{4'}$, and
  furthermore, $\alpha_{3'}g_2 = g'_2\alpha_3$ and $\alpha_{2'}g_1 =
  g'_1\alpha_2$. Thus, $(\alpha_{2}, \alpha_{2'})$ is a morphism from
  $V$ to $W$ such that $\theta((\alpha_{2}, \alpha_{2'})) = \alpha$.

  Let $V \in \rep{Q_2}$ be indecomposable. By part 2 of Lemma~\ref{lem:indec_endo}, $\End V$ is local. By the above result that $\theta$ is fully faithful, $\End \theta(V) \cong \End V$, and so $\End \theta(V)$ is local. Therefore, $\theta(V) \in \rep{C}$ is indecomposable by part 1 of Lemma~\ref{lem:indec_endo}.

  If $\alpha:\theta(V) \rightarrow \theta(W)$ is an isomorphism with
  inverse $\beta$ then $(\alpha_2,\alpha_{2'}):V\rightarrow W$ is an
  isomorphism with inverse $(\beta_2,\beta_{2'})$. Thus, if $\theta(V)
  \cong \theta(W)$, then $V\cong W$. This shows that $\theta$
  preserves isomorphism classes.
\end{proof}

The quiver $Q_2$ is representation infinite, and its indecomposable representations are well-known. See for example Proposition~1.6 of \cite{barotbook}. In particular,  one example of an infinite family is the regular indecomposables
$
  R_n(\lambda):
  \begin{tikzcd}
    K^n \rar[shift left]{I} \rar[shift right,swap]{J_n(\lambda)} & K^{n}
  \end{tikzcd}
$
for $n\geq 0$ and $\lambda\in K$. The following corollary is immediate from Theorem~\ref{th:cube}.
\begin{corollary}
  The commutative cube $C$ is representation infinite.
\end{corollary}

By the above arguments, $\theta(R_n(\lambda))$:
\[
  \theta(R_n(\lambda)):
  \begin{tikzcd}[row sep=1.2em]
    & K^{n} \ar{rr}{} & & 0
    \\
    K^n \ar{ur}{I} \ar{rr}[near end]{1} & & K^n \ar{ur}{} &
    \\
    & K^{n} \ar{rr}[near start]{1} \ar{uu}[near start]{1} & & K^{n} \ar{uu}{}
    \\
    0 \ar{uu}{} \ar{ur}{} \ar{rr}{} & & K^n \ar{uu}[near start]{1} \ar{ur}[swap]{J_n(\lambda)} &
  \end{tikzcd}
\]
are indecomposable and pairwise non-isomorphic for $n\geq 0$.

\subsection{Topological construction}

We give a topological realization for $\theta(R_n(0))$ in Fig.~\ref{fig:cube_realiz}. In the back face, we have $n$ half filled-in strips arranged side by side horizontally in an alternating pattern (Fig.~\ref{fig:cube_realiz} shows the case $n$ even). Using the lower left corner, we are able to flip the pattern. Coming from the front face, and with the given choice of basis, the induced map $H_1(\iota_2)$ is $J_n(0)$ while $H_1(\iota_3)$ is the identity $I$.

\begin{figure}[hbt]
  \caption{Topological realization of $\theta(R_n(0))$.} \label{fig:cube_realiz}
  \newcommand{\myscale}{0.4}

  \newcommand\defineboardpts{
    \coordinate (Ia) at (0,1);
    \coordinate (Ib) at (0,0);
    \coordinate (Ic) at (0,-1);

    \coordinate (IaSh) at (0,0.8);
    \coordinate (IbuSh) at (0,0.2);
    \coordinate (IbdSh) at (0,-0.2);
    \coordinate (IcSh) at (0,-0.8);
  }

  \newcommand{\shiftptpt}[2]{
    ($ #1 + (1*#2,0) $)
  }

  \newcommand{\drawuphk}[2]{
    \draw[#1] \shiftptpt{(Ib)}{#2} -- \shiftptpt{(Ia)}{#2} -- \shiftptpt{(Ia)}{#2+1} -- \shiftptpt{(Ib)}{#2+1};
  }

  \newcommand{\drawdwhk}[2]{
    \draw[#1] \shiftptpt{(Ib)}{#2} -- \shiftptpt{(Ic)}{#2} -- \shiftptpt{(Ic)}{#2+1} -- \shiftptpt{(Ib)}{#2+1};
  }

  \newcommand{\drawdwleftcee}[2]{
    \draw[#1] \shiftptpt{(Ib)}{#2+1} -- \shiftptpt{(Ib)}{#2} -- \shiftptpt{(Ic)}{#2} -- \shiftptpt{(Ic)}{#2+1} ;
  }
  \newcommand{\drawdwrightcee}[2]{
    \draw[#1] \shiftptpt{(Ib)}{#2} -- \shiftptpt{(Ib)}{#2+1} --  \shiftptpt{(Ic)}{#2+1} -- \shiftptpt{(Ic)}{#2} ;
  }

  \newcommand{\drawupleftcee}[2]{
    \draw[#1] \shiftptpt{(Ia)}{#2+1} -- \shiftptpt{(Ia)}{#2} -- \shiftptpt{(Ib)}{#2} -- \shiftptpt{(Ib)}{#2+1} ;
  }
  \newcommand{\drawuprightcee}[2]{
    \draw[#1] \shiftptpt{(Ia)}{#2} -- \shiftptpt{(Ia)}{#2+1} --  \shiftptpt{(Ib)}{#2+1} -- \shiftptpt{(Ib)}{#2} ;
  }

  \newcommand{\shrunkfirst}[2]{
    \draw[#1] \shiftptpt{(IbuSh)}{#2+0.2} -- \shiftptpt{(IbuSh)}{#2+0.8} -- \shiftptpt{(IaSh)}{#2+0.8} -- \shiftptpt{(IaSh)}{#2+0.2}--cycle;
  }

  \newcommand{\shrunkdwcycle}[2]{
    \draw[#1] \shiftptpt{(IbdSh)}{#2+0.2} -- \shiftptpt{(IbdSh)}{#2+1.8} -- \shiftptpt{(IcSh)}{#2+1.8} -- \shiftptpt{(IcSh)}{#2+0.2}--cycle;
  }

  \newcommand{\shrunkupcycle}[2]{
    \draw[#1] \shiftptpt{(IbuSh)}{#2+0.2} -- \shiftptpt{(IbuSh)}{#2+1.8} -- \shiftptpt{(IaSh)}{#2+1.8} -- \shiftptpt{(IaSh)}{#2+0.2}--cycle;
  }

  \newcommand{\shrunkdwleftcee}[2]{
    \draw[#1] \shiftptpt{(IbdSh)}{#2+1} -- \shiftptpt{(IbdSh)}{#2+0.2} -- \shiftptpt{(IcSh)}{#2+0.2} -- \shiftptpt{(IcSh)}{#2+1};
  }

  \newcommand{\shrunkdwrightcee}[2]{
    \draw[#1] \shiftptpt{(IbdSh)}{#2} -- \shiftptpt{(IbdSh)}{#2+0.8} -- \shiftptpt{(IcSh)}{#2+0.8} -- \shiftptpt{(IcSh)}{#2};
  }

  \centering
  \begin{tikzcd}[ampersand replacement = \&, row sep = 1em, column sep = 1em,nodes={inner sep = 0.5em}, arrows={hookrightarrow}]
    \& 
    \begin{tikzpicture}[scale=\myscale,baseline=(current bounding box.center),rotate=0]
      \defineboardpts
      \drawuphk{black}{0}
      \drawdwhk{black,fill}{0}
      \drawuphk{black,fill}{1}
      \drawdwhk{black}{1}
      \drawuphk{black}{2}
      \drawdwhk{black,fill}{2}
      \drawuphk{black,fill}{4}
      \drawdwhk{black}{4}
      \node at \shiftptpt{(Ib)}{3.5} {$\hdots$};

      \node [yshift=1em,xshift=0.75em] at \shiftptpt{(Ia)}{0}{$a_1$};
      \node [yshift=-1em,xshift=0.75em] at \shiftptpt{(Ic)}{1}{$a_2$};
      \node [yshift=1em,xshift=0.75em] at \shiftptpt{(Ia)}{2}{$a_3$};
      \node [yshift=-1em,xshift=0.75em] at \shiftptpt{(Ic)}{4}{$a_n$};
    \end{tikzpicture}
    \ar{rr}{}
    \&
    \& 
    \begin{tikzpicture}[scale=\myscale,baseline=(current bounding box.center),rotate=0]
      \defineboardpts
      \drawuphk{black,fill}{0}
      \drawdwhk{black,fill}{0}
      \drawuphk{black,fill}{1}
      \drawdwhk{black,fill}{1}
      \drawuphk{black,fill}{2}
      \drawdwhk{black,fill}{2}
      \drawuphk{black,fill}{4}
      \drawdwhk{black,fill}{4}
      \node at \shiftptpt{(Ib)}{3.5} {$\hdots$};
    \end{tikzpicture}
    \\  
    \begin{tikzpicture}[scale=\myscale,baseline=(current bounding box.center),rotate=0]
      \defineboardpts
      \drawuphk{black}{0}
      \drawdwhk{black,fill}{0}
      \drawuphk{black,fill}{1}
      \drawdwhk{black}{1}
      \drawuphk{black}{2}
      \drawdwhk{black,fill}{2}
      \drawuphk{black,fill}{4}
      \drawdwhk{black}{4}
      \node at \shiftptpt{(Ib)}{3.5} {$\hdots$};
    \end{tikzpicture}
    \ar{ur}{\iota_3}
    \ar{rr}{}
    \&
    \&
    \begin{tikzpicture}[scale=\myscale,baseline=(current bounding box.center),rotate=0]
      \defineboardpts
      \drawuphk{black}{0}
      \drawdwhk{black,fill}{0}
      \drawuphk{black,fill}{1}
      \drawdwhk{black}{1}
      \drawuphk{black}{2}
      \drawdwhk{black,fill}{2}
      \drawuphk{black,fill}{4}
      \drawdwhk{black}{4}
      \node at \shiftptpt{(Ib)}{3.5} {$\hdots$};
    \end{tikzpicture}
    \ar{ur}{}
    \\
    \& 
    \begin{tikzpicture}[scale=\myscale,baseline=(current bounding box.center),rotate=0]
      \defineboardpts
      \drawuphk{black}{0}
      \drawdwhk{black}{0}
      \drawuphk{black}{1}
      \drawdwhk{black}{1}
      \drawuphk{black}{2}
      \drawdwhk{black}{2}
      \drawuphk{black}{4}
      \drawdwhk{black}{4}
      \node at \shiftptpt{(Ib)}{3.5} {$\hdots$};

      \node [yshift=-1em,xshift=0.75em] at \shiftptpt{(Ic)}{0}{$a_1$};
      \node [yshift=-1em,xshift=0.75em] at \shiftptpt{(Ic)}{1}{$a_2$};
      \node [yshift=-1em,xshift=0.75em] at \shiftptpt{(Ic)}{2}{$a_3$};
      \node [yshift=-1em,xshift=0.75em] at \shiftptpt{(Ic)}{4}{$a_n$};
    \end{tikzpicture}
    \ar{rr}
    \ar[shorten <=0em, shorten >=0em]{uu}{}
    \&
    \& 
    \begin{tikzpicture}[scale=\myscale,baseline=(current bounding box.center),rotate=0]
      \defineboardpts
      \drawuphk{black,fill}{0}
      \drawdwhk{black}{0}
      \drawuphk{black}{1}
      \drawdwhk{black,fill}{1}
      \drawuphk{black,fill}{2}
      \drawdwhk{black}{2}
      \drawuphk{black}{4}
      \drawdwhk{black,fill}{4}
      \node at \shiftptpt{(Ib)}{3.5} {$\hdots$};
      \node [yshift=-1em,xshift=0.75em] at \shiftptpt{(Ic)}{0}{$a_1$};
      \node [yshift=1em,xshift=0.75em] at \shiftptpt{(Ia)}{1}{$a_2$};
      \node [yshift=-1em,xshift=0.75em] at \shiftptpt{(Ic)}{2}{$a_3$};
      \node [yshift=1em,xshift=0.75em] at \shiftptpt{(Ia)}{4}{$a_n$};
    \end{tikzpicture}
    \ar[shorten <=0em, shorten >=0em]{uu}{}
    \\
    \begin{tikzpicture}[scale=\myscale,baseline=(current bounding box.center)]
      \defineboardpts
      \draw[fill] \shiftptpt{(Ib)}{2} circle (2pt);
      \path[use as bounding box] \shiftptpt{(Ia)}{0} rectangle \shiftptpt{(Ic)}{4};
    \end{tikzpicture}
    \ar[shorten <=0em, shorten >=0em]{uu}{}
    \ar{ur}{}
    \ar{rr}{}
    \&
    \& 
    \begin{tikzpicture}[scale=\myscale,baseline=(current bounding box.center),rotate=0]
      \defineboardpts
      \drawuphk{black}{0}
      \drawdwleftcee{black}{0}

      \drawupleftcee{black}{1}
      \drawdwrightcee{black}{1}

      \drawuprightcee{black}{2}
      \drawdwleftcee{black}{2}

      \node at \shiftptpt{(Ib)}{3.5} {$\hdots$};

      \drawupleftcee{black}{4}
      \drawdwrightcee{black}{4}

      \shrunkfirst{blue}{0}
      \shrunkdwcycle{blue}{0}
      \shrunkupcycle{blue}{1}
      \shrunkdwrightcee{blue}{4}
      \shrunkdwleftcee{blue}{2}

      \node [yshift=1em,xshift=0.75em,blue] at \shiftptpt{(Ia)}{0}{$z_1$};
      \node [yshift=-1em,xshift=0.75em,blue] at \shiftptpt{(Ic)}{1}{$z_2$};
      \node [yshift=1em,xshift=0.75em,blue] at \shiftptpt{(Ia)}{2}{$z_3$};
      \node [yshift=-1em,xshift=0.75em,blue] at \shiftptpt{(Ic)}{4}{$z_n$};

    \end{tikzpicture}
    \ar[shorten <=0em, shorten >=0em]{uu}{}
    \ar[swap]{ur}{\iota_2}
  \end{tikzcd}
\end{figure}
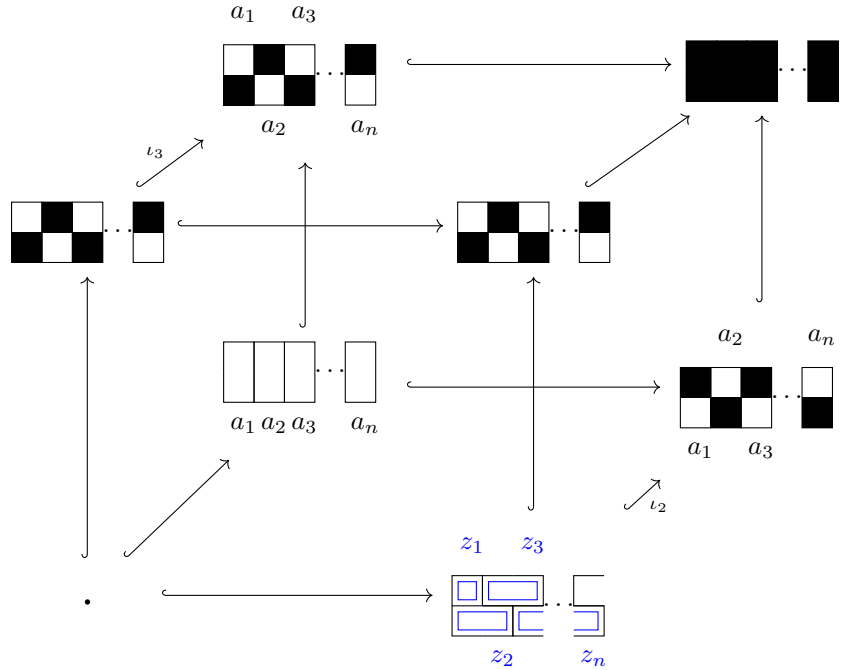


\section{Commutative grid $3\times 3$}

\subsection{Algebraic construction}

The Euclidean quiver of type $\tilde{D}_4$ is representation infinite for any orientation of the arrows.
\[
  \tilde{D}_4:
  \begin{tikzcd}[row sep=1em, column sep=1.2em]
     & 1 & \\
    2 \rar[dash] & 0 \uar[dash] \dar[dash] & 4 \lar[dash] \\
     & 3 &
  \end{tikzcd}
\]
We build infinite family of indecomposables for the $\tilde{D}_4$-type quiver with appropriate orientation and complete it to be indecomposable representations of a $3\times 3$ grid.
Up to symmetries, there are six different orientations of the $3\times 3$ grid.
We classify them according to the resulting orientation of the central $\tilde{D}_4$, shown in Fig.~\ref{fig:33config}.

\newcommand{\tgrid}[4]{
  \begin{tikzcd}[ampersand replacement=\&, row sep=1.2em, column sep = 1.2em]
    \bullet \rar[#1]{} \&
    \bullet \rar[#2]{} \&
    \bullet
    \\
    \bullet \rar[#1,blue]{} \uar[#4]{}\&
    \bullet \rar[#2,blue]{} \uar[#4,blue]{} \&
    \bullet \uar[#4]{}
    \\
    \bullet \rar[#1]{} \uar[#3]{} \&
    \bullet \rar[#2]{} \uar[#3,blue]{} \&
    \bullet \uar[#3]{}
  \end{tikzcd}
}


\begin{figure}[h]
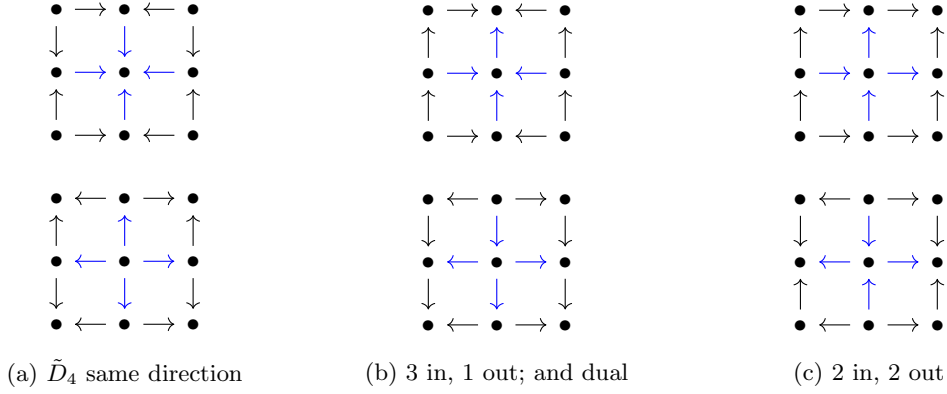

  \caption{Configurations for the $3\times 3$ grid.}
    \label{fig:33config}
  \begin{subfigure}{0.30\textwidth}
    \[
      \begin{array}{c}
        \tgrid{rightarrow}{leftarrow}{rightarrow}{leftarrow} \\ \\
        \tgrid{leftarrow}{rightarrow}{leftarrow}{rightarrow}
      \end{array}
    \]
    \caption{$\tilde{D}_4$ same direction}
  \end{subfigure}
  \begin{subfigure}{0.30\textwidth}
    \[
      \begin{array}{c}
    \tgrid{rightarrow}{leftarrow}{rightarrow}{rightarrow} \\ \\
        \tgrid{leftarrow}{rightarrow}{leftarrow}{leftarrow}
      \end{array}
    \]
    \caption{3 in, 1 out; and dual}
  \end{subfigure}
  \begin{subfigure}{0.30\textwidth}
    \[
      \begin{array}{c}
      \tgrid{rightarrow}{rightarrow}{rightarrow}{rightarrow} \\ \\
        \tgrid{leftarrow}{rightarrow}{rightarrow}{leftarrow}
      \end{array}
    \]
    \caption{2 in, 2 out}
  \end{subfigure}
\end{figure}



We fix the vector spaces to be $K^{2d}$ at the central vertex and $K^d$ elsewhere.
Note that at least two arrows of $\tilde{D}_4$ will be pointing in the same direction relative to the central vertex.
On two of these arrows, we assign matrices $\smat{I&0}$ and $\smat{0&I}$ or $\smat{I\\0}$ and $\smat{0\\I}$ depending on their orientation.
The remaining two arrows are assigned $\smat{I&J_d(\lambda)}$ or its transpose, and $\smat{I&I}$ or its transpose.
For example, with three arrows pointing in and one out, we can have:
\[
  \begin{tikzcd}[ampersand replacement=\&]
    \& K^d \& \\
    K^d \rar{\smat{I\\0}} \&
    K^{2d} \uar[swap]{\smat{I&J_d(\lambda)}}
    \& K^d \lar{\smat{I\\I}} \\
    \& K^d \uar{\smat{0\\I}} \&
  \end{tikzcd}.
\]
The proof of indecomposability is again by computation of the endomorphism ring.
Let $f=(f_0,\hdots, f_4)$ be an endomorphism.
In matrix form, $f_0 = \smat{A & B \\ C & D}:
K^{2d}\rightarrow K^{2d}$, where $0$ is the central vertex.
Without loss of generality, we assume that the pair of arrows pointing in the same direction and assigned matrices $\smat{I\\0}$ and $\smat{0\\I}$ (or $\smat{I&0}$ and $\smat{0&I}$), start from (or point towards) vertices $1$ and $2$, respectively.
From commutativity requirement for endomorphisms, $f_0 =\smat{A & B \\ C & D} = \smat{f_1 & 0 \\ 0 & f_2}$.
Suppose that the arrow assigned $\smat{I&I}$ (or its transpose) points to (or starts from) vertex $3$.
The commutativity requirement with $f_3$ then requires $f_1=f_3=f_2$.
Final commutativity requirement then forces $f_1=f_2=f_3=f_4$ and $f_1J_d(\lambda) = J_d(\lambda) f_1$.
Thus, the endomorphism ring is local.

Completing the example into the $3\times 3$ grid is easy.
In the squares where the representation $\tilde{D}_4$ provides a nonzero composition of maps, we use that composition on one arrow and the identity on the other arrow.
Otherwise we simply use a $0$ vector space and $0$ maps.
\[
  \begin{tikzcd}[ampersand replacement=\&]
    K^d \rar{I} \&
    K^d \&
    K^d \lar[swap]{I} \\
    K^d
    \rar{\smat{I\\0}}
    \uar{\smat{I&J_d(\lambda)}\smat{I\\0}=I}
    \&
    K^{2d} \uar[swap]{\smat{I&J_d(\lambda)}}
    \& K^d \lar{\smat{I\\I}}
    \uar[swap]{J_d(\lambda+1) = \smat{I&J_d(\lambda)}\smat{I\\I}}
    \\
    0 \uar \rar \& K^d \uar{\smat{0\\I}} \& 0 \lar\uar
  \end{tikzcd}
\]
This algebraic construction can then be realized through a diagram of topological spaces and inclusions using our previous \emph{sandal} construction. We do this in the next subsection.


\subsection{Topological Construction}\label{sec:3x3realization}

This subsection is split into two parts. First, we tackle the case where we have at least one arrow pointing towards the central vertex. Then we solve the remaining case where all arrows are outbound.

\subsubsection{At least one inbound arrow}

The space at the central vertex is $K^{2d}$.
We realize it using our sandal with $d$ straps and the basis used in Figure~\ref{fig:sandal_cycles} (second space in the top row).
As we have at least one arrow pointing to this space, we use again the structure from the second column of Figure~\ref{fig:sandal_cycles}.
We end up with the following inclusion map
{ 
  \input{drawing_macros}
\newcommand{\myscale}{0.65}

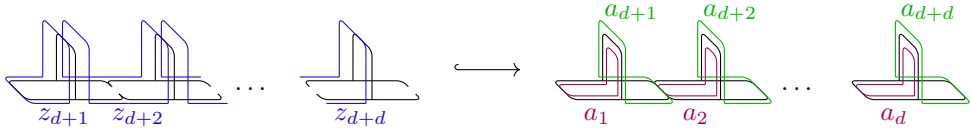
\[
\begin{tikzcd}[ampersand replacement = \&, row sep = 2.5em, column sep = 2.5em,nodes={inner sep = 1.2em}, arrows={hookrightarrow}]
  \begin{tikzpicture}[scale=\myscale,baseline=(current bounding box.center),rotate=90]
    \definesandalpts
    \drawuphk{black}{3}
    \drawstrp{black}{3}
    \drawdwhk{black}{3}
    \drawuphk{black}{2}
    \drawstrp{black}{2}
    \drawdwhk{black}{2}
    \drawuphk{black}{0}
    \drawstrp{black}{0}
    \drawdwhk{black}{0}
    \node[xshift=1em] at \shiftptno{(Sb)}{1.5}{.1} {$\hdots$};
    \path[use as bounding box] \shiftpt{(Id)}{3} rectangle \shiftpt{(Ib)}{-1};

    \node [coordinate] (Ina) at \shiftpt{(Ia)}{-1.2} {};
    \node [coordinate] (Inb) at \shiftpt{(Ib)}{-0.8} {};

    \node [coordinate] (I0a) at \shiftpt{(Ia)}{-0.2} {};
    \node [coordinate] (I0b) at \shiftpt{(Ib)}{0.2} {};

    \node [coordinate] (I1a) at \shiftpt{(Ia)}{0.8} {};
    \node [coordinate] (I1b) at \shiftpt{(Ib)}{1.2} {};

    \node [coordinate] (I2a) at \shiftpt{(Ia)}{1.8} {};
    \node [coordinate] (I2b) at \shiftpt{(Ib)}{2.2} {};

    \expanduphk{LowerRowColor1}{3}{3.1}
    \expandstrp{LowerRowColor1}{3.1}

    \expandstrp{LowerRowColor1}{2.9}
    \expanddwlegs{LowerRowColor1}{3}{2.9}
    \expanduplegs{LowerRowColor1}{2}{2.1}
    \expandstrp{LowerRowColor1}{2.1}

    \expandstrp{LowerRowColor1}{1.9}
    \expanddwlegs{LowerRowColor1}{2}{1.9}

    \expanduplegs{LowerRowColor1}{0}{0.1}
    \expandstrp{LowerRowColor1}{0.1}
    \draw[fill, white] ($ (I2a)!0.33!(I2b) $) rectangle ($ (I2b)!0.33!(I2a) $);
    \draw[fill, white] ($ (I1a)!0.33!(I1b) $) rectangle ($ (I1b)!0.33!(I1a) $);
    \draw[fill, white] ($ (I0a)!0.33!(I0b) $) rectangle ($ (I0b)!0.33!(I0a) $);
    \draw[fill, white] ($ (Ina)!0.33!(Inb) $) rectangle ($ (Inb)!0.33!(Ina) $);

    \node[anchor=west, LowerRowColor1] at \shiftptno{(IbEx)}{3.25}{-.1} {$z_{d+1}$};
    \node[anchor=west, LowerRowColor1] at \shiftptno{(IbEx)}{2.5}{-.1} {$z_{d+2}$};
    \node[anchor=west, LowerRowColor1] at \shiftptno{(IbEx)}{.25}{-.1} {$z_{d+d}$};
  \end{tikzpicture}
  \rar\&
   \begin{tikzpicture}[scale=\myscale,baseline=(current bounding box.center),rotate=90]
    \definesandalpts
    \drawuphk{black}{3}
    \drawstrp{black}{3}
    \drawdwhk{black}{3}
    \drawuphk{black}{2}
    \drawstrp{black}{2}
    \drawdwhk{black}{2}
    \drawuphk{black}{0}
    \drawstrp{black}{0}
    \drawdwhk{black}{0}
    \node[xshift=1em] at \shiftptno{(Sb)}{1.5}{.1} {$\hdots$};
    \path[use as bounding box] \shiftpt{(Id)}{3} rectangle \shiftpt{(Ib)}{-1};

    \shrunkuphk{UpperRowColor1}{3}
    \shrunkstrp{UpperRowColor1}{3}
    \shrunkuphk{UpperRowColor1}{2}
    \shrunkstrp{UpperRowColor1}{2}
    \shrunkuphk{UpperRowColor1}{0}
    \shrunkstrp{UpperRowColor1}{0}
    \expandstrp{UpperRowColor2}{3}
    \expanddwhk{UpperRowColor2}{3}{3}
    \expandstrp{UpperRowColor2}{2}
    \expanddwhk{UpperRowColor2}{2}{2}
    \expandstrp{UpperRowColor2}{0}
    \expanddwhk{UpperRowColor2}{0}{0}

    \node[anchor=west, UpperRowColor1] at \shiftptno{(IbEx)}{3.25}{-.1} {$a_1$};
    \node[anchor=west, UpperRowColor1] at \shiftptno{(IbEx)}{2.25}{-.1} {$a_2$};
    \node[anchor=west, UpperRowColor1] at \shiftptno{(IbEx)}{.25}{-.1} {$a_d$};
    \node[anchor=north, UpperRowColor2] at \shiftptno{(SdSh)}{2.75}{.75} {$a_{d+1}$};
    \node[anchor=north, UpperRowColor2] at \shiftptno{(SdSh)}{1.75}{.75} {$a_{d+2}$};
    \node[anchor=north, UpperRowColor2] at \shiftptno{(SdSh)}{-.25}{.75} {$a_{d+d}$};
  \end{tikzpicture}
\end{tikzcd}
\]
}
whose induced homology map is $\smat{I\\J_d(0)}$ with the indicated choice of bases.

We continue the construction by adding the matrix containing two $I$ blocks.
If the corresponding arrow points into the central vertex, we use the following inclusion,
where we start with the soles without the straps:
{  
  \input{drawing_macros}
\newcommand{\myscale}{0.65}

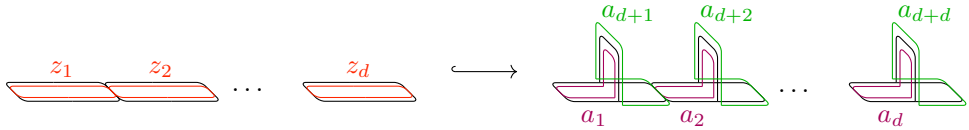
\[
\begin{tikzcd}[ampersand replacement = \&, row sep = 2.5em, column sep = 2.5em,nodes={inner sep = 1.2em}, arrows={hookrightarrow}]
  \begin{tikzpicture}[scale=\myscale,baseline=(current bounding box.center),rotate=90]
    \definesandalpts
    \drawuphk{black}{3}
    \drawdwhk{black}{3}
    \drawuphk{black}{2}
    \drawdwhk{black}{2}
    \drawuphk{black}{0}
    \drawdwhk{black}{0}
    \node[xshift=1em] at \shiftptno{(Sb)}{1.5}{.1} {$\hdots$};
    \path[use as bounding box] \shiftpt{(Id)}{3} rectangle \shiftpt{(Ib)}{-1};

    \shrunkuphk{LowerRowColor2}{3}
    \shrunkdwhk{LowerRowColor2}{3}
    \shrunkuphk{LowerRowColor2}{2}
    \shrunkdwhk{LowerRowColor2}{2}
    \shrunkuphk{LowerRowColor2}{0}
    \shrunkdwhk{LowerRowColor2}{0}

    \node[anchor=east, LowerRowColor2] at \shiftptno{(SaSh)}{2.5}{.2} {$z_{1}$};
    \node[anchor=east, LowerRowColor2] at \shiftptno{(SaSh)}{1.5}{.2} {$z_{2}$};
    \node[anchor=east, LowerRowColor2] at \shiftptno{(SaSh)}{-.5}{.2} {$z_{d}$};
  \end{tikzpicture}
  \rar\&
   \begin{tikzpicture}[scale=\myscale,baseline=(current bounding box.center),rotate=90]
    \definesandalpts
    \drawuphk{black}{3}
    \drawstrp{black}{3}
    \drawdwhk{black}{3}
    \drawuphk{black}{2}
    \drawstrp{black}{2}
    \drawdwhk{black}{2}
    \drawuphk{black}{0}
    \drawstrp{black}{0}
    \drawdwhk{black}{0}
    \node[xshift=1em] at \shiftptno{(Sb)}{1.5}{.1} {$\hdots$};
    \path[use as bounding box] \shiftpt{(Id)}{3} rectangle \shiftpt{(Ib)}{-1};

    \shrunkuphk{UpperRowColor1}{3}
    \shrunkstrp{UpperRowColor1}{3}
    \shrunkuphk{UpperRowColor1}{2}
    \shrunkstrp{UpperRowColor1}{2}
    \shrunkuphk{UpperRowColor1}{0}
    \shrunkstrp{UpperRowColor1}{0}
    \expandstrp{UpperRowColor2}{3}
    \expanddwhk{UpperRowColor2}{3}{3}
    \expandstrp{UpperRowColor2}{2}
    \expanddwhk{UpperRowColor2}{2}{2}
    \expandstrp{UpperRowColor2}{0}
    \expanddwhk{UpperRowColor2}{0}{0}

    \node[anchor=west, UpperRowColor1] at \shiftptno{(IbEx)}{3.25}{-.1} {$a_1$};
    \node[anchor=west, UpperRowColor1] at \shiftptno{(IbEx)}{2.25}{-.1} {$a_2$};
    \node[anchor=west, UpperRowColor1] at \shiftptno{(IbEx)}{.25}{-.1} {$a_d$};
    \node[anchor=north, UpperRowColor2] at \shiftptno{(SdSh)}{2.75}{.75} {$a_{d+1}$};
    \node[anchor=north, UpperRowColor2] at \shiftptno{(SdSh)}{1.75}{.75} {$a_{d+2}$};
    \node[anchor=north, UpperRowColor2] at \shiftptno{(SdSh)}{-.25}{.75} {$a_{d+d}$};
  \end{tikzpicture}
\end{tikzcd}
\]
}
with induced map $\smat{I\\I}$.
With an outbound arrow, we fill in the sole of the sandal:
{
  \input{drawing_macros}
\newcommand{\myscale}{0.65}

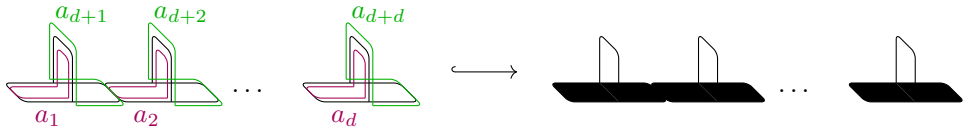
\[
\begin{tikzcd}[ampersand replacement = \&, row sep = 2.5em, column sep = 2.5em,nodes={inner sep = 1.2em}, arrows={hookrightarrow}]
   \begin{tikzpicture}[scale=\myscale,baseline=(current bounding box.center),rotate=90]
    \definesandalpts
    \drawuphk{black}{3}
    \drawstrp{black}{3}
    \drawdwhk{black}{3}
    \drawuphk{black}{2}
    \drawstrp{black}{2}
    \drawdwhk{black}{2}
    \drawuphk{black}{0}
    \drawstrp{black}{0}
    \drawdwhk{black}{0}
    \node[xshift=1em] at \shiftptno{(Sb)}{1.5}{.1} {$\hdots$};
    \path[use as bounding box] \shiftpt{(Id)}{3} rectangle \shiftpt{(Ib)}{-1};

    \shrunkuphk{UpperRowColor1}{3}
    \shrunkstrp{UpperRowColor1}{3}
    \shrunkuphk{UpperRowColor1}{2}
    \shrunkstrp{UpperRowColor1}{2}
    \shrunkuphk{UpperRowColor1}{0}
    \shrunkstrp{UpperRowColor1}{0}
    \expandstrp{UpperRowColor2}{3}
    \expanddwhk{UpperRowColor2}{3}{3}
    \expandstrp{UpperRowColor2}{2}
    \expanddwhk{UpperRowColor2}{2}{2}
    \expandstrp{UpperRowColor2}{0}
    \expanddwhk{UpperRowColor2}{0}{0}

    \node[anchor=west, UpperRowColor1] at \shiftptno{(IbEx)}{3.25}{-.1} {$a_1$};
    \node[anchor=west, UpperRowColor1] at \shiftptno{(IbEx)}{2.25}{-.1} {$a_2$};
    \node[anchor=west, UpperRowColor1] at \shiftptno{(IbEx)}{.25}{-.1} {$a_d$};
    \node[anchor=north, UpperRowColor2] at \shiftptno{(SdSh)}{2.75}{.75} {$a_{d+1}$};
    \node[anchor=north, UpperRowColor2] at \shiftptno{(SdSh)}{1.75}{.75} {$a_{d+2}$};
    \node[anchor=north, UpperRowColor2] at \shiftptno{(SdSh)}{-.25}{.75} {$a_{d+d}$};
  \end{tikzpicture}
   \rar\&
   \begin{tikzpicture}[scale=\myscale,baseline=(current bounding box.center),rotate=90]
    \definesandalpts
    \drawuphk{black,fill}{3}
    \drawstrp{black}{3}
    \drawdwhk{black,fill}{3}
    \drawuphk{black,fill}{2}
    \drawstrp{black}{2}
    \drawdwhk{black,fill}{2}
    \drawuphk{black,fill}{0}
    \drawstrp{black}{0}
    \drawdwhk{black,fill}{0}
    \node[xshift=1em] at \shiftptno{(Sb)}{1.5}{.1} {$\hdots$};
    \path[use as bounding box] \shiftpt{(Id)}{3} rectangle \shiftpt{(Ib)}{-1};

  \end{tikzpicture}
\end{tikzcd}
\]
}
with induced map $\smat{I&I}$.

In order to complete our construction, we add simple inclusion maps as needed, maintaining as many identity maps as possible. This provide us with the following diagrams for the five possible cases.
First we give the diagram of the topological spaces together with representatives for the homology bases.
Second, we also provide the induced indecomposable persistence module, obtained by applying $H_1(\cdot)$.

In what follows, we list up all of the topological realizations, up to rotation of diagram. To save space, we draw only copy number $1$ and $d$ of the sandal piece (and put ``$\hdots$'' in between) for the entries in the $3\times 3$ grid where the topological space and choice of basis is the obvious one. Otherwise, we proceed with our usual illustration of pieces $1$, $2$ and $d$.

\paragraph{3 In 1 Out} With three inbound and one outbound arrows to the central vertex, the diagram:
{
  \input{drawing_macros}
\newcommand{\myscale}{.5}
\[
\begin{tikzcd}[ampersand replacement = \&, row sep = 2em, column sep = 2.5em,nodes={inner sep = 1.2em}, arrows={hookrightarrow}]
  \begin{tikzpicture}[scale=\myscale,baseline=(current bounding box.center),rotate=90]
    \definesandalpts    
    \drawuphk{black}{2}
    \drawstrp{black}{2}
    \drawuphk{black}{0}
    \drawstrp{black}{0}
    \node[xshift=0em] at \shiftpt{(Sb)}{1} {$\cdots$};
    \path[use as bounding box] \shiftpt{(Id)}{2} rectangle \shiftpt{(Ib)}{-1};
  \end{tikzpicture}
\rar \&
   \begin{tikzpicture}[scale=\myscale,baseline=(current bounding box.center),rotate=90]
    \definesandalpts
    \drawuphk{black,fill}{3}
    \drawstrp{black}{3}
    \drawdwhk{black,fill}{3}
    \drawuphk{black,fill}{2}
    \drawstrp{black}{2}
    \drawdwhk{black,fill}{2}
    \drawuphk{black,fill}{0}
    \drawstrp{black}{0}
    \drawdwhk{black,fill}{0}
    \node[xshift=0em] at \shiftpt{(Sb)}{1} {$\cdots$};
    \path[use as bounding box] \shiftpt{(Id)}{3} rectangle \shiftpt{(Ib)}{-1};
  \end{tikzpicture}
   \rar[hookleftarrow] \& 
  \begin{tikzpicture}[scale=\myscale,baseline=(current bounding box.center),rotate=90]
    \definesandalpts    
    \drawstrp{black}{2}
    \drawdwhk{black}{2}
    \drawstrp{black}{0}
    \drawdwhk{black}{0}
    \node[xshift=0em] at \shiftpt{(Sb)}{1} {$\cdots$};
    \path[use as bounding box] \shiftpt{(Id)}{2} rectangle \shiftpt{(Ib)}{-1};
    \end{tikzpicture}
  \\
  \begin{tikzpicture}[scale=\myscale,baseline=(current bounding box.center),rotate=90]
    \definesandalpts    
    \drawuphk{black}{2}
    \drawstrp{black}{2}
    \drawuphk{black}{0}
    \drawstrp{black}{0}
    \node[xshift=0em] at \shiftpt{(Sb)}{1} {$\cdots$};
    \path[use as bounding box] \shiftpt{(Id)}{2} rectangle \shiftpt{(Ib)}{-1};
  \end{tikzpicture}
\rar\uar\&
\begin{tikzpicture}[scale=\myscale,baseline=(current bounding box.center),rotate=90]
    \definesandalpts
    \drawuphk{black}{3}
    \drawstrp{black}{3}
    \drawdwhk{black}{3}
    \drawuphk{black}{2}
    \drawstrp{black}{2}
    \drawdwhk{black}{2}
    \drawuphk{black}{0}
    \drawstrp{black}{0}
    \drawdwhk{black}{0}
    \node[xshift=0em] at \shiftpt{(Sb)}{1} {$\cdots$};
    \path[use as bounding box] \shiftpt{(Id)}{3} rectangle \shiftpt{(Ib)}{-1};

    \shrunkuphk{UpperRowColor1}{3}
    \shrunkstrp{UpperRowColor1}{3}
    \shrunkuphk{UpperRowColor1}{2}
    \shrunkstrp{UpperRowColor1}{2}
    \shrunkuphk{UpperRowColor1}{0}
    \shrunkstrp{UpperRowColor1}{0}
    \expandstrp{UpperRowColor2}{3}
    \expanddwhk{UpperRowColor2}{3}{3}
    \expandstrp{UpperRowColor2}{2}
    \expanddwhk{UpperRowColor2}{2}{2}
    \expandstrp{UpperRowColor2}{0}
    \expanddwhk{UpperRowColor2}{0}{0}
    
    \node[anchor=west, UpperRowColor1] at \shiftptno{(IbEx)}{3.25}{-.1} {$a_1$};
    \node[anchor=west, UpperRowColor1] at \shiftptno{(IbEx)}{2.25}{-.1} {$a_2$};
    \node[anchor=west, UpperRowColor1] at \shiftptno{(IbEx)}{.25}{-.1} {$a_d$};
    \node[anchor=north, UpperRowColor2] at \shiftptno{(SdSh)}{2.75}{.9} {$a_{d+1}$};
    \node[anchor=north, UpperRowColor2] at \shiftptno{(SdSh)}{1.75}{.9} {$a_{d+2}$};
    \node[anchor=north, UpperRowColor2] at \shiftptno{(SdSh)}{-.25}{.9} {$a_{d+d}$};    
  \end{tikzpicture}
  \uar \rar[hookleftarrow]  \&
  \begin{tikzpicture}[scale=\myscale,baseline=(current bounding box.center),rotate=90]
  \definesandalpts    
    \drawstrp{black}{2}
    \drawdwhk{black}{2}
    \drawstrp{black}{0}
    \drawdwhk{black}{0}
    \node[xshift=0em] at \shiftpt{(Sb)}{1} {$\cdots$};
    \path[use as bounding box] \shiftpt{(Id)}{2} rectangle \shiftpt{(Ib)}{-1};
    \end{tikzpicture}
    \uar
  \\
  \begin{tikzpicture}[scale=\myscale,baseline=(current bounding box.center),rotate=90]
    \definesandalpts
    \draw[fill] \shiftpt{(Sa)}{1} circle (2pt);
    \path[use as bounding box] \shiftpt{(Id)}{2} rectangle \shiftpt{(Ib)}{-1};
  \end{tikzpicture}
  \rar\uar\&
  \begin{tikzpicture}[scale=\myscale,baseline=(current bounding box.center),rotate=90]
    \definesandalpts
    \drawuphk{black}{3}
    \drawstrp{black}{3}
    \drawdwhk{black}{3}
    \drawuphk{black}{2}
    \drawstrp{black}{2}
    \drawdwhk{black}{2}
    \drawuphk{black}{0}
    \drawstrp{black}{0}
    \drawdwhk{black}{0}
    \node[xshift=0em] at \shiftpt{(Sb)}{1} {$\cdots$};
    \path[use as bounding box] \shiftpt{(Id)}{3} rectangle \shiftpt{(Ib)}{-1};

    \node [coordinate] (Ina) at \shiftpt{(Ia)}{-1.2} {};
    \node [coordinate] (Inb) at \shiftpt{(Ib)}{-0.8} {};

    \node [coordinate] (I0a) at \shiftpt{(Ia)}{-0.2} {};
    \node [coordinate] (I0b) at \shiftpt{(Ib)}{0.2} {};

    \node [coordinate] (I1a) at \shiftpt{(Ia)}{0.8} {};
    \node [coordinate] (I1b) at \shiftpt{(Ib)}{1.2} {};

    \node [coordinate] (I2a) at \shiftpt{(Ia)}{1.8} {};
    \node [coordinate] (I2b) at \shiftpt{(Ib)}{2.2} {};

    \expanduphk{LowerRowColor1}{3}{3.1}
    \expandstrp{LowerRowColor1}{3.1}

    \expandstrp{LowerRowColor1}{2.9}
    \expanddwlegs{LowerRowColor1}{3}{2.9}
    \expanduplegs{LowerRowColor1}{2}{2.1}
    \expandstrp{LowerRowColor1}{2.1}

    \expandstrp{LowerRowColor1}{1.9}
    \expanddwlegs{LowerRowColor1}{2}{1.9}

    \expanduplegs{LowerRowColor1}{0}{0.1}
    \expandstrp{LowerRowColor1}{0.1}
    \draw[fill, white] ($ (I2a)!0.35!(I2b) $) rectangle ($ (I2b)!0.35!(I2a) $);
    \draw[fill, white] ($ (I1a)!0.35!(I1b) $) rectangle ($ (I1b)!0.35!(I1a) $);
    \draw[fill, white] ($ (I0a)!0.35!(I0b) $) rectangle ($ (I0b)!0.35!(I0a) $);
    \draw[fill, white] ($ (Ina)!0.35!(Inb) $) rectangle ($ (Inb)!0.35!(Ina) $);

    \node[anchor=west, LowerRowColor1] at \shiftptno{(IbEx)}{3.65}{-.1} {$z_{d+1}$};
    \node[anchor=west, LowerRowColor1] at \shiftptno{(IbEx)}{2.75}{-.1} {$z_{d+2}$};
    \node[anchor=west, LowerRowColor1] at \shiftptno{(IbEx)}{.5}{-.1} {$z_{d+d}$};
  \end{tikzpicture}
  \uar \rar[hookleftarrow]  \& \uar
  \begin{tikzpicture}[scale=\myscale,baseline=(current bounding box.center),rotate=90]
    \definesandalpts
    \draw[fill] \shiftpt{(Sa)}{1} circle (2pt);
    \path[use as bounding box] \shiftpt{(Id)}{2} rectangle \shiftpt{(Ib)}{-1};
  \end{tikzpicture}
\end{tikzcd}
\]

}
realizes
\[
\begin{tikzcd}[ampersand replacement=\&]
    K^d \rar{I} \&
    K^d \&
    K^d \lar[swap]{I} \\
    K^d
    \rar{\smat{I\\0}}
    \uar{I}
    \&
    K^{2d} \uar[swap]{\smat{I&I}}
    \& K^d \lar{\smat{0\\I}}
    \uar[swap]{I}
    \\
    0 \uar \rar \& K^d \uar{\smat{I\\J_d(0)}} \& 0 \lar\uar
  \end{tikzcd}.
\]

\paragraph{1 In 3 Out} Next, with one inbound and three outbound arrows, we have the diagram:
{
  \input{drawing_macros}
\newcommand{\myscale}{.5}

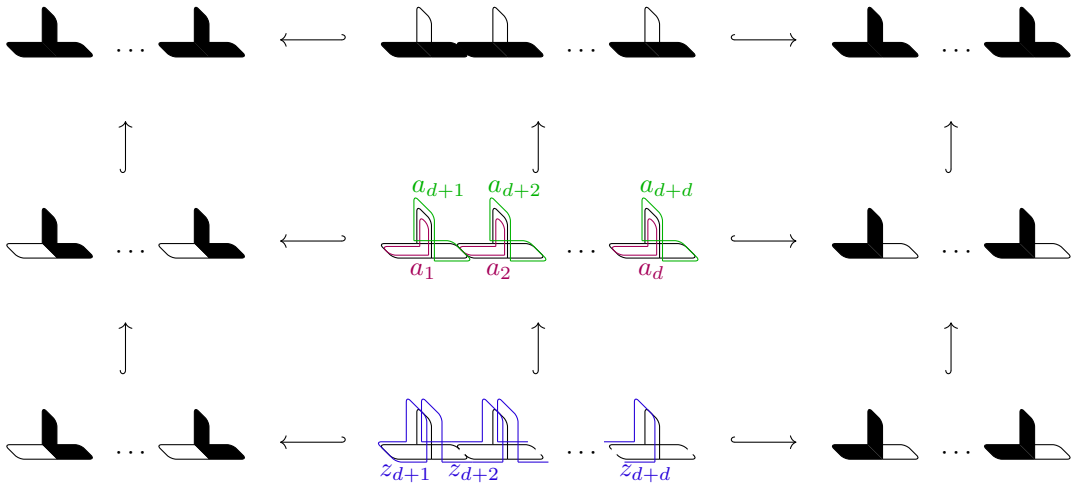
\[
\begin{tikzcd}[ampersand replacement = \&, row sep = 2em, column sep = 2.5em,nodes={inner sep = 1.2em}, arrows={hookrightarrow}]
  \begin{tikzpicture}[scale=\myscale,baseline=(current bounding box.center),rotate=90]
    \definesandalpts    
    \drawuphk{black,fill}{2}
    \drawstrp{black,fill}{2}
    \drawdwhk{black,fill}{2}
    \drawuphk{black,fill}{0}
    \drawstrp{black,fill}{0}
    \drawdwhk{black,fill}{0}
    \node[xshift=0em] at \shiftpt{(Sb)}{1} {$\cdots$};
    \path[use as bounding box] \shiftpt{(Id)}{2} rectangle \shiftpt{(Ib)}{-1};
  \end{tikzpicture}
\rar[hookleftarrow] \&
    \begin{tikzpicture}[scale=\myscale,baseline=(current bounding box.center),rotate=90]
    \definesandalpts
    \drawuphk{black,fill}{3}
    \drawstrp{black}{3}
    \drawdwhk{black,fill}{3}
    \drawuphk{black,fill}{2}
    \drawstrp{black}{2}
    \drawdwhk{black,fill}{2}
    \drawuphk{black,fill}{0}
    \drawstrp{black}{0}
    \drawdwhk{black,fill}{0}
    \node[xshift=0em] at \shiftpt{(Sb)}{1} {$\cdots$};
    \path[use as bounding box] \shiftpt{(Id)}{3} rectangle \shiftpt{(Ib)}{-1};
  \end{tikzpicture}
  \rar \&
  \begin{tikzpicture}[scale=\myscale,baseline=(current bounding box.center),rotate=90]
    \definesandalpts    
    \drawuphk{black,fill}{2}
    \drawstrp{black,fill}{2}
    \drawdwhk{black,fill}{2}
    \drawuphk{black,fill}{0}
    \drawstrp{black,fill}{0}
    \drawdwhk{black,fill}{0}
    \node[xshift=0em] at \shiftpt{(Sb)}{1} {$\cdots$};
    \path[use as bounding box] \shiftpt{(Id)}{2} rectangle \shiftpt{(Ib)}{-1};
    \end{tikzpicture}
  \\
  \begin{tikzpicture}[scale=\myscale,baseline=(current bounding box.center),rotate=90]
    \definesandalpts    
    \drawuphk{black}{2}
    \drawstrp{black,fill}{2}
    \drawdwhk{black,fill}{2}
    \drawuphk{black}{0}
    \drawstrp{black,fill}{0}
    \drawdwhk{black,fill}{0}    
    \node[xshift=0em] at \shiftpt{(Sb)}{1} {$\cdots$};
    \path[use as bounding box] \shiftpt{(Id)}{2} rectangle \shiftpt{(Ib)}{-1};
  \end{tikzpicture}
\uar\rar[hookleftarrow] \&
\begin{tikzpicture}[scale=\myscale,baseline=(current bounding box.center),rotate=90]
    \definesandalpts
    \drawuphk{black}{3}
    \drawstrp{black}{3}
    \drawdwhk{black}{3}
    \drawuphk{black}{2}
    \drawstrp{black}{2}
    \drawdwhk{black}{2}
    \drawuphk{black}{0}
    \drawstrp{black}{0}
    \drawdwhk{black}{0}
    \node[xshift=0em] at \shiftpt{(Sb)}{1} {$\cdots$};
    \path[use as bounding box] \shiftpt{(Id)}{3} rectangle \shiftpt{(Ib)}{-1};

    \shrunkuphk{UpperRowColor1}{3}
    \shrunkstrp{UpperRowColor1}{3}
    \shrunkuphk{UpperRowColor1}{2}
    \shrunkstrp{UpperRowColor1}{2}
    \shrunkuphk{UpperRowColor1}{0}
    \shrunkstrp{UpperRowColor1}{0}
    \expandstrp{UpperRowColor2}{3}
    \expanddwhk{UpperRowColor2}{3}{3}
    \expandstrp{UpperRowColor2}{2}
    \expanddwhk{UpperRowColor2}{2}{2}
    \expandstrp{UpperRowColor2}{0}
    \expanddwhk{UpperRowColor2}{0}{0}

    \node[anchor=west, UpperRowColor1] at \shiftptno{(IbEx)}{3.25}{-.1} {$a_1$};
    \node[anchor=west, UpperRowColor1] at \shiftptno{(IbEx)}{2.25}{-.1} {$a_2$};
    \node[anchor=west, UpperRowColor1] at \shiftptno{(IbEx)}{.25}{-.1} {$a_d$};
    \node[anchor=north, UpperRowColor2] at \shiftptno{(SdSh)}{2.75}{.9} {$a_{d+1}$};
    \node[anchor=north, UpperRowColor2] at \shiftptno{(SdSh)}{1.75}{.9} {$a_{d+2}$};
    \node[anchor=north, UpperRowColor2] at \shiftptno{(SdSh)}{-.25}{.9} {$a_{d+d}$};   
  \end{tikzpicture}
  \uar\rar\&
 \begin{tikzpicture}[scale=\myscale,baseline=(current bounding box.center),rotate=90]
    \definesandalpts    
    \drawuphk{black,fill}{2}
    \drawstrp{black,fill}{2}
    \drawdwhk{black}{2}
    \drawuphk{black,fill}{0}
    \drawstrp{black,fill}{0}
    \drawdwhk{black}{0}
    \node[xshift=0em] at \shiftpt{(Sb)}{1} {$\cdots$};
    \path[use as bounding box] \shiftpt{(Id)}{2} rectangle \shiftpt{(Ib)}{-1};
  \end{tikzpicture}
    \uar
  \\
  \begin{tikzpicture}[scale=\myscale,baseline=(current bounding box.center),rotate=90]
    \definesandalpts
    \drawuphk{black}{2}
    \drawstrp{black,fill}{2}
    \drawdwhk{black,fill}{2}
    \drawuphk{black}{0}
    \drawstrp{black,fill}{0}
    \drawdwhk{black,fill}{0}
    \node[xshift=0em] at \shiftpt{(Sb)}{1} {$\cdots$};
    \path[use as bounding box] \shiftpt{(Id)}{2} rectangle \shiftpt{(Ib)}{-1};
  \end{tikzpicture}
  \uar\rar[hookleftarrow] \&
  \begin{tikzpicture}[scale=\myscale,baseline=(current bounding box.center),rotate=90]
    \definesandalpts
    \drawuphk{black}{3}
    \drawstrp{black}{3}
    \drawdwhk{black}{3}
    \drawuphk{black}{2}
    \drawstrp{black}{2}
    \drawdwhk{black}{2}
    \drawuphk{black}{0}
    \drawstrp{black}{0}
    \drawdwhk{black}{0}
    \node[xshift=0em] at \shiftpt{(Sb)}{1} {$\cdots$};
    \path[use as bounding box] \shiftpt{(Id)}{3} rectangle \shiftpt{(Ib)}{-1};

    \node [coordinate] (Ina) at \shiftpt{(Ia)}{-1.2} {};
    \node [coordinate] (Inb) at \shiftpt{(Ib)}{-0.8} {};

    \node [coordinate] (I0a) at \shiftpt{(Ia)}{-0.2} {};
    \node [coordinate] (I0b) at \shiftpt{(Ib)}{0.2} {};

    \node [coordinate] (I1a) at \shiftpt{(Ia)}{0.8} {};
    \node [coordinate] (I1b) at \shiftpt{(Ib)}{1.2} {};

    \node [coordinate] (I2a) at \shiftpt{(Ia)}{1.8} {};
    \node [coordinate] (I2b) at \shiftpt{(Ib)}{2.2} {};

    \expanduphk{LowerRowColor1}{3}{3.1}
    \expandstrp{LowerRowColor1}{3.1}

    \expandstrp{LowerRowColor1}{2.9}
    \expanddwlegs{LowerRowColor1}{3}{2.9}
    \expanduplegs{LowerRowColor1}{2}{2.1}
    \expandstrp{LowerRowColor1}{2.1}

    \expandstrp{LowerRowColor1}{1.9}
    \expanddwlegs{LowerRowColor1}{2}{1.9}

    \expanduplegs{LowerRowColor1}{0}{0.1}
    \expandstrp{LowerRowColor1}{0.1}
    \draw[fill, white] ($ (I2a)!0.35!(I2b) $) rectangle ($ (I2b)!0.35!(I2a) $);
    \draw[fill, white] ($ (I1a)!0.35!(I1b) $) rectangle ($ (I1b)!0.35!(I1a) $);
    \draw[fill, white] ($ (I0a)!0.35!(I0b) $) rectangle ($ (I0b)!0.35!(I0a) $);
    \draw[fill, white] ($ (Ina)!0.35!(Inb) $) rectangle ($ (Inb)!0.35!(Ina) $);

    \node[anchor=west, LowerRowColor1] at \shiftptno{(IbEx)}{3.65}{-.1} {$z_{d+1}$};
    \node[anchor=west, LowerRowColor1] at \shiftptno{(IbEx)}{2.75}{-.1} {$z_{d+2}$};
    \node[anchor=west, LowerRowColor1] at \shiftptno{(IbEx)}{.5}{-.1} {$z_{d+d}$};
  \end{tikzpicture}
  \uar\rar \&\uar
 \begin{tikzpicture}[scale=\myscale,baseline=(current bounding box.center),rotate=90]
    \definesandalpts    
    \drawuphk{black,fill}{2}
    \drawstrp{black,fill}{2}
    \drawdwhk{black}{2}
    \drawuphk{black,fill}{0}
    \drawstrp{black,fill}{0}
    \drawdwhk{black}{0}
    \node[xshift=0em] at \shiftpt{(Sb)}{1} {$\cdots$};
    \path[use as bounding box] \shiftpt{(Id)}{2} rectangle \shiftpt{(Ib)}{-1};
  \end{tikzpicture}
\end{tikzcd}
\]

}
realizing the persistence module
\[
  \begin{tikzcd}[ampersand replacement=\&]
    0 \& \lar
    K^d \rar\&
    0  \\
    K^d
    \uar
    \&
    \lar[swap]{\smat{I&0}}
    K^{2d} \uar[swap]{\smat{I&I}}
    \rar{\smat{0&I}} \& K^d 
    \uar
    \\
    K^d \uar{I} \& \lar{I} K^d \uar{\smat{I\\J_d(0)}} \rar[swap]{J_d(0)} \& K^d \uar{I}
  \end{tikzcd}
\]

\paragraph{4 In} With four inbound:
{
  \input{drawing_macros}
\newcommand{\myscale}{.5}

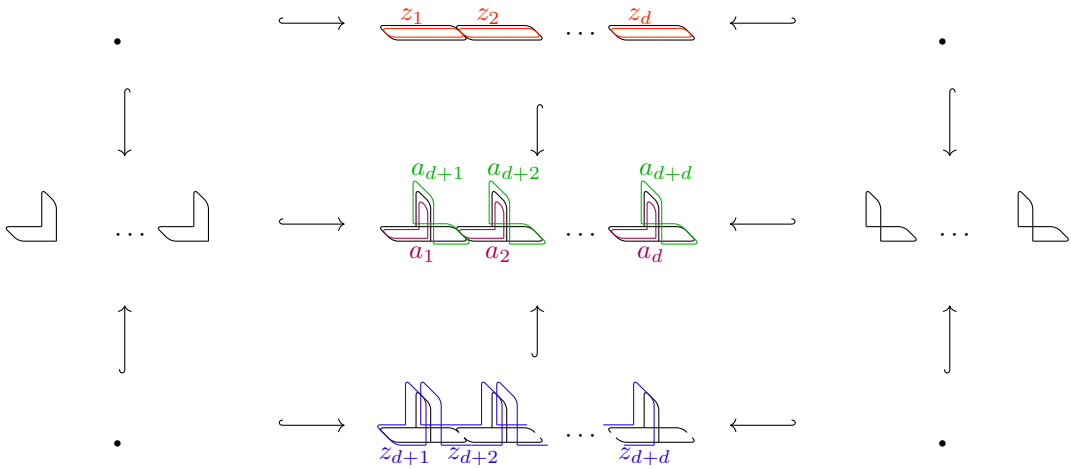
\[
\begin{tikzcd}[ampersand replacement = \&, row sep = 2em, column sep = 2.5em,nodes={inner sep = 1.2em}, arrows={hookrightarrow}]
 \begin{tikzpicture}[scale=\myscale,baseline=(current bounding box.center),rotate=90]
    \definesandalpts
    \draw[fill] \shiftpt{(Sa)}{1} circle (2pt);
    \path[use as bounding box] \shiftpt{(Id)}{2} rectangle \shiftpt{(Ib)}{-1};
  \end{tikzpicture}
\rar \dar \&
   \begin{tikzpicture}[scale=\myscale,baseline=(current bounding box.center),rotate=90]
    \definesandalpts
    \drawuphk{black}{3}
    \drawdwhk{black}{3}
    \drawuphk{black}{2}
    \drawdwhk{black}{2}
    \drawuphk{black}{0}
    \drawdwhk{black}{0}
    \node[xshift=0em] at \shiftpt{(Sb)}{1} {$\cdots$};
    \path[use as bounding box] \shiftpt{(Id)}{3} rectangle \shiftpt{(Ib)}{-1};

    \shrunkuphk{LowerRowColor2}{3}
    \shrunkdwhk{LowerRowColor2}{3}
    \shrunkuphk{LowerRowColor2}{2}
    \shrunkdwhk{LowerRowColor2}{2}
    \shrunkuphk{LowerRowColor2}{0}
    \shrunkdwhk{LowerRowColor2}{0}

    \node[anchor=east, LowerRowColor2] at \shiftptno{(SaSh)}{2.5}{.2} {$z_{1}$};
    \node[anchor=east, LowerRowColor2] at \shiftptno{(SaSh)}{1.5}{.2} {$z_{2}$};
    \node[anchor=east, LowerRowColor2] at \shiftptno{(SaSh)}{-.5}{.2} {$z_{d}$};
  \end{tikzpicture}
  \dar \rar[hookleftarrow]  \&\dar
  \begin{tikzpicture}[scale=\myscale,baseline=(current bounding box.center),rotate=90]
    \definesandalpts
    \draw[fill] \shiftpt{(Sa)}{1} circle (2pt);
    \path[use as bounding box] \shiftpt{(Id)}{2} rectangle \shiftpt{(Ib)}{-1};
  \end{tikzpicture}
  \\
  \begin{tikzpicture}[scale=\myscale,baseline=(current bounding box.center),rotate=90]
    \definesandalpts    
    \drawuphk{black}{2}
    \drawstrp{black}{2}
    \drawuphk{black}{0}
    \drawstrp{black}{0}
    \node[xshift=0em] at \shiftpt{(Sb)}{1} {$\cdots$};
    \path[use as bounding box] \shiftpt{(Id)}{2} rectangle \shiftpt{(Ib)}{-1};
  \end{tikzpicture}
\rar\&
\begin{tikzpicture}[scale=\myscale,baseline=(current bounding box.center),rotate=90]
    \definesandalpts
    \drawuphk{black}{3}
    \drawstrp{black}{3}
    \drawdwhk{black}{3}
    \drawuphk{black}{2}
    \drawstrp{black}{2}
    \drawdwhk{black}{2}
    \drawuphk{black}{0}
    \drawstrp{black}{0}
    \drawdwhk{black}{0}
    \node[xshift=0em] at \shiftpt{(Sb)}{1} {$\cdots$};
    \path[use as bounding box] \shiftpt{(Id)}{3} rectangle \shiftpt{(Ib)}{-1};

    \shrunkuphk{UpperRowColor1}{3}
    \shrunkstrp{UpperRowColor1}{3}
    \shrunkuphk{UpperRowColor1}{2}
    \shrunkstrp{UpperRowColor1}{2}
    \shrunkuphk{UpperRowColor1}{0}
    \shrunkstrp{UpperRowColor1}{0}
    \expandstrp{UpperRowColor2}{3}
    \expanddwhk{UpperRowColor2}{3}{3}
    \expandstrp{UpperRowColor2}{2}
    \expanddwhk{UpperRowColor2}{2}{2}
    \expandstrp{UpperRowColor2}{0}
    \expanddwhk{UpperRowColor2}{0}{0}

    \node[anchor=west, UpperRowColor1] at \shiftptno{(IbEx)}{3.25}{-.1} {$a_1$};
    \node[anchor=west, UpperRowColor1] at \shiftptno{(IbEx)}{2.25}{-.1} {$a_2$};
    \node[anchor=west, UpperRowColor1] at \shiftptno{(IbEx)}{.25}{-.1} {$a_d$};
    \node[anchor=north, UpperRowColor2] at \shiftptno{(SdSh)}{2.75}{.9} {$a_{d+1}$};
    \node[anchor=north, UpperRowColor2] at \shiftptno{(SdSh)}{1.75}{.9} {$a_{d+2}$};
    \node[anchor=north, UpperRowColor2] at \shiftptno{(SdSh)}{-.25}{.9} {$a_{d+d}$};   
  \end{tikzpicture}
  \rar[hookleftarrow] \&
  \begin{tikzpicture}[scale=\myscale,baseline=(current bounding box.center),rotate=90]
  \definesandalpts    
    \drawstrp{black}{2}
    \drawdwhk{black}{2}
    \drawstrp{black}{0}
    \drawdwhk{black}{0}
    \node[xshift=0em] at \shiftpt{(Sb)}{1} {$\cdots$};
    \path[use as bounding box] \shiftpt{(Id)}{2} rectangle \shiftpt{(Ib)}{-1};
    \end{tikzpicture}
  \\
  \begin{tikzpicture}[scale=\myscale,baseline=(current bounding box.center),rotate=90]
    \definesandalpts
    \draw[fill] \shiftpt{(Sa)}{1} circle (2pt);
    \path[use as bounding box] \shiftpt{(Id)}{2} rectangle \shiftpt{(Ib)}{-1};
  \end{tikzpicture}
  \rar\uar\&
  \begin{tikzpicture}[scale=\myscale,baseline=(current bounding box.center),rotate=90]
    \definesandalpts
    \drawuphk{black}{3}
    \drawstrp{black}{3}
    \drawdwhk{black}{3}
    \drawuphk{black}{2}
    \drawstrp{black}{2}
    \drawdwhk{black}{2}
    \drawuphk{black}{0}
    \drawstrp{black}{0}
    \drawdwhk{black}{0}
    \node[xshift=0em] at \shiftpt{(Sb)}{1} {$\cdots$};
    \path[use as bounding box] \shiftpt{(Id)}{3} rectangle \shiftpt{(Ib)}{-1};

    \node [coordinate] (Ina) at \shiftpt{(Ia)}{-1.2} {};
    \node [coordinate] (Inb) at \shiftpt{(Ib)}{-0.8} {};

    \node [coordinate] (I0a) at \shiftpt{(Ia)}{-0.2} {};
    \node [coordinate] (I0b) at \shiftpt{(Ib)}{0.2} {};

    \node [coordinate] (I1a) at \shiftpt{(Ia)}{0.8} {};
    \node [coordinate] (I1b) at \shiftpt{(Ib)}{1.2} {};

    \node [coordinate] (I2a) at \shiftpt{(Ia)}{1.8} {};
    \node [coordinate] (I2b) at \shiftpt{(Ib)}{2.2} {};

    \expanduphk{LowerRowColor1}{3}{3.1}
    \expandstrp{LowerRowColor1}{3.1}

    \expandstrp{LowerRowColor1}{2.9}
    \expanddwlegs{LowerRowColor1}{3}{2.9}
    \expanduplegs{LowerRowColor1}{2}{2.1}
    \expandstrp{LowerRowColor1}{2.1}

    \expandstrp{LowerRowColor1}{1.9}
    \expanddwlegs{LowerRowColor1}{2}{1.9}

    \expanduplegs{LowerRowColor1}{0}{0.1}
    \expandstrp{LowerRowColor1}{0.1}
    \draw[fill, white] ($ (I2a)!0.35!(I2b) $) rectangle ($ (I2b)!0.35!(I2a) $);
    \draw[fill, white] ($ (I1a)!0.35!(I1b) $) rectangle ($ (I1b)!0.35!(I1a) $);
    \draw[fill, white] ($ (I0a)!0.35!(I0b) $) rectangle ($ (I0b)!0.35!(I0a) $);
    \draw[fill, white] ($ (Ina)!0.35!(Inb) $) rectangle ($ (Inb)!0.35!(Ina) $);

    \node[anchor=west, LowerRowColor1] at \shiftptno{(IbEx)}{3.65}{-.1} {$z_{d+1}$};
    \node[anchor=west, LowerRowColor1] at \shiftptno{(IbEx)}{2.75}{-.1} {$z_{d+2}$};
    \node[anchor=west, LowerRowColor1] at \shiftptno{(IbEx)}{.5}{-.1} {$z_{d+d}$};
  \end{tikzpicture}
  \uar\rar[hookleftarrow] \&\uar
  \begin{tikzpicture}[scale=\myscale,baseline=(current bounding box.center),rotate=90]
    \definesandalpts
    \draw[fill] \shiftpt{(Sa)}{1} circle (2pt);
    \path[use as bounding box] \shiftpt{(Id)}{2} rectangle \shiftpt{(Ib)}{-1};
  \end{tikzpicture}
\end{tikzcd}
\]
}
for the persistence module:
\[
  \begin{tikzcd}[ampersand replacement=\&]
    0 \rar\dar \&
    K^d \dar{\smat{I\\I}}\&
    0 \lar \dar \\
    K^d
    \rar{\smat{I\\0}}
    \&
    K^{2d}
    \& K^d \lar{\smat{0\\I}}
    \\
    0 \uar \rar \& K^d \uar{\smat{I\\J_d(0)}} \& 0 \lar\uar
  \end{tikzcd}.
\]

\paragraph{2 In 2 Out, version 1} Two inbound and two outbound, first type:
{
  \input{drawing_macros}
\newcommand{\myscale}{.5}

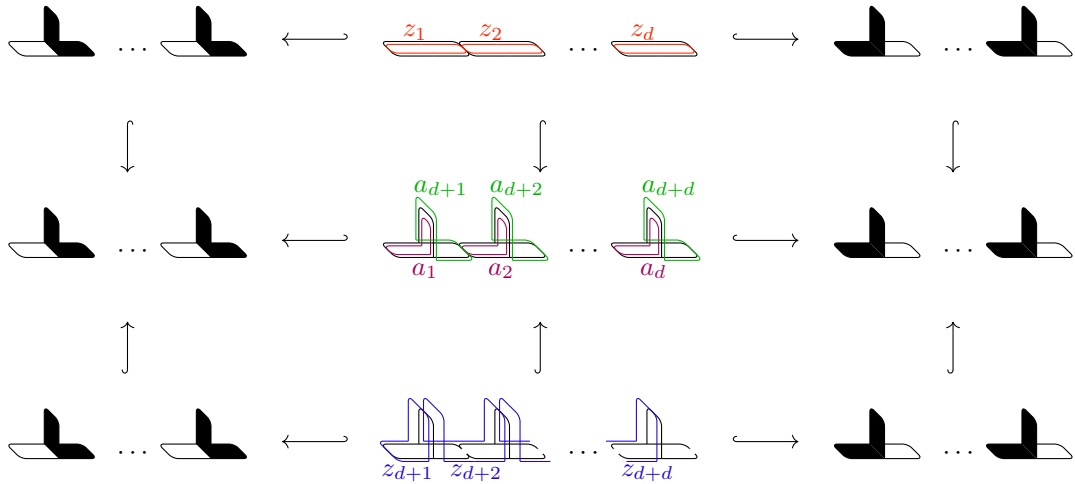
\[
\begin{tikzcd}[ampersand replacement = \&, row sep = 2em, column sep = 2.5em,nodes={inner sep = 1.2em}, arrows={hookrightarrow}]
  \begin{tikzpicture}[scale=\myscale,baseline=(current bounding box.center),rotate=90]
    \definesandalpts    
    \drawuphk{black}{2}
    \drawstrp{black,fill}{2}
    \drawdwhk{black,fill}{2}
    \drawuphk{black}{0}
    \drawstrp{black,fill}{0}
    \drawdwhk{black,fill}{0}
    \node[xshift=0em] at \shiftpt{(Sb)}{1} {$\cdots$};
    \path[use as bounding box] \shiftpt{(Id)}{2} rectangle \shiftpt{(Ib)}{-1};
  \end{tikzpicture}
\dar \rar[hookleftarrow] \&
   \begin{tikzpicture}[scale=\myscale,baseline=(current bounding box.center),rotate=90]
    \definesandalpts
    \drawuphk{black}{3}
    \drawdwhk{black}{3}
    \drawuphk{black}{2}
    \drawdwhk{black}{2}
    \drawuphk{black}{0}
    \drawdwhk{black}{0}
    \node[xshift=0em] at \shiftpt{(Sb)}{1} {$\cdots$};
    \path[use as bounding box] \shiftpt{(Id)}{3} rectangle \shiftpt{(Ib)}{-1};

    \shrunkuphk{LowerRowColor2}{3}
    \shrunkdwhk{LowerRowColor2}{3}
    \shrunkuphk{LowerRowColor2}{2}
    \shrunkdwhk{LowerRowColor2}{2}
    \shrunkuphk{LowerRowColor2}{0}
    \shrunkdwhk{LowerRowColor2}{0}

    \node[anchor=east, LowerRowColor2] at \shiftptno{(SaSh)}{2.5}{.2} {$z_{1}$};
    \node[anchor=east, LowerRowColor2] at \shiftptno{(SaSh)}{1.5}{.2} {$z_{2}$};
    \node[anchor=east, LowerRowColor2] at \shiftptno{(SaSh)}{-.5}{.2} {$z_{d}$};
  \end{tikzpicture}
  \rar \dar \&
  \begin{tikzpicture}[scale=\myscale,baseline=(current bounding box.center),rotate=90]
    \definesandalpts    
    \drawuphk{black,fill}{2}
    \drawstrp{black,fill}{2}
    \drawdwhk{black}{2}
    \drawuphk{black,fill}{0}
    \drawstrp{black,fill}{0}
    \drawdwhk{black}{0}
    \node[xshift=0em] at \shiftpt{(Sb)}{1} {$\cdots$};
    \path[use as bounding box] \shiftpt{(Id)}{2} rectangle \shiftpt{(Ib)}{-1};
  \end{tikzpicture}
\dar  \\
  \begin{tikzpicture}[scale=\myscale,baseline=(current bounding box.center),rotate=90]
    \definesandalpts    
    \drawuphk{black}{2}
    \drawstrp{black,fill}{2}
    \drawdwhk{black,fill}{2}
    \drawuphk{black}{0}
    \drawstrp{black,fill}{0}
    \drawdwhk{black,fill}{0}
    \node[xshift=0em] at \shiftpt{(Sb)}{1} {$\cdots$};    
    \path[use as bounding box] \shiftpt{(Id)}{2} rectangle \shiftpt{(Ib)}{-1};
  \end{tikzpicture}
\rar[hookleftarrow] \&
\begin{tikzpicture}[scale=\myscale,baseline=(current bounding box.center),rotate=90]
    \definesandalpts
    \drawuphk{black}{3}
    \drawstrp{black}{3}
    \drawdwhk{black}{3}
    \drawuphk{black}{2}
    \drawstrp{black}{2}
    \drawdwhk{black}{2}
    \drawuphk{black}{0}
    \drawstrp{black}{0}
    \drawdwhk{black}{0}
    \node[xshift=0em] at \shiftpt{(Sb)}{1} {$\cdots$};
    \path[use as bounding box] \shiftpt{(Id)}{3} rectangle \shiftpt{(Ib)}{-1};

    \shrunkuphk{UpperRowColor1}{3}
    \shrunkstrp{UpperRowColor1}{3}
    \shrunkuphk{UpperRowColor1}{2}
    \shrunkstrp{UpperRowColor1}{2}
    \shrunkuphk{UpperRowColor1}{0}
    \shrunkstrp{UpperRowColor1}{0}
    \expandstrp{UpperRowColor2}{3}
    \expanddwhk{UpperRowColor2}{3}{3}
    \expandstrp{UpperRowColor2}{2}
    \expanddwhk{UpperRowColor2}{2}{2}
    \expandstrp{UpperRowColor2}{0}
    \expanddwhk{UpperRowColor2}{0}{0}

    \node[anchor=west, UpperRowColor1] at \shiftptno{(IbEx)}{3.25}{-.1} {$a_1$};
    \node[anchor=west, UpperRowColor1] at \shiftptno{(IbEx)}{2.25}{-.1} {$a_2$};
    \node[anchor=west, UpperRowColor1] at \shiftptno{(IbEx)}{.25}{-.1} {$a_d$};
    \node[anchor=north, UpperRowColor2] at \shiftptno{(SdSh)}{2.75}{.9} {$a_{d+1}$};
    \node[anchor=north, UpperRowColor2] at \shiftptno{(SdSh)}{1.75}{.9} {$a_{d+2}$};
    \node[anchor=north, UpperRowColor2] at \shiftptno{(SdSh)}{-.25}{.9} {$a_{d+d}$};
  \end{tikzpicture}
  \rar\&
 \begin{tikzpicture}[scale=\myscale,baseline=(current bounding box.center),rotate=90]
    \definesandalpts    
    \drawuphk{black,fill}{2}
    \drawstrp{black,fill}{2}
    \drawdwhk{black}{2}
    \drawuphk{black,fill}{0}
    \drawstrp{black,fill}{0}
    \drawdwhk{black}{0}
    \node[xshift=0em] at \shiftpt{(Sb)}{1} {$\cdots$};
    \path[use as bounding box] \shiftpt{(Id)}{2} rectangle \shiftpt{(Ib)}{-1};
  \end{tikzpicture}
  \\
  \begin{tikzpicture}[scale=\myscale,baseline=(current bounding box.center),rotate=90]
    \definesandalpts    
    \drawuphk{black}{2}
    \drawstrp{black,fill}{2}
    \drawdwhk{black,fill}{2}
    \drawuphk{black}{0}
    \drawstrp{black,fill}{0}
    \drawdwhk{black,fill}{0}
    \node[xshift=0em] at \shiftpt{(Sb)}{1} {$\cdots$};
    \path[use as bounding box] \shiftpt{(Id)}{2} rectangle \shiftpt{(Ib)}{-1};
  \end{tikzpicture}
  \rar[hookleftarrow] \uar\&
  \begin{tikzpicture}[scale=\myscale,baseline=(current bounding box.center),rotate=90]
    \definesandalpts
    \drawuphk{black}{3}
    \drawstrp{black}{3}
    \drawdwhk{black}{3}
    \drawuphk{black}{2}
    \drawstrp{black}{2}
    \drawdwhk{black}{2}
    \drawuphk{black}{0}
    \drawstrp{black}{0}
    \drawdwhk{black}{0}
    \node[xshift=0em] at \shiftpt{(Sb)}{1} {$\cdots$};
    \path[use as bounding box] \shiftpt{(Id)}{3} rectangle \shiftpt{(Ib)}{-1};

    \node [coordinate] (Ina) at \shiftpt{(Ia)}{-1.2} {};
    \node [coordinate] (Inb) at \shiftpt{(Ib)}{-0.8} {};

    \node [coordinate] (I0a) at \shiftpt{(Ia)}{-0.2} {};
    \node [coordinate] (I0b) at \shiftpt{(Ib)}{0.2} {};

    \node [coordinate] (I1a) at \shiftpt{(Ia)}{0.8} {};
    \node [coordinate] (I1b) at \shiftpt{(Ib)}{1.2} {};

    \node [coordinate] (I2a) at \shiftpt{(Ia)}{1.8} {};
    \node [coordinate] (I2b) at \shiftpt{(Ib)}{2.2} {};

    \expanduphk{LowerRowColor1}{3}{3.1}
    \expandstrp{LowerRowColor1}{3.1}

    \expandstrp{LowerRowColor1}{2.9}
    \expanddwlegs{LowerRowColor1}{3}{2.9}
    \expanduplegs{LowerRowColor1}{2}{2.1}
    \expandstrp{LowerRowColor1}{2.1}

    \expandstrp{LowerRowColor1}{1.9}
    \expanddwlegs{LowerRowColor1}{2}{1.9}

    \expanduplegs{LowerRowColor1}{0}{0.1}
    \expandstrp{LowerRowColor1}{0.1}
    \draw[fill, white] ($ (I2a)!0.35!(I2b) $) rectangle ($ (I2b)!0.35!(I2a) $);
    \draw[fill, white] ($ (I1a)!0.35!(I1b) $) rectangle ($ (I1b)!0.35!(I1a) $);
    \draw[fill, white] ($ (I0a)!0.35!(I0b) $) rectangle ($ (I0b)!0.35!(I0a) $);
    \draw[fill, white] ($ (Ina)!0.35!(Inb) $) rectangle ($ (Inb)!0.35!(Ina) $);

    \node[anchor=west, LowerRowColor1] at \shiftptno{(IbEx)}{3.65}{-.1} {$z_{d+1}$};
    \node[anchor=west, LowerRowColor1] at \shiftptno{(IbEx)}{2.75}{-.1} {$z_{d+2}$};
    \node[anchor=west, LowerRowColor1] at \shiftptno{(IbEx)}{.5}{-.1} {$z_{d+d}$};
  \end{tikzpicture}
  \uar\rar \&\uar
 \begin{tikzpicture}[scale=\myscale,baseline=(current bounding box.center),rotate=90]
    \definesandalpts    
    \drawuphk{black,fill}{2}
    \drawstrp{black,fill}{2}
    \drawdwhk{black}{2}
    \drawuphk{black,fill}{0}
    \drawstrp{black,fill}{0}
    \drawdwhk{black}{0}
    \node[xshift=0em] at \shiftpt{(Sb)}{1} {$\cdots$};
    \path[use as bounding box] \shiftpt{(Id)}{2} rectangle \shiftpt{(Ib)}{-1};
  \end{tikzpicture}
\end{tikzcd}
\]
}
realizes
\[
\begin{tikzcd}[ampersand replacement=\&]
  K^d \dar{I} \& \lar{I}
  K^d \rar{I} \dar{\smat{I\\I}} \&
  K^d \dar{I} \\
  K^d
  \&
  \lar[swap]{\smat{I&0}}
  K^{2d}
  \rar{\smat{0&I}} \& K^d 
  \\
  K^d \uar{I} \& \lar{I} K^d \uar{\smat{I\\J_d(0)}} \rar{J_d(0)} \& K^d \uar{I}
\end{tikzcd}
\]

\paragraph{2 In 2 Out, version 2} Two inbound and two outbound, second type:
{
  \input{drawing_macros}
\newcommand{\myscale}{.5}
\[
  \begin{tikzcd}[ampersand replacement = \&, row sep = 2em, column sep = 2.5em,nodes={inner sep = 1.2em}, arrows={hookrightarrow}]
    \begin{tikzpicture}[scale=\myscale,baseline=(current bounding box.center),rotate=90]
      \definesandalpts
      \drawuphk{black}{2}
      \drawstrp{black,fill}{2}
      \drawdwhk{black,fill}{2}
      \drawuphk{black}{0}
      \drawstrp{black,fill}{0}
      \drawdwhk{black,fill}{0}
      \node[xshift=0em] at \shiftpt{(Sb)}{1} {$\cdots$};
      \path[use as bounding box] \shiftpt{(Id)}{2} rectangle \shiftpt{(Ib)}{-1};
    \end{tikzpicture}
    \dar \&\lar
    \begin{tikzpicture}[scale=\myscale,baseline=(current bounding box.center),rotate=90]
      \definesandalpts
      \drawuphk{black}{3}
      \drawdwhk{black}{3}
      \drawuphk{black}{2}
      \drawdwhk{black}{2}
      \drawuphk{black}{0}
      \drawdwhk{black}{0}
      \node[xshift=0em] at \shiftpt{(Sb)}{1} {$\cdots$};
      \path[use as bounding box] \shiftpt{(Id)}{3} rectangle \shiftpt{(Ib)}{-1};

      \shrunkuphk{LowerRowColor2}{3}
      \shrunkdwhk{LowerRowColor2}{3}
      \shrunkuphk{LowerRowColor2}{2}
      \shrunkdwhk{LowerRowColor2}{2}
      \shrunkuphk{LowerRowColor2}{0}
      \shrunkdwhk{LowerRowColor2}{0}

      \node[anchor=east, LowerRowColor2] at \shiftptno{(SaSh)}{2.5}{.2} {$z_{1}$};
      \node[anchor=east, LowerRowColor2] at \shiftptno{(SaSh)}{1.5}{.2} {$z_{2}$};
      \node[anchor=east, LowerRowColor2] at \shiftptno{(SaSh)}{-.5}{.2} {$z_{d}$};
    \end{tikzpicture}
    \dar \& \lar
    \begin{tikzpicture}[scale=\myscale,baseline=(current bounding box.center),rotate=90]
      \definesandalpts
      \draw[fill] \shiftpt{(Sa)}{1} circle (2pt);
      \path[use as bounding box] \shiftpt{(Id)}{2} rectangle \shiftpt{(Ib)}{-1};
    \end{tikzpicture}
    \dar  \\
    \begin{tikzpicture}[scale=\myscale,baseline=(current bounding box.center),rotate=90]
      \definesandalpts
      \drawuphk{black}{2}
      \drawstrp{black,fill}{2}
      \drawdwhk{black,fill}{2}
      \drawuphk{black}{0}
      \drawstrp{black,fill}{0}
      \drawdwhk{black,fill}{0}
      \node[xshift=0em] at \shiftpt{(Sb)}{1} {$\cdots$};
      \path[use as bounding box] \shiftpt{(Id)}{2} rectangle \shiftpt{(Ib)}{-1};
    \end{tikzpicture}
    \dar \&\lar
    \begin{tikzpicture}[scale=\myscale,baseline=(current bounding box.center),rotate=90]
      \definesandalpts
      \drawuphk{black}{3}
      \drawstrp{black}{3}
      \drawdwhk{black}{3}
      \drawuphk{black}{2}
      \drawstrp{black}{2}
      \drawdwhk{black}{2}
      \drawuphk{black}{0}
      \drawstrp{black}{0}
      \drawdwhk{black}{0}
      \node[xshift=0em] at \shiftpt{(Sb)}{1} {$\cdots$};
      \path[use as bounding box] \shiftpt{(Id)}{3} rectangle \shiftpt{(Ib)}{-1};

      \shrunkuphk{UpperRowColor1}{3}
      \shrunkstrp{UpperRowColor1}{3}
      \shrunkuphk{UpperRowColor1}{2}
      \shrunkstrp{UpperRowColor1}{2}
      \shrunkuphk{UpperRowColor1}{0}
      \shrunkstrp{UpperRowColor1}{0}
      \expandstrp{UpperRowColor2}{3}
      \expanddwhk{UpperRowColor2}{3}{3}
      \expandstrp{UpperRowColor2}{2}
      \expanddwhk{UpperRowColor2}{2}{2}
      \expandstrp{UpperRowColor2}{0}
      \expanddwhk{UpperRowColor2}{0}{0}

      \node[anchor=west, UpperRowColor1] at \shiftptno{(IbEx)}{3.25}{-.1} {$a_1$};
      \node[anchor=west, UpperRowColor1] at \shiftptno{(IbEx)}{2.25}{-.1} {$a_2$};
      \node[anchor=west, UpperRowColor1] at \shiftptno{(IbEx)}{.25}{-.1} {$a_d$};
      \node[anchor=north, UpperRowColor2] at \shiftptno{(SdSh)}{2.75}{.9} {$a_{d+1}$};
      \node[anchor=north, UpperRowColor2] at \shiftptno{(SdSh)}{1.75}{.9} {$a_{d+2}$};
      \node[anchor=north, UpperRowColor2] at \shiftptno{(SdSh)}{-.25}{.9} {$a_{d+d}$};
    \end{tikzpicture}
    \dar \& \lar
    \begin{tikzpicture}[scale=\myscale,baseline=(current bounding box.center),rotate=90]
      \definesandalpts
      \drawuphk{black}{3}
      \drawstrp{black}{3}
      \drawdwhk{black}{3}
      \drawuphk{black}{2}
      \drawstrp{black}{2}
      \drawdwhk{black}{2}
      \drawuphk{black}{0}
      \drawstrp{black}{0}
      \drawdwhk{black}{0}
      \node[xshift=0em] at \shiftpt{(Sb)}{1} {$\cdots$};
      \path[use as bounding box] \shiftpt{(Id)}{3} rectangle \shiftpt{(Ib)}{-1};

      \node [coordinate] (Ina) at \shiftpt{(Ia)}{-1.2} {};
      \node [coordinate] (Inb) at \shiftpt{(Ib)}{-0.8} {};

      \node [coordinate] (I0a) at \shiftpt{(Ia)}{-0.2} {};
      \node [coordinate] (I0b) at \shiftpt{(Ib)}{0.2} {};

      \node [coordinate] (I1a) at \shiftpt{(Ia)}{0.8} {};
      \node [coordinate] (I1b) at \shiftpt{(Ib)}{1.2} {};

      \node [coordinate] (I2a) at \shiftpt{(Ia)}{1.8} {};
      \node [coordinate] (I2b) at \shiftpt{(Ib)}{2.2} {};

      \expanduphk{LowerRowColor1}{3}{3.1}
      \expandstrp{LowerRowColor1}{3.1}

      \expandstrp{LowerRowColor1}{2.9}
      \expanddwlegs{LowerRowColor1}{3}{2.9}
      \expanduplegs{LowerRowColor1}{2}{2.1}
      \expandstrp{LowerRowColor1}{2.1}

      \expandstrp{LowerRowColor1}{1.9}
      \expanddwlegs{LowerRowColor1}{2}{1.9}

      \expanduplegs{LowerRowColor1}{0}{0.1}
      \expandstrp{LowerRowColor1}{0.1}
      \draw[fill, white] ($ (I2a)!0.35!(I2b) $) rectangle ($ (I2b)!0.35!(I2a) $);
      \draw[fill, white] ($ (I1a)!0.35!(I1b) $) rectangle ($ (I1b)!0.35!(I1a) $);
      \draw[fill, white] ($ (I0a)!0.35!(I0b) $) rectangle ($ (I0b)!0.35!(I0a) $);
      \draw[fill, white] ($ (Ina)!0.35!(Inb) $) rectangle ($ (Inb)!0.35!(Ina) $);

      \node[anchor=west, LowerRowColor1] at \shiftptno{(IbEx)}{3.65}{-.1} {$z_{d+1}$};
      \node[anchor=west, LowerRowColor1] at \shiftptno{(IbEx)}{2.75}{-.1} {$z_{d+2}$};
      \node[anchor=west, LowerRowColor1] at \shiftptno{(IbEx)}{.5}{-.1} {$z_{d+d}$};
    \end{tikzpicture}
    \dar
    \\
    \begin{tikzpicture}[scale=\myscale,baseline=(current bounding box.center),rotate=90]
      \definesandalpts
      \drawuphk{black, fill}{2}
      \drawstrp{black,fill}{2}
      \drawdwhk{black,fill}{2}
      \drawuphk{black, fill}{0}
      \drawstrp{black,fill}{0}
      \drawdwhk{black,fill}{0}
      \node[xshift=0em] at \shiftpt{(Sb)}{1} {$\cdots$};
      \path[use as bounding box] \shiftpt{(Id)}{2} rectangle \shiftpt{(Ib)}{-1};
    \end{tikzpicture}
    \&\lar
    \begin{tikzpicture}[scale=\myscale,baseline=(current bounding box.center),rotate=90]
      \definesandalpts
      \drawuphk{black,fill}{3}
      \drawstrp{black,fill}{3}
      \drawdwhk{black}{3}
      \drawuphk{black,fill}{2}
      \drawstrp{black,fill}{2}
      \drawdwhk{black}{2}
      \drawuphk{black,fill}{0}
      \drawstrp{black,fill}{0}
      \drawdwhk{black}{0}
      \node[xshift=0em] at \shiftpt{(Sb)}{1} {$\cdots$};
    \end{tikzpicture}
    \& \lar
    \begin{tikzpicture}[scale=\myscale,baseline=(current bounding box.center),rotate=90]
      \definesandalpts
      \drawuphk{black,fill}{3}
      \drawstrp{black,fill}{3}
      \drawdwhk{black}{3}
      \drawuphk{black,fill}{2}
      \drawstrp{black,fill}{2}
      \drawdwhk{black}{2}
      \drawuphk{black,fill}{0}
      \drawstrp{black,fill}{0}
      \drawdwhk{black}{0}
      \node[xshift=0em] at \shiftpt{(Sb)}{1} {$\cdots$};
    \end{tikzpicture}
  \end{tikzcd}
\]
}
giving
\[
  \begin{tikzcd}[ampersand replacement=\&]
    K^d \dar{I} \& \lar{I}
    K^d  \dar{\smat{I\\I}} \&
   0 \lar \dar \\
    K^d \dar
    \&
    \lar{\smat{I&0}}
    \dar[swap]{\smat{0&I}}
    K^{2d}
     \& K^d \dar{J_d(0)} \lar{\smat{I\\J_d(0)}}
    \\
    0  \& \lar K^d  \& K^d  \lar{I}
  \end{tikzcd}
\]

\FloatBarrier
\subsubsection{No inbound arrow}

For the case with no inbound arrow, we use a slightly different construction.
We use the same space for the central vertex, but with a different choice of basis for homology:
{
  {
  \input{drawing_macros}
  \newcommand{\myscale}{.8}

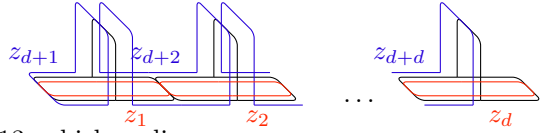
  \[
    \begin{tikzpicture}[scale=\myscale,baseline=(current bounding box.center),rotate=90]
      \definesandalpts
      \drawuphk{black}{3}
      \drawstrp{black}{3}
      \drawdwhk{black}{3}
      \drawuphk{black}{2}
      \drawstrp{black}{2}
      \drawdwhk{black}{2}
      \drawuphk{black}{0}
      \drawstrp{black}{0}
      \drawdwhk{black}{0}
      \node at \shiftpt{(Sb)}{1} {$\hdots$};
      \path[use as bounding box] \shiftpt{(Id)}{3} rectangle \shiftpt{(Ib)}{-1};

      \expanduphk{LowerRowColor1}{3}{3.1}
      \expandstrp{LowerRowColor1}{3.1}
      \expandstrp{LowerRowColor1}{2.9}
      \expanddwlegs{LowerRowColor1}{3}{2.9}
      \expanduplegs{LowerRowColor1}{2}{2.1}
      \expandstrp{LowerRowColor1}{2.1}
      \expandstrp{LowerRowColor1}{1.9}
      \expanddwlegs{LowerRowColor1}{2}{1.9}
      \expanduplegs{LowerRowColor1}{0}{0.1}
      \expandstrp{LowerRowColor1}{0.1}

      \shrunkuphk{LowerRowColor2}{3}
      \shrunkdwhk{LowerRowColor2}{3}
      \shrunkuphk{LowerRowColor2}{2}
      \shrunkdwhk{LowerRowColor2}{2}
      \shrunkuphk{LowerRowColor2}{0}
      \shrunkdwhk{LowerRowColor2}{0}

      \node[anchor=east, LowerRowColor2] at \shiftptno{(SaSh)}{2.5}{-.3} {$z_{1}$};
      \node[anchor=east, LowerRowColor2] at \shiftptno{(SaSh)}{1.5}{-.3} {$z_{2}$};
      \node[anchor=east, LowerRowColor2] at \shiftptno{(SaSh)}{-.5}{-.3} {$z_{d}$};
      \node[anchor=west, LowerRowColor1] at \shiftptno{(IbEx)}{3.5}{.4} {$z_{d+1}$};
      \node[anchor=west, LowerRowColor1] at \shiftptno{(IbEx)}{2.5}{.4} {$z_{d+2}$};
      \node[anchor=west, LowerRowColor1] at \shiftptno{(IbEx)}{.5}{.4} {$z_{d+d}$};
    \end{tikzpicture}
  \]
}
}%
In the diagram in Figure~\ref{fig:f06_4out}, which realizes
\[
  \begin{tikzcd}[ampersand replacement=\&]
    0 \& \lar
    K^d  \rar \&
    0 \\
    K^d \uar \dar
    \&
    \lar{\smat{I&0}}
    \uar[swap]{\smat{I&J_d(0)}}
    \dar {\smat{I&I}}
    K^{2d} \rar{\smat{0&I}}
    \& K^d \dar \uar
    \\
    0  \& \lar K^d  \rar \& 0
  \end{tikzcd},
\]
the leftmost space on the second row is obtained by \emph{covering} some cycles on the center so that the homology generators represented by $z_{d+1},\hdots, z_{d+d}$ are mapped to zero, and $z_1,\hdots, z_d$ are essentially unaffected. The result is a structure that looks like a tent with no floor; or in the original analogy, it is no longer a sandal but a shoe whose sole has $d$ holes: $z_1$ to $z_d$. Going to the space in the upper left, we simply fill these in.

\begin{figure}[ht]\caption{Four arrows outbound}
  \label{fig:f06_4out}
\input{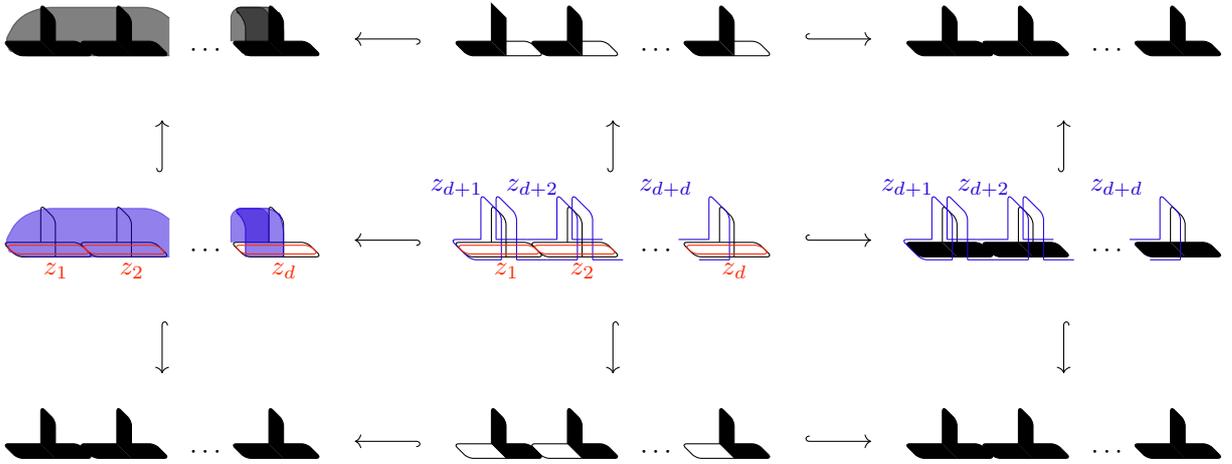}
\newcommand{\myscale}{.5}


\newcommand{\halftentcover}[3]{
  \draw[rounded corners, fill, #1]
  \shiftpt{(Ib)}{#2} -- \shiftpt{(Ib)}{#3} -- \shiftptno{(Ia)}{#3}{-0.05} -- \shiftpt{(Ic)}{#3} --
  \shiftptno{(Sd)}{#3}{-0.05} -- \shiftptno{(Id)}{#2}{-0.05} -- \shiftpt{(Ic)}{#2};
}

\newcommand{\halftentup}[3]{
  \begin{scope}[transparency group, opacity=0.5]
    \draw[rounded corners, fill, #1] \shiftpt{(Sb)}{#3} -- \shiftptno{(Sc)}{#3}{-0.05} -- \shiftptno{(Sd)}{#3-0.2}{-0.05} -- \shiftptno{(Id)}{#2}{-0.05} -- \shiftpt{(Ic)}{#2} --  \shiftpt{(Ib)}{#2};

    \draw[rounded corners, fill, #1] \shiftpt{(Sa)}{#2} -- \shiftptno{(Sd)}{#3}{-0.05} -- \shiftptno{(Id)}{#2}{-0.05} -- \shiftpt{(Ia)}{#2};
  \end{scope}

  \draw[rounded corners, semitransparent, #1] \shiftpt{(Sb)}{#3} -- \shiftptno{(Sc)}{#3}{-0.05} -- \shiftptno{(Sd)}{#3-0.2}{-0.05} -- \shiftptno{(Id)}{#2}{-0.05} -- \shiftpt{(Ic)}{#2} --  \shiftpt{(Ib)}{#2};
}

\[
  \begin{tikzcd}[ampersand replacement = \&, row sep = 2em, column sep = 2.5em,nodes={inner sep = 1.2em}, arrows={hookrightarrow}]
    \begin{tikzpicture}[scale=\myscale,baseline=(current bounding box.center),rotate=90]
      \definesandalpts
      \drawuphk{black, fill}{3}
      \drawstrp{black,fill}{3}
      \drawdwhk{black,fill}{3}
      \drawuphk{black, fill}{2}
      \drawstrp{black,fill}{2}
      \drawdwhk{black,fill}{2}
      \drawuphk{black, fill}{0}
      \drawstrp{black,fill}{0}
      \drawdwhk{black,fill}{0}
      \node[xshift=0em] at \shiftpt{(Sb)}{1} {$\cdots$};
      \path[use as bounding box] \shiftpt{(Id)}{3} rectangle \shiftpt{(Ib)}{-1};

      \halftentcover{semitransparent, black}{1}{3}

      \halftentup{black}{0}{0}
    \end{tikzpicture}
    \&\lar
    \begin{tikzpicture}[scale=\myscale,baseline=(current bounding box.center),rotate=90]
      \definesandalpts
      \drawuphk{black,fill}{3}
      \drawstrp{black,fill,rounded corners=0pt}{3}
      \drawdwhk{black}{3}
      \drawuphk{black,fill}{2}
      \drawstrp{black,fill}{2}
      \drawdwhk{black}{2}
      \drawuphk{black,fill}{0}
      \drawstrp{black,fill}{0}
      \drawdwhk{black}{0}
      \node[xshift=0em] at \shiftpt{(Sb)}{1} {$\cdots$};
      \path[use as bounding box] \shiftpt{(Id)}{3} rectangle \shiftpt{(Ib)}{-1};
    \end{tikzpicture} \rar \&
    \begin{tikzpicture}[scale=\myscale,baseline=(current bounding box.center),rotate=90]
      \definesandalpts
      \drawuphk{black, fill}{3}
      \drawstrp{black,fill}{3}
      \drawdwhk{black,fill}{3}
      \drawuphk{black, fill}{2}
      \drawstrp{black,fill}{2}
      \drawdwhk{black,fill}{2}
      \drawuphk{black, fill}{0}
      \drawstrp{black,fill}{0}
      \drawdwhk{black,fill}{0}
      \node[xshift=0em] at \shiftpt{(Sb)}{1} {$\cdots$};
      \path[use as bounding box] \shiftpt{(Id)}{3} rectangle \shiftpt{(Ib)}{-1};
    \end{tikzpicture} \\
    \begin{tikzpicture}[scale=\myscale,baseline=(current bounding box.center),rotate=90]
      \definesandalpts
      \drawuphk{black}{3}
      \drawstrp{black}{3}
      \drawdwhk{black}{3}
      \drawuphk{black}{2}
      \drawstrp{black}{2}
      \drawdwhk{black}{2}
      \drawuphk{black}{0}
      \drawstrp{black}{0}
      \drawdwhk{black}{0}
      \node[xshift=0em] at \shiftpt{(Sb)}{1} {$\cdots$};
      \path[use as bounding box] \shiftpt{(Id)}{3} rectangle \shiftpt{(Ib)}{-1};

      \halftentup{LowerRowColor1}{0}{0}
      \halftentcover{semitransparent, LowerRowColor1}{1}{3}

      \shrunkuphk{LowerRowColor2}{3}
      \shrunkdwhk{LowerRowColor2}{3}
      \shrunkuphk{LowerRowColor2}{2}
      \shrunkdwhk{LowerRowColor2}{2}
      \shrunkuphk{LowerRowColor2}{0}
      \shrunkdwhk{LowerRowColor2}{0}
      \node[anchor=east, LowerRowColor2] at \shiftptno{(SaSh)}{2.25}{-.3} {$z_{1}$};
      \node[anchor=east, LowerRowColor2] at \shiftptno{(SaSh)}{1.25}{-.3} {$z_{2}$};
      \node[anchor=east, LowerRowColor2] at \shiftptno{(SaSh)}{-.75}{-.3} {$z_{d}$};
    \end{tikzpicture}
    \dar \uar \&\lar
    \begin{tikzpicture}[scale=\myscale,baseline=(current bounding box.center), rotate=90]
      \definesandalpts
      \drawuphk{black}{3}
      \drawstrp{black}{3}
      \drawdwhk{black}{3}
      \drawuphk{black}{2}
      \drawstrp{black}{2}
      \drawdwhk{black}{2}
      \drawuphk{black}{0}
      \drawstrp{black}{0}
      \drawdwhk{black}{0}
      \node[xshift=0em] at \shiftpt{(Sb)}{1} {$\cdots$};
      \path[use as bounding box] \shiftpt{(Id)}{3} rectangle \shiftpt{(Ib)}{-1};

      \expanduphk{LowerRowColor1}{3}{3.1}
      \expandstrp{LowerRowColor1}{3.1}
      \expandstrp{LowerRowColor1}{2.9}
      \expanddwlegs{LowerRowColor1}{3}{2.9}
      \expanduplegs{LowerRowColor1}{2}{2.1}
      \expandstrp{LowerRowColor1}{2.1}
      \expandstrp{LowerRowColor1}{1.9}
      \expanddwlegs{LowerRowColor1}{2}{1.9}
      \expanduplegs{LowerRowColor1}{0}{0.1}
      \expandstrp{LowerRowColor1}{0.1}

      \shrunkuphk{LowerRowColor2}{3}
      \shrunkdwhk{LowerRowColor2}{3}
      \shrunkuphk{LowerRowColor2}{2}
      \shrunkdwhk{LowerRowColor2}{2}
      \shrunkuphk{LowerRowColor2}{0}
      \shrunkdwhk{LowerRowColor2}{0}

      \node[anchor=east, LowerRowColor2] at \shiftptno{(SaSh)}{2.25}{-.3} {$z_{1}$};
      \node[anchor=east, LowerRowColor2] at \shiftptno{(SaSh)}{1.25}{-.3} {$z_{2}$};
      \node[anchor=east, LowerRowColor2] at \shiftptno{(SaSh)}{-.75}{-.3} {$z_{d}$};

      \node[anchor=north, LowerRowColor1] at \shiftptno{(SdSh)}{3.5}{.9} {$z_{d+1}$};
      \node[anchor=north, LowerRowColor1] at \shiftptno{(SdSh)}{2.5}{.9} {$z_{d+2}$};
      \node[anchor=north, LowerRowColor1] at \shiftptno{(SdSh)}{.75}{.9} {$z_{d+d}$};
    \end{tikzpicture}
    \rar \uar \dar\&
    \begin{tikzpicture}[scale=\myscale,baseline=(current bounding box.center), rotate=90]
    \definesandalpts
    \drawuphk{black,fill}{3}
    \drawstrp{black}{3}
    \drawdwhk{black,fill}{3}
    \drawuphk{black,fill}{2}
    \drawstrp{black}{2}
    \drawdwhk{black,fill}{2}
    \drawuphk{black,fill}{0}
    \drawstrp{black}{0}
    \drawdwhk{black,fill}{0}
    \node[xshift=0em] at \shiftpt{(Sb)}{1} {$\cdots$};
    \path[use as bounding box] \shiftpt{(Id)}{3} rectangle \shiftpt{(Ib)}{-1};

    \expanduphk{LowerRowColor1}{3}{3.1}
    \expandstrp{LowerRowColor1}{3.1}
    \expandstrp{LowerRowColor1}{2.9}
    \expanddwlegs{LowerRowColor1}{3}{2.9}
    \expanduplegs{LowerRowColor1}{2}{2.1}
    \expandstrp{LowerRowColor1}{2.1}
    \expandstrp{LowerRowColor1}{1.9}
    \expanddwlegs{LowerRowColor1}{2}{1.9}
    \expanduplegs{LowerRowColor1}{0}{0.1}
    \expandstrp{LowerRowColor1}{0.1}
    
    \node[anchor=north, LowerRowColor1] at \shiftptno{(SdSh)}{3.5}{.9} {$z_{d+1}$};
    \node[anchor=north, LowerRowColor1] at \shiftptno{(SdSh)}{2.5}{.9} {$z_{d+2}$};
    \node[anchor=north, LowerRowColor1] at \shiftptno{(SdSh)}{.75}{.9} {$z_{d+d}$};    
  \end{tikzpicture}
    \uar \dar
    \\
    \begin{tikzpicture}[scale=\myscale,baseline=(current bounding box.center),rotate=90]
      \definesandalpts
      \drawuphk{black, fill}{3}
      \drawstrp{black,fill}{3}
      \drawdwhk{black,fill}{3}
      \drawuphk{black, fill}{2}
      \drawstrp{black,fill}{2}
      \drawdwhk{black,fill}{2}
      \drawuphk{black, fill}{0}
      \drawstrp{black,fill}{0}
      \drawdwhk{black,fill}{0}
      \node[xshift=0em] at \shiftpt{(Sb)}{1} {$\cdots$};
      \path[use as bounding box] \shiftpt{(Id)}{3} rectangle \shiftpt{(Ib)}{-1};
    \end{tikzpicture}
    \&\lar
    \begin{tikzpicture}[scale=\myscale,baseline=(current bounding box.center),rotate=90]
      \definesandalpts
      \drawuphk{black}{3}
      \drawstrp{black,fill}{3}
      \drawdwhk{black,fill}{3}
      \drawuphk{black}{2}
      \drawstrp{black,fill}{2}
      \drawdwhk{black,fill}{2}
      \drawuphk{black}{0}
      \drawstrp{black,fill}{0}
      \drawdwhk{black,fill}{0}
      \node[xshift=0em] at \shiftpt{(Sb)}{1} {$\cdots$};
      \path[use as bounding box] \shiftpt{(Id)}{3} rectangle \shiftpt{(Ib)}{-1};
    \end{tikzpicture}
    \rar \&
    \begin{tikzpicture}[scale=\myscale,baseline=(current bounding box.center),rotate=90]
      \definesandalpts
      \drawuphk{black, fill}{3}
      \drawstrp{black,fill}{3}
      \drawdwhk{black,fill}{3}
      \drawuphk{black, fill}{2}
      \drawstrp{black,fill}{2}
      \drawdwhk{black,fill}{2}
      \drawuphk{black, fill}{0}
      \drawstrp{black,fill}{0}
      \drawdwhk{black,fill}{0}
      \node[xshift=0em] at \shiftpt{(Sb)}{1} {$\cdots$};
      \path[use as bounding box] \shiftpt{(Id)}{3} rectangle \shiftpt{(Ib)}{-1};
    \end{tikzpicture}
  \end{tikzcd}  
\]
\end{figure}

We thus have a topological realization for all variants of the $3\times 3$ grid.

\FloatBarrier


\section{Discussion}
We have illustrated constructions of infinite families of indecomposable persistence modules together with topological realizations over the small commutative grids. By embedding, this provides constructions for all possible representation infinite commutative grids.

In addition to our families of indecomposables, other parametrized families might be of interest.
More generally, for representation tame commutative grids, could we realize parametrized families that generate all indecomposables?

\section*{Acknowledgement}
A large part of this work was performed while both authors were affiliated with WPI-AIMR, Tohoku University, Sendai, Japan. E.G.E. was partially supported by JST CREST Mathematics 15656429. This is an expanded version of the extended conference abstract \cite{buchet_et_al:LIPIcs:2018:8728} presented in SoCG 2018 and the authors are thankful to the reviewers that helped improve its quality.

\FloatBarrier
\bibliography{jitsugen_journal}


\end{document}